\documentclass[11pt,a4paper,leqno]{article}%
\usepackage{amssymb}
\usepackage{amsmath}
\usepackage{amsfonts}
\usepackage{graphicx}

\providecommand{\U}[1]{\protect\rule{.1in}{.1in}}
\newtheorem{theorem}{Theorem}[section]

\newtheorem{lemma}[theorem]{Lemma}

\newtheorem{proposition}[theorem]{Proposition}
\newtheorem{remark}[theorem]{Remark}

\newenvironment{proof}[1][Proof]{\noindent\textbf{#1.} }{\ \rule{0.5em}{0.5em}}
\numberwithin{equation}{section}

\title{Scaling limits for weakly pinned Gaussian random fields \\under the presence of two possible candidates}
\date{June 27, 2014}
\author{Erwin Bolthausen$^{1)}$, Taizo Chiyonobu$^{2)}$ and Tadahisa Funaki$^{3)}$}

\begin{document}
\maketitle

\begin{abstract}
\noindent We study the scaling limit and prove the law of large numbers for
weakly pinned Gaussian random fields under the critical situation that two
possible candidates of the limits exist at the level of large deviation
principle. This paper extends the results of \cite{BFO}, \cite{FO} for 
one dimensional fields to higher dimensions: $d\ge3 $, at least if the strength 
of pinning is sufficiently large.
\footnote{
\hskip -6mm
$^{1)}$ Institut f\"ur Mathematik, Universit\"at Z\"{u}rich, 
Winterthurerstrasse 190, 8057 Z\"urich, Switzerland.}
\footnote{
e-mail: eb@math.uzh.ch}
\footnote{
\hskip -6mm
$^{2)}$ Department of Mathematics, Kwansei-Gakuin University,
Sanda-city, Hyogo 669-1337, Japan.} 
\footnote{
e-mail:chiyo@kwansei.ac.jp} 
\footnote{
\hskip -6mm
$^{3)}$ Graduate School of Mathematical Sciences, The University of Tokyo,
Komaba, Tokyo 153-8914, Japan.} 
\footnote{
e-mail:funaki@ms.u-tokyo.ac.jp, Fax: +81-3-5465-7011} 
\footnote{
\hskip -6mm
\textit{2010Mathematics Subject Classification: Primary 60K35; Secondary 60F10, 82B41}}
\footnote{
\hskip -6mm
\textit{Abbreviated title $($running
head$)$: Scaling limits for weakly pinned Gaussian random fields.}} 
\footnote{
\hskip -6mm
\textit{The author$^{3)}$ is supported in part by the JSPS Grants 
$($A$)$ 22244007, $($B$)$ 26287014 and 26610019.}}
\footnote{
\hskip -6mm
\textit{The visit of the author$^{1)}$ to Tokyo was supported by the program of
Leading Graduate Course for Frontiers of Mathematical Sciences and Physics at
the University of Tokyo, and by an SNF grant 
No 200020$\underline{\phantom{a} }$138141.}}
\end{abstract}

\section{Introduction and main result}

This paper is concerned with weakly pinned Gaussian random fields which
are microscopically defined on a $d$-dimensional region $D_N$ of 
large size $N$.  We study its macroscopic limit by scaling down
its size to $O(1)$  as $N\to\infty$ under the critical situation 
that two possible candidates of the limits exist at the level of 
rough large deviations.  We work out which one really appears 
in the limit assuming that $d\ge 3$ and the strength $\varepsilon>0$
of the pinning is sufficiently large.

\subsection{Weakly pinned Gaussian random fields} 

We work on the $d$-dimensional square lattice
$D_{N}=\{0,1,2,\ldots,N\}\times{\mathbb{T}}_{N}^{d-1}$ and denote its elements
by $i=(i_{1},i_{2},\ldots,i_{d})\equiv(i_{1},\underline{i})\in D_{N}$, where
${\mathbb{T}}_{N}^{d-1}=({\mathbb{Z}}/N{\mathbb{Z}})^{d-1}$ is the
$(d-1)$-dimensional lattice torus. In other words, we consider the lattice
under periodic boundary conditions for the coordinates except the first
one. The left and right boundaries of $D_{N}$ are denoted by $\partial
_{L}D_{N}=\{0\}\times{\mathbb{T}}_{N}^{d-1}$ and $\partial_{R}D_{N}%
=\{N\}\times{\mathbb{T}}_{N}^{d-1}$, respectively. We set $\partial
D_{N}=\partial_{L}D_{N}\cup\partial_{R}D_{N}$ and $D_N^\circ=D_{N}%
\setminus\partial D_{N}$.

The Hamiltonian is associated with an ${\mathbb{R}}$-valued field $\phi
=(\phi_{i})_{i\in D_{N}}\in{\mathbb{R}}^{D_{N}}$ over $D_{N}$ by
\begin{equation}
H_{N}(\phi)=\frac{1}{2}\sum_{\langle i,j\rangle\subset D_{N}}(\phi_{i}%
-\phi_{j})^{2}, \label{eq:1.1}%
\end{equation}
where the sum is taken over all undirected bonds $\langle i,j\rangle$ in
$D_{N}$, i.e., all pairs $\{i,j\}$ such that $i,j\in D_{N}$ and $|i-j|=1$. 
We sometimes denote $\phi_i$ by $\phi(i)$.  For
given $a,b >0$, we impose the Dirichlet boundary condition for
$\phi$ at $\partial D_{N}$ by
\begin{equation}
\phi_{i}=aN\quad\text{for }\;i\in\partial_{L}D_{N},\quad\phi_{i}%
=bN\quad\text{for }\;i\in\partial_{R}D_{N}. \label{eq:1.2}%
\end{equation}
For $\varepsilon\geq0$, the strength of the pinning force
toward $0$ acting on the field $\phi$, we introduce the Gibbs probability
measure on ${\mathbb{R}}^{D_{N}^{\circ}}$:
\begin{equation}
\mu_{N}^{aN,bN,\varepsilon}(d\phi)=\frac{1}{Z_{N}^{aN,bN,\varepsilon}%
}\mathrm{e}^{-H_{N}^{aN,bN}(\phi)}\prod_{i\in D_{N}^{\circ}}\left[
\varepsilon\delta_{0}(d\phi_{i})+d\phi_i\right], \label{eq:1.3}%
\end{equation}
where $Z_{N}^{aN,bN,\varepsilon}$ is the normalizing constant (partition
function) and $H_{N}^{aN,bN}(\phi)$ is the Hamiltonian $H_{N}(\phi)$
with the boundary condition \eqref{eq:1.2}. We sometimes regard
$\mu_{N}^{aN,bN,\varepsilon}$ as a probability measure on ${\mathbb{R}}%
^{D_{N}}$ by extending it over $\partial D_{N}$ due to the condition \eqref{eq:1.2}.

\subsection{Scaling and large deviation rate functional} 
Let $D= [0,1]\times{\mathbb{T}}^{d-1}$ be
the macroscopic region corresponding to $D_{N}$, where ${\mathbb{T}}^{d-1}
=({\mathbb{R}}/{\mathbb{Z}})^{d-1}$ is the $(d-1)$-dimensional unit torus. 
We associate a macroscopic height field  
$h^N: D\to {\mathbb R}$ with the microscopic one 
$\phi\in{\mathbb{R}}^{D_{N}}$ as a step function defined by
\begin{equation} \label{eq:Ma-H-1}
h^N(t) = \frac1N \phi(i), \quad t =(t_1,\underline{t}) 
\in B\big(\frac{i}N,\frac1N\big)\cap D, \; i \in D_N,
\end{equation}
where $B\big(\frac{i}N,\frac1N\big)$ denotes the box $\big[\frac{i-\frac12}N, 
\frac{i+\frac12}N\big)^d$ with the center $\frac{i}N$ and sidelength $\frac1N$
considered periodically in the direction of $\underline{t}$.
It is sometimes convenient to introduce another macroscopic filed $h^N$,
denoted by $h^N_{\text{PL}}$, as a polilinear interpolation of 
$\frac1N \phi(i)$:
\begin{equation} \label{eq:Ma-H-2}
h^N_{\rm{PL}}(t) = \frac1N \sum_{v\in \{0,1\}^d}
\left[ \prod_{\alpha=1}^d \big( v_\alpha\{Nt_\alpha\} + 
 (1-v_\alpha)(1-\{Nt_\alpha\}) \big) \right] \phi([Nt]+v),
\end{equation}
where $[\cdot]$ and $\{\cdot\}$ stand for the integer and
the fractional parts, respectively, see (1.17) in \cite{DGI}.
Note that $h^N_{\rm{PL}}\in C(D,\mathbb{R})$.
We will prove that $h^N$ and $h_{\text{PL}}^N$ are close enough
in a superexponential sense; see Lemma \ref{lem:Ma-H-1} below.
Our goal is to study the asymptotic behavior of
$h^{N}$ distributed under $\mu_{N}^{aN,bN,\varepsilon}$ as $N\to\infty$.

We will prove that a large deviation principle (LDP) holds for $h^{N}$
under $\mu_{N}^{aN,bN,\varepsilon}$, roughly stating
\[
\mu_{N}^{aN,bN,\varepsilon}(h^{N}\sim h)\sim\mathrm{e}^{-N^{d}\Sigma^{\ast
}(h)},
\]
as $N\rightarrow\infty$ with an unnormalized rate functional
\begin{equation}
\Sigma(h)=\frac{1}{2}\int_{D}|\nabla h(t)|^{2}dt-\xi^{\varepsilon}\left\vert
\{t\in D;h(t)=0\}\right\vert , \label{eq:1.Sigma}%
\end{equation}
for $h:D\rightarrow{\mathbb{R}}$; see \eqref{eq:1.11}.
The functional $\Sigma^{\ast}$ is the
normalization of $\Sigma$ such that $\min\Sigma^{\ast}=0$ by adding a suitable
constant, i.e., $\Sigma^{\ast}(h)=\Sigma(h)-\min\Sigma$. The non-negative
constant $\xi^{\varepsilon}$ is the free energy determined by
\begin{equation}
\xi^{\varepsilon}=\lim_{\ell\rightarrow\infty}\frac{1}{|\Lambda_{\ell}|}%
\log\frac{Z_{\Lambda_{\ell}}^{0,\varepsilon}}{Z_{\Lambda_{\ell}}^{0}},
\label{eq:1.5}%
\end{equation}
where $\Lambda_{\ell}=\{1,2,\ldots,\ell\}^{d}\Subset{\mathbb{Z}}^{d}$,
$|\Lambda_{\ell}|=\ell^{d}$, and $Z_{\Lambda_{\ell}}^{0,\varepsilon}$ and
$Z_{\Lambda_{\ell}}^{0}$ are the partition functions on $\Lambda_{\ell}$ with
$0$-boundary conditions with and without pinning, respectively. 
It is known that $\xi^\varepsilon$ exists, and that the
field is localized by the pinning effect (even if $d=1,2$), meaning that
$\xi^{\varepsilon}>0$ for all $\varepsilon>0$ (and all $d\geq1$); see, e.g.,
Section 7 of \cite{Fu05} or Remark 6.1 of \cite{FSa}.

\subsection{Minimizers of the rate functional} 

The functional $\Sigma$ is defined for
functions $h$ on $D$, which satisfy the (macroscopic) boundary conditions:
\begin{equation}
h(0,\underline{t})=a,\quad h(1,\underline{t})=b. \label{eq:1.6}%
\end{equation}
We denote $t=(t_{1},\underline{t})\in D=[0,1]\times{\mathbb{T}}^{d-1}$. Since
the boundary conditions \eqref{eq:1.6} and the functional $\Sigma$ are
translation-invariant in the variable $\underline{t}$, the minimizers of
$\Sigma$ are functions of $t_{1}$ only and the minimizing problem can be
reduced to the 1D case; see Lemma \ref{lem:1.1} below. Thus the candidates of
the minimizers of $\Sigma$ are of the forms:
\[
\hat{h}(t)=\hat{h}^{(1)}(t_{1}),\quad\bar{h}(t)=\bar{h}^{(1)}(t_{1}),
\]
where $\hat{h}^{(1)}$ and $\bar{h}^{(1)}$ are the candidates of the minimizers
in the one-dimensional problem under the condition $h(0)=a, h(1)=b$, that is,
$\bar{h}^{(1)}(t_1) = (1-t_1)a + t_1b, t_1\in [0,1]$, and, when $a+b<\sqrt{2\xi^\varepsilon}$,
\begin{equation*}
\hat{h}^{(1)}(t_1) = \left\{
\begin{aligned}
(s_1^L-t_1)a/s_1^L, \qquad & t_1\in [0,s_1^L],\\
0, \qquad & t_1\in [s_1^L,s_1^R],\\
(t_1-s_1^R)b/(1-s_1^R), \qquad & t_1\in [s_1^R,1],
\end{aligned}
\right.
\end{equation*}
where $0<s_1^L<s_1^R<1$ are determined by 
$a/s_1^L= b/(1-s_1^R) = \sqrt{2\xi^\varepsilon}$; see Section 3.1 below, 
Section 1.3 and Appendix B of \cite{BFO} or Section 6.4 of \cite{Fu05}.

\begin{lemma}
\label{lem:1.1} The set of the minimizers of the functional $\Sigma$ is
contained in $\{\hat{h},\bar{h}\}$.
\end{lemma}

\begin{proof}
Consider the functional
\[
\Sigma^{(1)}(g)=\frac{1}{2}\int_{0}^{1}\dot{g}(t_{1})^{2}dt_{1}-\xi
^{\varepsilon}\left\vert \{t_{1}\in\lbrack0,1];g(t_{1})=0\}\right\vert
\]
for functions $g=g(t_{1})$ with a single variable $t_{1}\in\lbrack0,1]$. Then,
for $h=h(t)\equiv h(t_{1},\underline{t})$, one can rewrite $\Sigma(h)$ as
\begin{equation}
\Sigma(h)=\int_{\mathbb{T}^{d-1}}\Sigma^{(1)}(h(\cdot,\underline
{t}))\,d\underline{t}+\frac{1}{2}\int_{D}|\nabla_{\underline{t}}%
h(t_{1},\underline{t})|^{2}\,dt, \label{eq:1.7}%
\end{equation}
where
\[
\nabla_{\underline{t}}h=\left(  \frac{\partial h}{\partial t_{2}},\ldots
,\frac{\partial h}{\partial t_{d}}\right)  ,\quad\underline{t}=(t_{2}%
,\ldots,t_{d}).
\]
However, since the minimizers of $\Sigma^{(1)}$ are $\hat{h}^{(1)}$ or
$\bar{h}^{(1)}$ (see \cite{BFO}, \cite{Fu05}), we see that
\[
\Sigma^{(1)}(h(\cdot,\underline{t}))\geq\Sigma^{(1)}(\hat{h}^{(1)}%
)\wedge\Sigma^{(1)}(\bar{h}^{(1)}),
\]
and this inequality integrated in $\underline{t}$ combined with \eqref{eq:1.7}
implies
\begin{equation}
\Sigma(h)\geq\Sigma(\hat{h})\wedge\Sigma(\bar{h}) \label{eq:1.8}%
\end{equation}
for all $h=h(t)$. Moreover, from \eqref{eq:1.7} again, the identity holds in
\eqref{eq:1.8} if and only if
\[
\int_{D}|\nabla_{\underline{t}}h(t_{1},t)|^{2}\,dt=0,
\]
which implies that $h$ is a function of $t_{1}$ only.
\end{proof}

\subsection{Main result} 

We are concerned with the critical situation where 
$\Sigma(\hat{h}) = \Sigma(\bar{h})$ holds with $\hat{h}\not =
\bar{h}$, which is equivalent to $\sqrt{a}+\sqrt{b} = (2\xi^{\varepsilon
})^{1/4}$, see Proposition B.1 of \cite{BFO}. Note that this condition implies
$0<s_1^L<s_1^R<1$ for $\hat{h}^{(1)}$.  Otherwise, from
\eqref{eq:1.11} below, $h^{N}$ converges to the unique minimizer of $\Sigma$
($\hat{h}$ in case $\Sigma(\hat{h}) < \Sigma(\bar{h})$ and $\bar{h}$
in case $\Sigma(\bar{h}) < \Sigma(\hat{h})$) as $N\to\infty$ in probability. 
Our main result is 

\begin{theorem} \label{thm:main} 
We assume $\Sigma(\hat{h}) = \Sigma(\bar{h})$.  Then,
if $d\geq3$ and if $\varepsilon>0$ is sufficiently large, we have that
\[
\lim_{N\rightarrow\infty}\mu_{N}^{aN,bN,\varepsilon}\left(  \Vert h^{N}%
-\hat{h}\Vert_{L^{1}(D)}\leq\delta\right)  =1,
\]
for every $\delta>0$.
\end{theorem}

\begin{remark}
One can even take $\delta=N^{-\alpha}$ with some $\alpha>0$.
\end{remark}

We conjecture that neither the conditions on the dimension $d$, nor the
one on $\varepsilon$ being large, are necessary for the result. For $d=1 $,
the convergence to $\hat{h}$ was proved in [3], [7]. The largeness of
$\varepsilon$ is used here in an essential way to prove the lower bound
(1.11). The other parts of the proof don't use it.  The condition $d\geq3$ is used at a
number of places where it is convenient that the random walk on $\mathbb{Z}%
^{d}$ is transient. We believe, however, that a proof for $d=2$ would only be
technically more involved.

\subsection{Outline of the proof} 

The proof of Theorem
\ref{thm:main} will be completed in the following three steps. In the first
step, we show the following lower bound: For every $\alpha<1$ and $1\leq p\leq2$,
\begin{equation}
\frac{Z_{N}^{aN,bN,\varepsilon}}{Z_{N}^{aN,bN}}\mu_{N}^{aN,bN,\varepsilon
}(\Vert h^{N}-\hat{h}\Vert_{L^{p}(D)}\leq N^{-\alpha})\geq e^{cN^{d-1}}
\label{eq:1.9}%
\end{equation}
with $c=c_{\varepsilon}>0$ for $N\geq N_{0}$ if $\varepsilon>0$ is
sufficiently large, where $Z_{N}^{aN,bN}=Z_{N}^{aN,bN,0}$ (i.e.,
$\varepsilon=0$). The second step establishes an upper bound for the
probability of the event that the surface stays near $\bar h$:
\begin{equation}
\frac{Z_{N}^{aN,bN,\varepsilon}}{Z_{N}^{aN,bN}}\mu_{N}^{aN,bN,\varepsilon
}(\Vert h^{N}-\bar{h}\Vert_{L^{p}(D)}\leq(\log N)^{-\alpha_{0}})\leq2
\label{eq:1.10}%
\end{equation}
with some $\alpha_{0}>0$ and $N\geq N_{0}$. 
In the last step, we prove a large deviation type estimate:
\begin{equation}
\lim_{N\rightarrow\infty}\mu_{N}^{aN,bN,\varepsilon}\left(
\operatorname{dist}_{L^{1}}(h^{N},\{\hat{h},\bar{h}\})\geq N^{-\alpha_{1}%
}\right)  =0 \label{eq:1.11}%
\end{equation}
for some $\alpha_{1}>0$. These three estimates \eqref{eq:1.9}--\eqref{eq:1.11}
conclude the proof of Theorem \ref{thm:main}. In fact, choosing $\alpha$ such
that $0<\alpha<(\alpha_{1}\wedge1)$, \eqref{eq:1.9} together with
\eqref{eq:1.10} implies
\[
\lim_{N\rightarrow\infty}\frac{\mu_{N}^{aN,bN,\varepsilon}(\Vert h^{N}-\hat
{h}\Vert_{L^{1}(D)}\leq N^{-\alpha})}{\mu_{N}^{aN,bN,\varepsilon}(\Vert
h^{N}-\bar{h}\Vert_{L^{1}(D)}\leq N^{-\alpha})}=\infty,
\]
since $N^{-\alpha}\leq(\log N)^{-\alpha_{0}}$ for $N$ large, and at the same
time the sum of the numerator and the denominator converges to $1$ from
\eqref{eq:1.11} since $\alpha<\alpha_{1}$.  

A difficulty is stemming from 
the fact that for $d\geq2$ a statement like \eqref{eq:1.11}
cannot be correct with the $L^{1}$-distance replaced by the $L^{\infty
}$-distance. If \eqref{eq:1.11} would be correct in sup-norm, then $h^{N}$ 
would stay, with large probability, either $L^{\infty}$-close to $\bar{h}$ or $\hat{h}$.
However, if it would stay close to $\bar{h}$ in sup-norm, the field $\phi$
would nowhere be $0$, and therefore \eqref{eq:1.10} would be trivial, with the bound
$1$.

\begin{remark}
An estimate weaker than \eqref{eq:1.9}:
\begin{equation}
\frac{Z_{N}^{aN,bN,\varepsilon}}{Z_{N}^{aN,bN}}\geq\mathrm{e}^{cN^{d-1}}
\label{eq:1.12}%
\end{equation}
is enough to conclude the proof of Theorem \ref{thm:main}. In fact, this
combined with \eqref{eq:1.10} implies that $\mu_{N}^{aN,bN,\varepsilon}(\Vert
h^{N}-\bar{h}\Vert_{L^{p}(D)}\leq(\log N)^{-\alpha_{0}})$ tends to $0$ as
$N\rightarrow\infty$.
\end{remark}

The three estimates \eqref{eq:1.9}, \eqref{eq:1.10} and \eqref{eq:1.11} will
be proved in Sections 4, 5 and 6, respectively. Section 2 gathers some
necessary estimates on the partition functions and Green's functions. 
Section 3 contains an analytic stability result which is important in Section 6. 
The capacity plays a role in Section 5.  The arguments in Section 6
are similar to those in \cite{BI}, but there is an additional complication
here due to the non-zero boundary conditions.  To overcome this,
we introduce fields on an extended set with zero boundary conditions.

\section{Estimates on partition functions and Green's functions}

\subsection{Reduction to $0$-boundary conditions, the case without pinning}

Let $E_{n}=\left\{  1,2,\ldots,n\right\}  \times\mathbb{T}_{N}^{d-1}\subset
D_N^\circ$ for $1\leq n\leq N-1$. For $A\subset D_N^\circ$, we denote
$\partial A=\left\{  i\in D_{N}\backslash A:\left\vert i-j\right\vert
=1\ \mathrm{for\ some\ }j\in A\right\}  $ and $\bar{A}=A\cup\partial A$. For
$A$ such that $E_{n}\subset A$ with some $n\geq1$ and for $\alpha,\beta
\in\mathbb{R}$, the partition function $Z_{A}^{\alpha,\beta}$ without pinning
is defined by%
\begin{equation} \label{eq:2.1-b}
Z_{A}^{\alpha,\beta}=\int_{\mathbb{R}^{A}}\mathrm{e}^{-H_{A}^{\alpha,\beta
}\left(  \phi\right)  }\prod_{i\in A}d\phi_{i},
\end{equation}
where $H_{A}^{\alpha,\beta}\left(  \phi\right)  $ is the Hamiltonian
(\ref{eq:1.1}) with the sum taken over all $\left\langle i,j\right\rangle
\subset\bar{A}$ under the boundary condition%
\begin{equation} \label{eq:2.2-b}
\phi_{i}=\alpha\ \mathrm{for\ }i\in\partial_{L}A,\ \phi_{i}=\beta
\ \mathrm{for\ }i\in\partial_{R}A,
\end{equation}
where $\partial_{L}A=\partial_{L}D_{N}$ and $\partial_{R}A=\partial
A\backslash\partial_{L}A$($=\partial A\cap\left\{  i:i_{1}\geq2\right\}  $).
For general $A\subset D_N^\circ$, we denote $Z_{A}^{0}$ the partition function
without pinning defined by (\ref{eq:2.1-b}) under the boundary condition
$\phi_{i}=0,$ $i\in\partial A$.

\begin{lemma}  \label{lem:1}
{\rm (1)} We have
\[
Z_{E_{n-1}}^{\alpha,\beta}=\mathrm{e}^{-\frac{N^{d-1}}{2n}(\alpha-\beta)^{2}%
}Z_{E_{n-1}}^{0,0}.%
\]
In particular,
\begin{equation} \label{eq:2.3-b}
Z_{N}^{aN,bN}=\mathrm{e}^{-\frac{N^{d}}{2}(a-b)^{2}}Z_{N}^{0,0}.
\end{equation}
{\rm (2)} If $A\supset E_{n-1}$ for some $n\geq2$, we have
\begin{equation}
Z_{A}^{\alpha,\beta,0}\geq\mathrm{e}^{-\frac{N^{d-1}}{2n}\left(  \alpha
-\beta\right)  ^{2}}Z_{A}^{0,0}. \label{eq:3.4}%
\end{equation}
\end{lemma}

\begin{proof}
We first recall the summation by parts formula for the Hamiltonian
$H_{A}^{\psi}(\phi)$ for $A\subset D_{N}^{\circ}$ with the general boundary
condition $\psi=(\psi_{i})_{i\in\partial A}$:
\[
H_{A}^{\psi}(\phi)=-\frac{1}{2}\left(  (\phi-\bar{\phi}^{A,\psi}),\Delta
_{A}(\phi-\bar{\phi}^{A,\psi})\right)  _{A}+\left(  \operatorname*{BT}\right)
,
\]
where $(\phi^{1},\phi^{2})_{A}=\sum_{i\in A}\phi_{i}^{1}\phi_{i}^{2}$ stands
for the inner product of $\phi^{1},\phi^{2}\in\mathbb{R}^{A}$, $\Delta
_{A}\equiv\Delta$ is the discrete Laplacian on $A$, $\bar{\phi}=\bar{\phi
}^{A,\psi}$ is the solution of the Laplace equation:
\begin{equation}
\left\{
\begin{array}
[c]{cc}%
\left(  \Delta\bar{\phi}\right)  _{i}=0 & i\in A\\
\bar{\phi}_{i}=\psi_{i} & i\in\partial A
\end{array}
\right.  \label{eq:3.5}%
\end{equation}
and the boundary term (BT) is given by
\[
\left(  \operatorname*{BT}\right)  =\frac{1}{2}\sum_{i\in A,j\in\partial
A:|i-j|=1}\psi_{j}\{\psi_{j}-\bar{\phi}_{i}^{A,\psi}\},
\]
see the proof of Proposition 3.1 of \cite{Fu05} (which is stated only for
$A\Subset\mathbb{Z}^{d}$, but the same holds for $A\subset D_N^\circ$).

When $A=E_{n-1}$ and the boundary condition $\psi$ is given as in
\eqref{eq:2.2-b}, the Laplace equation \eqref{eq:3.5} has an explicit solution
$\bar{\phi}=\bar{\phi}^{E_{n-1},\psi}$:
\begin{equation}
\bar{\phi}_{i}=\frac{1}{n}\left(  \beta i_{1}+\alpha(n-i_{1})\right)  ,\quad
i\in\bar{E}_{n-1}. \label{eq:3.6}%
\end{equation}
Thus, in this case, the boundary term is given by
\[
\left(  \operatorname*{BT}\right)  =\frac{N^{d-1}}{2n}\left(  \alpha
-\beta\right)  ^{2},
\]
which shows the first assertion in (1). In particular, \eqref{eq:2.3-b} follows
by noting that $Z_{N}^{aN,bN}=Z_{E_{N-1}}^{aN,bN}$.

To prove (2), we may assume $\alpha>0$ by symmetry. Let $\bar{\phi}^{A}$ be
the solution of the Laplace equation \eqref{eq:3.5} on $A$ with $\psi$ given
by \eqref{eq:2.2-b} and set $\bar{\phi}^{n-1}:=\bar{\phi}^{E_{n-1}}$. Then, we
have
\begin{equation}
\bar{\phi}_{i}^{A}\geq\bar{\phi}_{i}^{n-1}\quad\text{ for all }\;i\in\bar
{E}_{n-1}. \label{eq:3.7}%
\end{equation}
Indeed, since $\alpha>0$, the maximum principle implies that $\bar{\phi}%
^{A}\geq0$ on $\partial_{R}E_{n-1}$ and, in particular, two harmonic functions
$\bar{\phi}^{A}$ and $\bar{\phi}^{n-1}$ on $E_{n-1}$ satisfy $\bar{\phi}%
^{A}\geq\bar{\phi}^{n-1}$ on $\partial E_{n-1}$. Therefore, by the comparison
principle, we obtain \eqref{eq:3.7}.

Consider now the boundary term (BT) of $H_{A}^{\alpha,0}(\phi)$. Then, the
contribution from the pair $\left\langle i,j\right\rangle $ such that
$j\in\partial_{R}A$ vanishes, since $\psi_{j}=0$ for such $j$. On the other
hand, for $i\in A,j\in\partial_{L}A$ such that $|i-j|=1$, we see from
\eqref{eq:3.7} and then by \eqref{eq:3.6},
\[
\psi_{j}\{\psi_{j}-\bar{\phi}_{i}^{A,\psi}\}\leq\alpha\{\alpha-\bar{\phi}%
_{i}^{n-1}\}=\frac{1}{n}\alpha^{2}.
\]
This completes the proof of (2).
\end{proof}

\begin{remark}
If $A\subset E_{n-1}$, one can similarly show an upper bound on
$Z_{A}^{\alpha,0}$ (i.e.\ an inequality opposite to \eqref{eq:3.4}), but this
will not be used.
\end{remark}

\subsection{Estimates on the partition functions with $0$-boundary conditions
without pinning}

In the subsequent part of Section 2, we will only consider the partition
functions under the $0$-boundary conditions. The superscripts \lq\lq$RW^{d,N}%
$" and \lq\lq$RW^{d}$" refer to simple random walks $\{\eta_{n}%
\}_{n=0,1,2,\ldots}$ on ${\mathbb{Z}}\times{\mathbb{T}}_{N}^{d-1}$ and
${\mathbb{Z}}^{d}$, respectively, and $k$ in $P_{k}^{RW^{d}}$ or
$P_{k}^{RW^{d,N}}$ refers to the starting point of the random walk. We
introduce three quantities:
\begin{align*}
&  q = \sum_{n=1}^{\infty}\frac1{2n} P_{0}^{RW^{d}}(\eta_{2n}=0),\\
&  q^{N} = \sum_{n=1}^{\infty}\frac1{2n} P_{0}^{RW^{d,N}}(\eta_{2n}=0),\\
&  r = \sum_{n=1}^{\infty}\frac1{2n} E_{0}^{RW^{d}} \left[  \max_{1\le m
\le2n} |\eta_{m}| \cdot1_{\{\eta_{2n}=0\}} \right]  .
\end{align*}
Note that $q<\infty$ for all $d\ge1$ and $r<\infty$ for $d\ge2$ (the case that
$d\ge3$ is easy, while the case that $d=2$ is discussed in \cite{BI}, p.543).
Indeed, if $d\ge3$, $r<\bar c=G(0,0)$, the Green's function defined below
in Sections \ref{section:2.3} and \ref{section:2.4}.

The next lemma, in particular its assertion (1), is shown similarly to
the proof of Proposition 4.2.2 or Lemma 2.3.1-a) in \cite{BI},
only keeping in mind the fact that our random walk \lq\lq$RW^{d,N}$" is
periodic in the second to the $d$th components.

\begin{lemma}
\label{lem:2}
{\rm (1)} Assume that $d\geq2$ and $N$ is even, and let $A\subset
D_{N}^{\circ}$. Then, we have that
\[
\frac{1}{2}\left(  \log\frac{\pi}{d}+q^{N}\right)  |A|-r\max_{n=1,2,\ldots
}|\partial A_{n}|\leq\log Z_{A}^{0}\leq\frac{1}{2}\left(  \log\frac{\pi}%
{d}+q^{N}\right)  |A|,
\]
where $|A|=\sharp\{i\in A\}$ is the number of points in $A$ and $A_{n}=\{i\in
A;\min\limits_{j\in D_{N}\setminus A}|i-j|\geq n\}$.\\
{\rm (2)} We have the estimate
\[
0\leq q^{N}-q\leq CN^{-d},
\]
with some $C>0$ for every $d\geq2$.
\end{lemma}

\begin{proof}
We recall the random walk representation for the partition function $Z_{A}%
^{0}$ from \cite{BI}, (4.1.1) and (4.1.3) noting that $\Delta_A=2d(P_A-I)$
in our setting:
\begin{equation}
\log Z_{A}^{0}=\frac{1}{2}\left(  |A|\log\frac{\pi}{d}+I\right)  ,
\label{eq:3.8}%
\end{equation}
where
\begin{equation}
I=\sum_{k\in A}\sum_{n=1}^{\infty}\frac{1}{2n}P_{k}^{RW^{d,N}}(\eta
_{2n}=k,{\tau_{A}}>2n) \label{eq:3.9}%
\end{equation}
and ${\tau_{A}}$ is the first exit time of $\eta$ from $A$; note that, since $N$
is even, $P_{k}^{RW^{d,N}}(\eta_{2n-1}=k)=0$. The upper bound for $\log
Z_{A}^{0}$ in (1) is immediate by dropping the event $\{{\tau_{A}}>2n\}$ from the
probability. To show the lower bound, we follow the calculations subsequent to
(4.2.8) in the proof of Proposition 4.2.2 of \cite{BI}:
\[
I=q^{N}|A|-\sum_{t=1}^{N-1}\sum_{k\in\partial A_{t}}\sum_{n=1}^{\infty}%
\frac{1}{2n}P_{k}^{RW^{d,N}}(\eta_{2n}=k,{\tau_{A}}\leq2n),
\]
note that $\partial A_{t}=\emptyset$ for $t\geq N$. Let $\tilde{A}%
\subset\mathbb{Z}^{d}$ be the periodic extension of $A$ in the second to the
$d$th coordinates. Then, since ${\tau_{A}}$ under $RW^{d,N}$ is the same as
${\tau_{{\tilde{A}}}}$ under $RW^{d}$ and ${\tau_{{\tilde{A}}}}\geq{\tau_{{k+S_{t}}}}$ for
$k\in\partial A_{t}$, we have
\[
P_{k}^{RW^{d,N}}(\eta_{2n}=k,{\tau_{A}}\leq2n)\leq P_{0}^{RW^{d}}(\eta
_{2n}=0,{\tau_{{S_{t}}}}\leq2n),
\]
where $S_{t}=[-t,t]^{d}\cap\mathbb{Z}^{d}$ is a box in $\mathbb{Z}^{d}$. The
rest is the same as in \cite{BI}.

We finally show the assertion (2). In the representation
\[
q^{N}-q=\sum_{n=1}^{\infty}\frac{1}{2n}P_{0}^{RW^{d}}\big(\eta_{2n}%
\in\{0\}\times(N\mathbb{Z}^{d-1}\setminus\{0\})\big),
\]
by applying the Aronson's type estimate for the random walk on $\mathbb{Z}^{d}
$:
\[
P_{0}^{RW^{d}}(\eta_{2n}=k)\leq\frac{C_{1}}{n^{d/2}}\mathrm{e}^{-|k|^{2}%
/C_{1}n},\quad k\in\mathbb{Z}^{d},
\]
with some $C_{1}>0$, we obtain that
\[
0\leq q^{N}-q\leq\frac{C_{1}}{2}\sum_{n=1}^{\infty}\frac{1}{n^{(d+2)/2}}%
\sum_{\underline{\ell}\in\mathbb{Z}^{d-1}\setminus\{0\}}\mathrm{e}%
^{-N^{2}|\underline{\ell}|^{2}/C_{1}n}.
\]
However, the last sum in $\underline{\ell}$ can be bounded by
\[
C_{2}\left(  1+\frac{\sqrt{n}}{N}\right)  \mathrm{e}^{-N^{2}/C_{2}n}%
\]
with some $C_{2}>0$. Indeed, the sum over $\{\underline{\ell}:1\leq
|\underline{\ell}|\leq10\}$ is bounded by $\sharp\{\underline{\ell}%
:1\leq|\underline{\ell}|\leq10\}\times\mathrm{e}^{-N^{2}/C_{1}n}$, while the
sum over $\{\underline{\ell}:|\underline{\ell}|\geq11\}$ can be bounded by the
integral:
\[
C_{3}\int_{\{x\in\mathbb{R}^{d-1}:|x|\geq10\}}\mathrm{e}^{-N^{2}|x|^{2}%
/C_{1}n}\,dx
\]
with some $C_{3}>0$ and this proves the above statement. Thus, we have
\[
0\leq q^{N}-q\leq\frac{C_{1}C_{2}}{2}\sum_{n=1}^{\infty}\frac{1}{n^{(d+2)/2}%
}\left(  1+\frac{\sqrt{n}}{N}\right)  \mathrm{e}^{-N^{2}/C_{2}n}.
\]
Again, estimating the sum in the right hand side by the integral:
\[
C_{4}\int_{1}^{\infty}\frac{1}{t^{(d+2)/2}}\left(  1+\frac{\sqrt{t}}%
{N}\right)  \mathrm{e}^{-N^{2}/C_{2}t}\,dt,
\]
with some $C_{4}>0$ and then changing the variables: $t=N^{2}/u$ in the
integral, the conclusion of (2) follows immediately.
\end{proof}

\subsection{Estimates on the Green's functions}  \label{section:2.3}

Let $G_{N}(i,j),i,j\in D_{N}$ be the Green's function on $D_{N}$ with Dirichlet
boundary condition at $\partial D_{N}$:
\begin{equation}
G_{N}(i,j)=\sum_{n=0}^{\infty}P_{i}(\eta_{n}=j,n<\sigma)\left(  =E_{i}^{RW^{d,N}}\left[
\sum_{n=0}^{\infty}1_{\{\eta_{n}=j,n<\sigma\}}\right]  \right)  ,
\label{eq:G-1}%
\end{equation}
where $\eta_{n}$ is the random walk on $D_{N}$ (or on $\mathbb{Z}\times
\mathbb{T}_N^{d-1}$) and 
$$
\sigma=\inf\{n\geq 0;\eta_{n}\in\partial D_{N}\}. 
$$
Let $\tilde{G}_{N}(i,j),i,j\in\tilde{D}%
_{N}:=\{0,1,2,\ldots,N\}\times{\mathbb{Z}}^{d-1}$ be the Green's function on
$\tilde{D}_{N}$ with Dirichlet boundary condition at $\partial\tilde{D}%
_{N}=\{0,N\}\times{\mathbb{Z}}^{d-1}$, which has a similar expression to
\eqref{eq:G-1} with the random walk $\tilde{\eta}_{n}$ on $\tilde{D}_{N}$ and
its hitting time $\tilde{\sigma}$ to $\partial\tilde{D}_{N}$. For $i$ or
$j\notin D_N^\circ:=D_{N}\backslash\partial D_{N}$, we put $G_{N}\left(
i,j\right)  :=0$, and similarly for $\tilde{G}_{N}.$ We also denote the Green's
function of the random walk on the whole lattice ${\mathbb{Z}}^{d}$ by
$G(i,j),i,j\in{\mathbb{Z}}^{d}$, which exists because we assume $d\geq3$.

Then, we easily see that
\begin{equation}
G_{N}(i,j)=\sum_{k\in{\mathbb{Z}}^{d-1}}\tilde{G}_{N}(i,j+kN),\quad i,j\in
D_{N},\label{eq:G-2}%
\end{equation}
where $D_{N}$ is naturally embedded in $\tilde{D}_{N}$ and $kN$ is identified
with $(0,kN)\in{\mathbb{Z}}^{d}$. In fact, the sum in the right hand side of
\eqref{eq:G-2} does not depend on the choice of $j\in\tilde{D}_{N}$, in the
equivalent class to the original $j\in D_{N}$ in modulo $N$ in the second to
$n$th components.

The function $\tilde G_{N}$ has the following estimates. For $e$ with $|e|=1$,
we denote $\nabla_{j,e}\tilde G_{N}(i,j)= \tilde G_{N}(i,j+e)-\tilde
G_{N}(i,j)$ and similar for $\nabla_{j,e} G_{N}(i,j)$.

\begin{lemma}
\label{lem:G-1}
{\rm (1)} For $i,j\in\tilde{D}_{N}$, we have%
\begin{equation} \label{eq:G-3}
|\nabla_{j,e}\tilde{G}_{N}(i,j)|\leq\frac{C}{1+|i-j|^{d-1}}+E_{i}\left[
\frac{C}{1+|\tilde{\eta}_{\tilde{\sigma}}-j|^{d-1}}\right]  
\end{equation}
with some $C>0.$ \\
{\rm (2)} With the natural embedding of $D_{N}\subset\tilde{D}_{N}$, we have
\[
\sup_{i\in\tilde{D}_{N}}\sum_{j\in kN+D_{N}}|\nabla_{j,e}\tilde{G}%
_{N}(i,j)|\leq CN.
\]
{\rm (3)} We have 
\begin{align} 
\tilde G_N(i,j) \le \frac{C}{N^{d-2}} e^{-c|i-j|/N}, \quad \text{ if } \;
  |i-j|\ge 5N,  \label{eq:G-4}
\end{align}
with some $C, c>0$.
\end{lemma}

\begin{proof}
To show \eqref{eq:G-3}, we rewrite $\tilde{G}_{N}(i,j)$ with the random walk
$\tilde{\eta}_{n}$ on ${\mathbb{Z}}^{d}$ and its hitting time $\tilde{\sigma}$
to $\partial\tilde{D}_{N}$ as
\begin{align}
\tilde{G}_{N}(i,j)  &  =\sum_{n=0}^{\infty}P_{i}(\tilde{\eta}_{n}%
=j,n<\tilde{\sigma})\label{eq:G-0}\\
&  =\sum_{n=0}^{\infty}P_{i}(\tilde{\eta}_{n}=j)-\sum_{n=0}^{\infty}%
P_{i}(\tilde{\eta}_{n}=j,n\geq\tilde{\sigma})\nonumber\\
&  =G(i,j)-E_{i}[G_{N}(\tilde{\eta}_{\tilde{\sigma}},j)],\nonumber
\end{align}
by the strong Markov property of $\tilde{\eta}_{n}$. Therefore, we have%
\[
|\nabla_{j,e}\tilde{G}_{N}(i,j)|\leq|\nabla_{j,e}G(i,j)|+E_{i}[|\nabla
_{j,e}G_{N}(\tilde{\eta}_{\tilde{\sigma}},j)|],
\]
and we obtain \eqref{eq:G-3} from the well-known estimate on the Green's
function $G$ on ${\mathbb{Z}}^{d}$ (e.g., \cite{L}, Theorem 1.5.5, p.32). This
proves (1). (2) is an immediate consequence of (1), as%
\[
\sup_{i}\sum_{j\in kN+D_{N}}\frac{1}{1+|i-j|^{d-1}}\leq CN.
\]
The next task is to show \eqref{eq:G-4}.  We assume $i\in D_N$
and $j=j_0+kN$ with  $j_0\in D_N$ and $k\in {\mathbb Z}^{d-1}$.
We denote $\Gamma_N(0) = \{\underbar{$i$} = (i_2,\ldots,i_d) \in {\mathbb Z}^{d-1},
0\le i_\ell<N, \ell=2,\ldots,d\}$ the box in ${\mathbb Z}^{d-1}$
with side length $N$ and divide ${\mathbb Z}^{d-1}$
into a disjoint union of boxes $\{\Gamma_N(\underbar{$i$}) = \underbar{$i$}
+ \Gamma(0); \underbar{$i$} \equiv 0 \text{ modulo }N\}$.  For $k\in {\mathbb Z}^{d-1}$,
let $\Gamma_{3N}(k)$ be the box with side length $3N$ with $\Gamma_N(\underbar{$i$})$
as its center, where $\underbar{$i$}$ is determined in such a manner
that $k \in \Gamma_N(\underbar{$i$})$.
We set $\bar\sigma := \inf\{n\ge 0; (\tilde\eta^{(2)}_n, \ldots,
\tilde\eta^{(d)}_n) \in \Gamma_{3N}(k)\}$.  Note that $i$ and $\Gamma_{3N}(k)$
are separate enough by the condition $|i-j|\ge 5N$.
Then, by the strong Markov property,
\begin{align*} 
\tilde G_N(i,j) & = E_i\left[\sum_{n=0}^\infty 1_{\{\tilde\eta_n=j_0+kN, n\le \tilde\sigma\}} \right]\\
& = E_i\left[E_{\tilde\eta_{\bar\sigma}}\left[\sum_{n=0}^\infty 1_{\{\tilde\eta_n=j_0+kN, 
n\le \tilde\sigma\}} \right], \bar\sigma<\tilde\sigma\right]\\
& = E_i\left[\tilde G_N(\tilde\eta_{\bar\sigma}, j_0+kN), 
\bar\sigma<\tilde\sigma\right]\\
& \le \frac{C}{N^{d-2}} P_i(\bar\sigma<\tilde\sigma),
\end{align*}
since $|\tilde\eta_{\bar\sigma} - (j_0+kN)| \ge N$ and $\tilde G_N(i,j) \le G(i,j)
\le \frac{C}{1+|i-j|^{d-2}}$.   The event $\{\bar\sigma<\tilde\sigma\}$ means that 
the 2nd--$d$th components of the random walk $\underbar{$\tilde\eta$}_n 
:= (\tilde\eta_n^{(2)},\ldots, \tilde\eta_n^{(d)})$ hits $\{i\in \tilde D_N; 
\underbar{$i$}\in \partial \Gamma_{3N}(k)\}$ before the 1st component 
of the random walk $\tilde\eta_n^{(1)}$ hits $\{0,N\}$ (namely, the random walk
$\tilde\eta$ hits $\partial\tilde D_N$).  In other words, 
$\underbar{$\tilde\eta$}$ passes at least $|k|-2$ 
boxes $\Gamma_N(\underbar{$i$})$ before $\tilde\eta_n^{(1)}$ reaches
the boundary of one box of the same size.  Such probability
can be bounded by the geometric distribution
so that we obtain the desired estimate.
\end{proof}

The following lemma will be used in the proof of Proposition
\ref{Prop_Coarse_graining}.

\begin{lemma}  \label{lem:2.5-b}
We have that
\[
\sup_{i\in D_{N}}\sum_{j\in D_{N}}|\nabla_{j,e}G_{N}(i,j)|\leq CN.
\]
\end{lemma}

\begin{proof}
For $k\in\mathbb{Z}^{d-1}$, we write $D_{N}^{\left(  k\right)  }$ for
$D_{N}+kN$ enlarged by \textquotedblleft one layer\textquotedblright, so that
for any $j,e$, we can find $k$ with $j,j\in D_{N}^{\left(  k\right)  }.$ Let
$\tau_{k}$ for the first entrance time of the random walk $\left\{
\tilde{\eta}_{n}\right\}  $ into $D_{N}^{\left(  k\right)  }$. ($\tau_{k}=0$
if $\tilde{\eta}_{n}\in D_{N}^{\left(  k\right)  }$). Remember that
$\tilde{\sigma}$ was the first hitting time of $\partial\tilde{D}_{N}.$ Using
the strong Markov property, we have for $j,j+e\in D_{N}^{\left(  k\right)  },$%
\[
\tilde{G}_{N}(i,j)-\tilde{G}_{N}(i,j+e)=E_{i}[\left(  G_{N}(\tilde{\eta}%
_{\tau_{k}},j)-G_{N}(\tilde{\eta}_{\tau_{k}},j+e)\right)  1_{\tau_{k}%
<\tilde{\sigma}}].
\]
We use the representation (\ref{eq:G-2}) which leads to%
\[
\sum_{j\in D_{N}}\left\vert \nabla_{j,e}G_{N}(i,j)\right\vert =\sum
_{j\in\tilde{D}_{N}}\left\vert \nabla_{j,e}\tilde{G}_{N}(i,j)\right\vert
\leq\sum_{k\in\mathbb{Z}^{d-1}}\sum_{j\in D_{N}^{\left(  k\right)  }%
}\left\vert \nabla_{j,e}\tilde{G}_{N}(i,j)\right\vert .
\]
Using Lemma \ref{lem:G-1}-(2), we have%
\[
\sum_{j\in D_{N}^{\left(  k\right)  }}\left\vert \nabla_{j,e}\tilde{G}%
_{N}(i,j)\right\vert \leq CNP_{i}\left(  \tau_{k}<\tilde{\sigma}\right)  ,
\]
implying%
\[
\sum_{j\in D_{N}}|\nabla_{j,e}G_{N}(i,j)|\leq CN\sum_{k}P_{i}\left(  \tau
_{k}<\tilde{\sigma}\right)  .
\]
It is however easy to see that for $i\in D_{N},$ $P_{i}\left(  \tau_{k}%
<\tilde{\sigma}\right)  $ is exponentially decaying in $\left\vert
k\right\vert $, so the sum on the left hand side is finite, with a bound which
is independent of $i\in D_{N}$.
\end{proof}

\subsection{Decoupling estimate, the case without pinning} \label{section:2.4}

The next lemma, which corresponds to Lemma 2.3.1-c) in \cite{BI}, is prepared
for the next subsection. We set
\begin{align*}
c_N := & \sup_{k\in D_N^\circ} \sum_{n=1}^\infty
P_k^{RW^{d,N}}(\eta_{2n}=k, 2n<\sigma) \\
= & \sup_{k\in D_N^\circ} G_N(k,k),
\intertext{and, recalling $d\ge 3$,}
\bar c := & G(0,0) = \sum_{n=1}^{\infty}P_{0}^{RW^{d}}(\eta_{2n}=0).
\end{align*}

\begin{lemma} \label{lem:3} 
Assume $d\geq3$. Then, we have the following two assertions. \\
{\rm (1)} $c_N$ is bounded: $c_N \le C.$  \\
{\rm (2)} For two disjoint sets $A,C\subset D_{N}^{\circ}$, if $N$ is even,
we have
\[
0\leq\log\frac{Z_{A\cup C}^{0}}{Z_{A}^{0}Z_{C}^{0}}\leq\frac{c_N}%
{2}|\partial_{A}C|,
\]
where $\partial_{A}C=\partial A\cap C$.
\end{lemma}

\begin{proof}
For (1), from \eqref{eq:G-2}, we have that
\begin{equation} \label{eq:G-8}
G_N(k,k) = \sum_{\ell\in {\mathbb Z}^{d-1}} \tilde G_N(k,k+\ell N).
\end{equation}
From \eqref{eq:G-0}, we see that $\tilde G_N(i,j) \le G(i,j)$.
Since $G(i,j)$ is bounded, the sum in the right hand side of \eqref{eq:G-8} 
over $\ell: |\ell|\le 5$ is bounded in $N$.
To show the sum over $|\ell|\ge 6$ is also bounded, we can apply the estimate
 \eqref{eq:G-4}:
$$
\sum_{\ell\in {\mathbb Z}^{d-1:|\ell|\ge 6}} \tilde G_N(k,k+\ell N)
\le C\sum_{|\ell|\ge 6} e^{-c|\ell|} < \infty.
$$
For (2), we follow the arguments in the middle of p.544 of \cite{BI}. From
\eqref{eq:3.8} and \eqref{eq:3.9}, we have
\begin{align*}
2\log\frac{Z_{A\cup C}^{0}}{Z_{A}^{0}Z_{C}^{0}}=  &  \sum_{k\in A}\sum
_{n=1}^{\infty}\frac{1}{2n}P_{k}^{RW^{d,N}}(\eta_{2n}=k,\tau_{A}%
<2n<\tau_{A\cup C})\\
&  +\sum_{k\in C}\sum_{n=1}^{\infty}\frac{1}{2n}P_{k}^{RW^{d,N}}(\eta
_{2n}=k,\tau_{C}<2n<\tau_{A\cup C}),
\end{align*}
note that \textquotedblleft$\tau_{A}=2n$\textquotedblright\ does not occur
under \textquotedblleft$\eta_{2n}=k\in A$\textquotedblright. The lower bound
in (1) is now clear. To show the upper bound, as in \cite{BI}, 
noting that $\tau_{A\cup C} \le \sigma$ for $A, C\subset D_N^\circ$,
we further estimate the right hand side by
\begin{align*}
&  \leq\sum_{k\in\partial_{A}C}\sum_{n=1}^{\infty}P_{k}^{RW^{d,N}}(\eta
_{2n}=k,\tau_{A\cup C}>2n)\\
&  \leq|\partial_{A}C|\sum_{n=1}^{\infty}P_{0}^{RW^{d,N}}(\eta_{2n}%
=0, 2n<\sigma)=c_N|\partial_{A}C|,
\end{align*}
which concludes the proof of the assertion (2).
\end{proof}

\subsection{Estimates on the partition functions with pinning}

For $A\subset D_{N}^{\circ}$, we set
\[
Z_{A}^{0,\varepsilon}=\int_{{\mathbb{R}}^{A}}\mathrm{e}^{-H_{A}^{0}(\phi
)}\prod_{i\in A}\left(  \varepsilon\delta_{0}(d\phi_{i})+d\phi_{i}\right)  .
\]
The next lemma, which corresponds to Lemma 2.3.1-b) in \cite{BI}, is proved
based on Lemma \ref{lem:3}.

\begin{lemma}
\label{lem:4} Assume $d\geq3$. Then, there exists a constant $\hat
{q}^{\varepsilon}>0$ such that
\[
\hat{q}^{\varepsilon}|A|-\frac{\bar c}{4}\left(  |\partial A|+4\ell_{1}%
(A)N^{d-2}\right)  \leq\log Z_{A}^{0,\varepsilon}\leq\hat{q}^{\varepsilon
}|A|+c_N\ell_{1}(A)N^{d-2},
\]
for every rectangles $A\subset D_{N}^{\circ}$, where $\ell_{1}(A)$ denotes the
side length of $A$ in the first coordinate's direction.
\end{lemma}

\begin{proof}
We follow the arguments from the bottom of p.544 to p.545 of \cite{BI} noting
that we are discussing under the periodic boundary condition for the second to
the $d$th coordinates. We first observe that
\begin{equation}
\log Z_{B}^{0,\varepsilon}+\log Z_{B^{\prime}}^{0,\varepsilon}\leq\log
Z_{B\cup B^{\prime}}^{0,\varepsilon}\leq\log Z_{B}^{0,\varepsilon}+\log
Z_{B^{\prime}}^{0,\varepsilon}+\frac{c_N}{2}|\partial_{B}B^{\prime}|,
\label{eq:3.a}%
\end{equation}
for every disjoint $B,B^{\prime}\subset D_{N}^{\circ}$. In fact, since
\[
Z_{B\cup B^{\prime}}^{0,\varepsilon}=\sum_{A\subset B}\sum_{C\subset
B^{\prime}}\varepsilon^{|B\setminus A|+|B^{\prime}\setminus C|}Z_{A\cup C}%
^{0},
\]
the lower bound in \eqref{eq:3.a} follows from $Z_{A\cup C}^{0}\geq Z_{A}%
^{0}Z_{C}^{0}$ (see the lower bound in Lemma \ref{lem:3}-(2)), while the upper
bound follows from
\[
Z_{A\cup C}^{0}\leq Z_{A}^{0}Z_{C}^{0}\mathrm{e}^{\frac{1}{2}c_N%
|\partial_{A}C|}\leq Z_{A}^{0}Z_{C}^{0}\mathrm{e}^{\frac{1}{2}c_N%
|\partial_{B}B^{\prime}|}.
\]
In a similar way, we have that
\begin{equation}
\log Z_{B}^{0,\varepsilon}+\log Z_{B^{\prime}}^{0,\varepsilon}\leq\log
Z_{B\cup B^{\prime}}^{0,\varepsilon}\leq\log Z_{B}^{0,\varepsilon}+\log
Z_{B^{\prime}}^{0,\varepsilon}+\frac{\bar c}{2}|\partial_{B}B^{\prime}|,
\label{eq:2.11}%
\end{equation}
for every disjoint $B,B^{\prime}\Subset\mathbb{Z}^{d}$ (or $B,B^{\prime
}\subset D_N^\circ$ which do not contain loops in periodic directions).

For $p=(p_{1},\ldots,p_{d})\in\mathbb{N}^{d}$, let $S_{p}=\prod_{\alpha=1}%
^{d}[1,p_{\alpha}]\cap\mathbb{Z}^{d}$ be the rectangle in $\mathbb{Z}^{d}$
with volume $|S_{p}|=\prod_{\alpha=1}^{d}p_{\alpha}$ and set $Q(p)=\frac
{1}{|S_{p}|}\log Z_{S_{p}}^{0,\varepsilon}$. Then, one can show that the
limit
\[
\hat{q}^{\varepsilon}=\lim_{m\rightarrow\infty}Q(2^{m}p)
\]
exists (independently of the choice of $p$) and
\begin{equation}
\hat{q}^{\varepsilon}-\frac{\bar c}{4}\frac{|\partial S_{p}|}{|S_{p}|}\leq
Q(p)\leq\hat{q}^{\varepsilon} \label{eq:2.12}%
\end{equation}
holds for every $p\in\mathbb{N}^{d}$. Indeed, as in \cite{BI}, \eqref{eq:2.11}
implies that
\[
Q(p)\leq\cdots\leq Q(2^{m-1}p)\leq Q(2^{m}p)\leq Q(2^{m-1}p)+\frac{\bar c}{4}%
\frac{|\partial S_{2^{m}p}|}{|S_{2^{m}p}|}.
\]
By letting $m\rightarrow\infty$, we obtain that
\[
Q(p)\leq\hat{q}^{\varepsilon}\leq Q(p)+\frac{\bar c}{4}\sum_{m=1}^{\infty}%
\frac{|\partial S_{2^{m}p}|}{|S_{2^{m}p}|}=Q(p)+\frac{\bar c}{4}\frac{|\partial
S_{p}|}{|S_{p}|},
\]
which implies \eqref{eq:2.12}.

The conclusion of the lemma follows from \eqref{eq:2.12} if $A=S_{p}\subset
D_{N}^{\circ}$ does not contain loops in periodic directions. In fact, for
such $A$, better inequalities hold:
\[
\hat{q}^{\varepsilon}|A|-\frac{\bar c}{4}|\partial A|\leq\log Z_{A}^{0,\varepsilon
}\leq\hat{q}^{\varepsilon}|A|.
\]
If the rectangle $A\subset D_{N}^{\circ}$ is periodically connected, we divide
it into two rectangles: $A=A_{1}\cup A_{2}$, where $A_{1}=A\cap\{i\in
D_{N}^{\circ};0\leq i_{2}\leq\frac{N}{2}-1\}$ and $A_{2}=A\cap\{i\in
D_{N}^{\circ};\frac{N}{2}\leq i_{2}\leq N-1\}$. Then, noting that
$|\partial_{A_{1}}A_{2}|=2\ell_{1}(A)N^{d-2}$, the conclusion follows from
\eqref{eq:3.a} (with $B=A_{1},B^{\prime}=A_{2}$) and \eqref{eq:2.12} (applied
for each of $A_{1}$ and $A_{2}$).
\end{proof}

\begin{remark}
\label{rem:3.3}
{\rm (1)} The $\varepsilon$-dependent quantity is only $\hat{q}^{\varepsilon
}$; $c_N$ and $\bar c$ are independent of $\varepsilon$.\\
{\rm (2)} Lemmas \ref{lem:2}, \ref{lem:4} (for $A\Subset{\mathbb{Z}}^{d}$)
and \eqref{eq:1.5} imply that $\xi^{\varepsilon}=\hat{q}^{\varepsilon}-\hat
{q}^{0}$ and $\hat{q}^{0}=\frac{1}{2}(\log\frac{\pi}{d}+q)$.
\end{remark}

\section{Stability result}  \label{section:3}

\subsection{Stability at macroscopic level}

Recall that the macroscopic energy (LD unnormalized rate functional) 
$\Sigma(h)$ of $h:D\to{\mathbb{R}}$ is given by \eqref{eq:1.Sigma}.
We set $\Sigma^{*}(h) =\Sigma(h)-\min\Sigma$. Note that $\Sigma(\bar h) =
\frac12 (a-b)^{2}$ and $\Sigma(\hat h) = \sqrt{2\xi^{\varepsilon}} (a+b) -
\xi^{\varepsilon}$, see p.446 of \cite{BFO}, and $\Sigma(\bar h) = \Sigma(\hat h) =
\min\Sigma$ from our assumption.

\begin{proposition}
\label{prop:macrostability} If $\delta_{1} >0$ is sufficiently small,
$\Sigma^{*}(h) < \delta_{1}$ implies $d_{L^{1}}(h,\{\bar h, \hat h\}) <
\delta_{2}$ with $\delta_{2}= c\delta_{1}^{1/4}$ and some $c>0$.
\end{proposition}

\begin{remark}
The metric $d_{L^{1}}$ can be extended to $d_{L^{p}}$ with $p\in\lbrack
1,\frac{2d}{d-2})$, but with different rates for $\delta_{2}$.
\end{remark}

We begin with the stability in one-dimension under a stronger $L^{\infty}$-topology.

\begin{lemma}
\label{lem:stability-1} If $\delta_{1}>0$ is sufficiently small, for
$g:[0,1]\rightarrow\mathbb{R}$, $\Sigma^{\ast}(g)<\delta_{1}$ implies
$d_{L^{\infty}}(g,\{\bar{h}^{(1)},\hat{h}^{(1)}\})<\delta_{2}$ with
$\delta_{2}=\sqrt{\delta_{1}}$.
\end{lemma}

\begin{proof}
Let us assume $d_{L^{\infty}}(g,\{\bar{h}^{(1)},\hat{h}^{(1)}\})\geq\delta
_{2}$, that is, $d_{L^{\infty}}(g,\bar{h}^{(1)})\geq\delta_{2}$ and
$d_{L^{\infty}}(g,\hat{h}^{(1)})\geq\delta_{2}$. First, we consider the case
where $g$ does not touch $0$, more precisely, $|\{t_{1}\in\lbrack
0,1];g(t_{1})=0\}|=0$. Then, the condition $d_{L^{\infty}}(g,\bar{h}%
^{(1)})\geq\delta_{2}$ implies
\begin{equation}
\Sigma^{\ast}(g)\geq2\delta_{2}^{2}. \label{eq:S-1}%
\end{equation}
Indeed, since the straight line has the lowest energy among curves which have
the same heights at both ends and do not touch $0$, we consider piecewise
linear functions $g^{t_{0}}$ with $t_{0}\in(0,1)$ defined by
\[
g^{t_{0}}(t_{1})=\left\{
\begin{array}
[c]{cc}%
a+\left\{  \left(  b-a\right)  \pm\frac{\delta_{2}}{t_{0}}\right\}  t_{1} &
\text{for } t_{1}\in\left[  0,t_{0}\right] \\
b+\left\{  \left(  b-a\right)  \pm\frac{\delta_{2}}{t_{0}-1}\right\}  \left(
t_{1}-1\right)  & \text{for } t_{1}\in\left[  t_{0},1\right]
\end{array}
\right.  .
\]
These functions satisfy $d_{L^{\infty}}(g^{t_{0}},\bar{h}^{(1)})=\delta_{2} $.
Thus, for $g$ satisfying $d_{L^{\infty}}(g,\bar{h}^{(1)})\geq\delta_{2} $ and
not touching $0$, we see that
\begin{align*}
\Sigma^{\ast}(g)  &  \geq\inf_{t_{0}\in(0,1)}\Sigma^{\ast}(g^{t_{0}})\\
&  =\inf_{t_{0}\in(0,1)}\frac{\delta_{2}^{2}}{2}\left(  \frac{1}{t_{0}}%
+\frac{1}{1-t_{0}}\right)  =2\delta_{2}^{2},
\end{align*}
by a simple computation, which proves \eqref{eq:S-1}.

Next, we consider the case where $g$ touches $0$, i.e., $|\{t_{1}\in
\lbrack0,1];g(t_{1})=0\}|>0$. Then, the condition $d_{L^{\infty}}(g,\hat
{h}^{(1)})\geq\delta_{2}$ implies
\begin{equation}
\Sigma^{\ast}(g)\geq\min\{2\sqrt{2\xi^{\varepsilon}}\delta_{2},\delta_{2}%
^{2}\}\;(=\delta_{2}^{2}\;\text{ if }\delta_{2}\leq2\sqrt{2\xi^{\varepsilon}%
}). \label{eq:S-2}%
\end{equation}
Indeed, it is known that the interval $[s_1^L,s_1^R]=\{t_{1}%
\in\lbrack0,1];\hat{h}^{(1)}(t_{1})=0\}$ of zeros of $\hat{h}^{(1)}$ is
determined by the so-called Young's relation:
\begin{equation}
\frac{a}{s_1^L}=\frac{b}{1-s_1^R}=\sqrt{2\xi^{\varepsilon}},
\label{eq:S-8}%
\end{equation}
see \cite{Fu05}, p.176, (6.26). Here we assume $a,b>0$ for simplicity.
First, consider the case where the discrepancy at least of size $\delta_{2}$
of $g$ from $\hat{h}^{(1)}$ occurs at $t_{0}\in\lbrack s_1^L
,s_1^R]$. For such $g$, the energy $\Sigma_{\lbrack0,t_{0}]}$ on the
interval $[0,t_{0}]$ has a lower bound:
\begin{align*}
\Sigma_{\lbrack0,t_{0}]}(g)  &  \geq\Sigma_{\lbrack0,t_{0}]}(\hat{g}%
_{[0,t_{0}]})\\
&  =\frac{a^{2}}{2s_1^L}-\sqrt{2\xi^{\varepsilon}}(t_{0}-s_1^L
-\theta)+\frac{\delta_{2}^{2}}{2\theta}\\
&  =(a+\delta_{2})\sqrt{2\xi^{\varepsilon}}-\xi^{\varepsilon}t_{0},
\end{align*}
where $\theta$ is determined by $\frac{\delta_{2}}{\theta}=\sqrt
{2\xi^{\varepsilon}}$, and $\hat{g}_{[0,t_{0}]}:[0,t_{0}]\rightarrow
\mathbb{R}$ is the minimizer of $\Sigma_{\lbrack0,t_{0}]}$ among the curves
$g:[0,t_{0}]\rightarrow\mathbb{R}$ satisfying $g(0)=a$ and $g(t_{0}%
)=\delta_{2}$. Note that, by Young's relation \eqref{eq:S-8}, $\{t_{1}%
\in\lbrack0,t_{0}];\hat{g}_{[0,t_{0}]}(t_{1})=0\}=[s_1^L,t_{0}-\theta
]$, and also $\delta_{2}$ is sufficiently small. Similarly, on the interval
$[t_{0},1]$, we can show that
\[
\Sigma_{\lbrack t_{0},1]}(g)\geq\Sigma_{\lbrack t_{0},1]}(\hat{g}_{[t_{0}%
,1]})=(\delta_{2}+b)\sqrt{2\xi^{\varepsilon}}-\xi^{\varepsilon}(1-t_{0}).
\]
Therefore, for $g$ mentioned above, we have that
\begin{equation}
\Sigma^{\ast}(g)\geq\Sigma_{\lbrack0,t_{0}]}(\hat{g}_{[0,t_{0}]}%
)+\Sigma_{\lbrack t_{0},1]}(\hat{g}_{[t_{0},1]})-\min\Sigma=2\sqrt
{2\xi^{\varepsilon}}\delta_{2}. \label{eq:S-3}%
\end{equation}
Next, consider the case where the discrepancy occurs at $t_{0}\in
\lbrack0,s_1^L]$. For such $g$, we have that
\begin{align*}
\Sigma_{\lbrack0,t_{0}]}(g)  &  \geq\Sigma_{\lbrack t_{0},1]}(g^{t_{0}}%
)=\frac{t_{0}}{2}\left(  \frac{a}{s_1^L}+\frac{\delta_{2}}{t_{0}%
}\right)  ^{2}\,\left(  =\frac{t_{0}}{2}\left(  \sqrt{2\xi^{\varepsilon}%
}+\frac{\delta_{2}}{t_{0}}\right)  ^{2}\right)  ,\\
\Sigma_{\lbrack t_{0},1]}(g)  &  \geq\Sigma_{\lbrack t_{0},1]}(\hat{g}%
_{[t_{0},1]})=\left(  (a-\sqrt{2\xi^{\varepsilon}}t_{0}-\delta_{2})+b\right)
\sqrt{2\xi^{\varepsilon}}-\xi^{\varepsilon}(1-t_{0}),
\end{align*}
where $g^{t_{0}}:[0,t_{0}]\rightarrow\mathbb{R}$ is a linear function
satisfying $g^{t_{0}}(t_{0})=\hat{h}^{(1)}(t_{0})-\delta_{2}\left(
=a-\frac{a}{s_1^L}t_{0}-\delta_{2}\right)  $, and $\hat{g}_{[t_{0}%
,1]}:[t_{0},1]\rightarrow\mathbb{R}$ is the minimizer of $\Sigma_{\lbrack
t_{0},1]}$ satisfying $\hat{g}_{[t_{0},1]}(t_{0})=\hat{h}^{(1)}(t_{0}%
)-\delta_{2}$. Therefore, for such $g$, we have that
\[
\Sigma^{\ast}(g)\geq\Sigma_{\lbrack t_{0},1]}(g^{t_{0}})+\Sigma_{\lbrack
t_{0},1]}(\hat{g}_{[t_{0},1]})-\min\Sigma=\frac{\delta_{2}^{2}}{2t_{0}}%
>\delta_{2}^{2},
\]
since $t_{0}<\frac{1}{2}$. The case where $t_{0}\in\lbrack s_1^R,1]$ is
similar, and this together with \eqref{eq:S-3} shows \eqref{eq:S-2}. The
conclusion of the lemma follows from \eqref{eq:S-1} and \eqref{eq:S-2} if
$\delta_{1}\leq(2\sqrt{2\xi^{\varepsilon}})^{2}$.
\end{proof}

We prepare another lemma.

\begin{lemma}
\label{lem:stability-2} Assume $d\ge2$. Then, $\Sigma^{*}(h) \le C_{1}$
implies $\|h\|_{L^{q}}\le C_{2}$ for every $2\le q \le\frac{2d}{d-2}$ (or
$2\le q <\infty$ when $d=2$) and some $C_{2}=C_{2}(q,C_{1})>0$.
\end{lemma}

\begin{proof}
The condition $\Sigma^{\ast}(h)\leq C_{1}$ shows
\[
\frac{1}{2}\int_{D}|\nabla h(t)|^{2}dt\leq C_{1}+\xi^{\varepsilon}+\min
\Sigma.
\]
This, together with Poincar\'{e} inequality noting that $h=a$ on $\partial
_{L}D$ and $h=b$ on $\partial_{R}D$, proves that $\Vert h\Vert_{W^{1,2}%
(D)}\leq C_{2}$. However, Sobolev's imbedding theorem (e.g., \cite{AF},
p85) implies the continuity of the imbedding $W^{1,2}(D)\subset L^{q}(D)$ for
$2\leq q\leq\frac{2d}{d-2}$ and this concludes the proof of the lemma.
\end{proof}

We are now at the position to give the proof of Proposition
\ref{prop:macrostability}.

\begin{proof}
[Proof of Proposition \ref{prop:macrostability}]Assume that $h$ satisfies
\begin{equation}
\Sigma^{\ast}(h)=\int_{\mathbb{T}^{d-1}}\Sigma^{(1),\ast}(h(\cdot
,\underline{t}))d\underline{t}+\frac{1}{2}\int_{D}|\nabla_{\underline{t}%
}h(t_{1},\underline{t})|^{2}dt<\delta_{1}, \label{eq:S-10}%
\end{equation}
where $\Sigma^{(1)}(g)$ is the energy of $g:[0,1]\rightarrow\mathbb{R}$,
$\Sigma^{(1),\ast}=\Sigma^{(1)}-\min\Sigma^{(1)}$; recall \eqref{eq:1.7}. 
Note that $\min\Sigma
^{(1)}=\min\Sigma$ so that we have the above expression for $\Sigma^{\ast}%
(h)$. We assume $\delta_{1}>0$ is sufficiently small. For $M\geq2$ chosen
later, set
\begin{align*}
S_{M\delta_{1}}^{d-1}  &  :=\{\underline{t}\in\mathbb{T}^{d-1};\Sigma
^{(1),\ast}(h(\cdot,\underline{t}))<M\delta_{1}\},\\
\hat{S}_{M\delta_{1}}^{d-1}  &  :=\mathbb{T}^{d-1}\setminus S_{M\delta_{1}%
}^{d-1}=\{\underline{t}\in\mathbb{T}^{d-1};\Sigma^{(1),\ast}(h(\cdot
,\underline{t}))\geq M\delta_{1}\}.
\end{align*}
Then, by \eqref{eq:S-10} and Chebyshev's inequality,
\[
|\hat{S}_{M\delta_{1}}^{d-1}|\leq\frac{\delta_{1}}{M\delta_{1}}=\frac{1}{M},
\]
and
\[
|S_{M\delta_{1}}^{d-1}|\geq1-\frac{1}{M}.
\]

We first estimate the contribution to $d_{L^{1}(D)}(h,\bar{h})=\Vert h-\bar
{h}\Vert_{L^{1}(D)}$ and $d_{L^{1}(D)}(h,\hat{h})$ from the region $\hat
{S}_{M\delta_{1}}^{d-1}$, or more generally regions $S\subset\mathbb{T}^{d-1}$
such that $|S|\leq\frac{1}{M}$:
\begin{align}
\int_{S}  &  \Vert h(\cdot,\underline{t})-\bar{h}^{(1)}\Vert_{L^{1}%
([0,1])}d\underline{t}\label{eq:S-4}\\
&  \leq\int_{S}\{\Vert h(\cdot,\underline{t})\Vert_{L^{1}([0,1])}+\Vert\bar
{h}^{(1)}\Vert_{L^{1}([0,1])}\}d\underline{t}\nonumber\\
&  =\int_{D}1_{[0,1]\times S}(t)|h(t)|dt+C|S|\nonumber\\
&  \leq\sqrt{|[0,1]\times S|}\;\Vert h\Vert_{L^{2}(D)}+\frac{C}{M}\nonumber\\
&  \leq\frac{C_{2}}{\sqrt{M}}+\frac{C}{M}\leq\frac{C_{3}}{\sqrt{M}},\nonumber
\end{align}
where $C=\Vert\bar{h}^{(1)}\Vert_{L^{1}([0,1])}<\infty$ and $C_{3}=C_{2}+C $.
We have applied Schwarz's inequality for the fourth line and Lemma
\ref{lem:stability-2} for the fifth line with $C_{2}=C_{2}(2,\delta_{1})$. We
similarly have
\[
\int_{S}\Vert h(\cdot,\underline{t})-\hat{h}^{(1)}\Vert_{L^{1}([0,1])}%
d\underline{t}\leq\frac{C_{4}}{\sqrt{M}}.
\]

For $\underline{t}\in S_{M\delta_{1}}^{d-1}$, by Lemma \ref{lem:stability-1},
we see that
\[
d_{L^{\infty}}(h(\cdot,\underline{t}),\{\bar{h}^{(1)},\hat{h}^{(1)}%
\})<\delta_{2}(=\sqrt{M\delta_{1}}).
\]
Set
\begin{align*}
S_{M\delta_{1}}^{d-1,(1)}  &  :=\{\underline{t}\in S_{M\delta_{1}}%
^{d-1};d_{L^{\infty}}(h(\cdot,\underline{t}),\bar{h}^{(1)})<\delta_{2}\},\\
S_{M\delta_{1}}^{d-1,(2)}  &  :=\{\underline{t}\in S_{M\delta_{1}}%
^{d-1};d_{L^{\infty}}(h(\cdot,\underline{t}),\hat{h}^{(1)})<\delta_{2}\}.
\end{align*}
If $|S_{M\delta_{1}}^{d-1,(2)}|\leq\frac{1}{M}$, we have from \eqref{eq:S-4}
that
\begin{align}
d_{L^{1}(D)}(h,\bar{h})  &  =\Vert h-\bar{h}\Vert_{L^{1}(D)}=\int
_{\mathbb{T}^{d-1}}\Vert h(\cdot,\underline{t})-\bar{h}^{(1)}\Vert
_{L^{1}([0,1])}d\underline{t}\label{eq:S-6}\\
&  \leq\sqrt{M\delta_{1}}+\frac{2C_{3}}{\sqrt{M}}=C_{5}\delta_{1}%
^{1/4},\nonumber
\end{align}
by dividing $\mathbb{T}^{d-1}=S_{M\delta_{1}}^{d-1,(1)}\cup(\hat{S}%
_{M\delta_{1}}^{d-1}\cup S_{M\delta_{1}}^{d-1,(2)})$ and choosing
$M=1/\sqrt{\delta_{1}}$, with $C_{5}=1+2C_{3}$. We have a similar bound:
\begin{equation}
d_{L^{1}(D)}(h,\hat{h})=\Vert h-\hat{h}\Vert_{L^{1}(D)}\leq C_{6}\delta
_{1}^{1/4}, \label{eq:S-7}%
\end{equation}
if $|S_{M\delta_{1}}^{d-1,(1)}|\leq\frac{1}{M}=\sqrt{\delta_{1}}$.

Therefore, the case where both $|S_{M\delta_{1}}^{d-1,(1)}|$, $|S_{M\delta
_{1}}^{d-1,(2)}|\geq\frac{1}{M}=\sqrt{\delta_{1}}$ is left. In this case,
since $|S_{M\delta_{1}}^{d-1}|\geq1-\frac{1}{M}\geq\frac{1}{2}$ (since
$M\geq2$), the volume of $S_{M\delta_{1}}^{d-1,(1)}$ or $S_{M\delta_{1}%
}^{d-1,(2)}$ is larger than $\frac{1}{4}$. Let us assume $|S_{M\delta_{1}%
}^{d-1,(1)}|\geq\frac{1}{4}$ and $|S_{M\delta_{1}}^{d-1,(2)}|\geq\sqrt
{\delta_{1}}$. The case $|S_{M\delta_{1}}^{d-1,(2)}|\geq\frac{1}{4}$ and
$|S_{M\delta_{1}}^{d-1,(1)}|\geq\sqrt{\delta_{1}}$ can be treated similarly.
Then, choosing a subset $S\subset S_{M\delta_{1}}^{d-1,(2)}$ such that
$|S|=\sqrt{\delta_{1}}$, we have that
\begin{align}
\int_{0}^{1}  &  dt_{1}\int_{S_{M\delta_{1}}^{d-1,(1)}}d\underline{t}\int
_{S}d\underline{t}^{\ast}\;\mathrm{Av}\!\!\int_{0}^{1}(\underline
{t}-\underline{t}^{\ast})\cdot\nabla_{\underline{t}}h(t_{1},\alpha
\underline{t}+(1-\alpha)\underline{t}^{\ast})d\alpha\label{eq:S-5}\\
&  =\int_{S_{M\delta_{1}}^{d-1,(1)}}d\underline{t}\int_{S}d\underline{t}%
^{\ast}\int_{0}^{1}\{h(t_{1},\underline{t})-h(t_{1},\underline{t}^{\ast
})\}dt_{1}\nonumber\\
&  \geq\frac{c_{\delta_{2}}}{4}\sqrt{\delta_{1}}\geq\frac{C_{7}}{4}%
\sqrt{\delta_{1}},\nonumber
\end{align}
by integrating in $\alpha$ first, where
\[
\mathrm{Av}\!\!\int_{0}^{1}f(\underline{t},\underline{t}^{\ast},\alpha
)d\alpha:=\frac{1}{2^{d-1}}\sum_{\underline{s}^{\ast}\in E}\int_{0}%
^{1}f(\underline{s},\underline{s}^{\ast},\alpha)d\alpha,
\]
by embedding $\underline{t},\underline{t}^{\ast}\in\mathbb{T}^{d-1}$ into
$\underline{s}\in\lbrack0,1)^{d-1}$ such that $\underline{s}=\underline{t}$
mod $1$ componentwisely and $E=\{\underline{s}^{\ast}\in\mathbb{R}%
^{d-1};\underline{s}^{\ast}=\underline{t}^{\ast}$ mod $1$ and $|\underline
{s}-\underline{s}^{\ast}|<\sqrt{d-1}\}$, and
\begin{align*}
c_{\delta_{2}}  &  =\int_{0}^{1}\{h(t_{1},\underline{t})-h(t_{1},\underline
{t}^{\ast})\}dt_{1}\\
&  \geq\Vert\bar{h}^{(1)}-\hat{h}^{(1)}\Vert_{L^{1}([0,1])}-2\delta_{2}\geq
C_{8},
\end{align*}
for some $C_{8}>0$, if $\delta_{2}=c\delta_{1}^{1/4}$ and therefore
$\delta_{1}$ are sufficiently small. Estimating $|\underline{t}-\underline
{t}^{\ast}|\leq\sqrt{d-1}$, the left hand side of \eqref{eq:S-5} is bounded
from above by
\begin{align*}
\sqrt{d-1}  &  \int_{0}^{1}dt_{1}\int_{\mathbb{T}^{d-1}}d\underline{t}\int
_{S}d\underline{t}^{\ast}\;\mathrm{Av}\!\!\int_{0}^{1}|\nabla_{\underline{t}%
}h(t_{1},\alpha\underline{t}+(1-\alpha)\underline{t}^{\ast})|d\alpha\\
&  =\sqrt{d-1}\int_{0}^{1}E\left[  1_{S}(\underline{t}^{\ast})|\nabla
_{\underline{t}}h(t_{1},\alpha\underline{t}+(1-\alpha)\underline{t}^{\ast
})|\right]  dt_{1}.
\end{align*}
Here, under the expectation, $\underline{t}$ and $\underline{t}^{\ast}$ are
$\mathbb{T}^{d-1}$-valued uniformly distributed random variables, $\alpha$ is
$[0,1]$-valued uniformly distributed random variable and $\{\underline
{t},\underline{t}^{\ast},\alpha\}$ are mutually independent. Then, by
Schwarz's inequality, we have that
\begin{align*}
&  E\left[  1_{S}(\underline{t}^{\ast})|\nabla_{\underline{t}}h(t_{1}%
,\alpha\underline{t}+(1-\alpha)\underline{t}^{\ast})|\right] \\
&  \quad\leq\sqrt{|S|}\sqrt{E[|\nabla_{\underline{t}}h(t_{1},\alpha
\underline{t}+(1-\alpha)\underline{t}^{\ast})|^{2}]}\\
&  \quad=\delta_{1}^{1/4}\left(  \int_{\mathbb{T}^{d-1}}|\nabla_{\underline
{t}}h(t_{1},\underline{t})|^{2}d\underline{t}\right)  ^{1/2},
\end{align*}
since $\alpha\underline{t}+(1-\alpha)\underline{t}^{\ast}$ is also
$\mathbb{T}^{d-1}$-valued uniformly distributed random variable. Thus,
applying Schwarz's inequality again, the left hand side of \eqref{eq:S-5} is
bounded from above by
\[
\sqrt{d-1}\delta_{1}^{1/4}\Vert\nabla_{\underline{t}}h\Vert_{L^{2}(D)}%
\leq\sqrt{d-1}\delta_{1}^{1/4}\sqrt{2\delta_{1}},
\]
by the condition \eqref{eq:S-10}. Combined with \eqref{eq:S-5}, this implies
$\sqrt{2(d-1)}\delta_{1}^{1/4}\geq\frac{C_{7}}{4}$, which contradicts that we
assume $\delta_{1}$ is sufficiently small. Thus, \eqref{eq:S-6} and
\eqref{eq:S-7} complete the proof of the proposition by taking $c=\max
\{C_{5},C_{6}\}$.
\end{proof}

\subsection{Stability at mesoscopic level}  \label{section:3.2}

Given $0<\beta<1$, we divide $D_{N}$ into $N^{d(1-\beta)}$ subboxes of
sidelength $N^{\beta}$. For the sake of simplicity, we assume that $N^{\beta}%
$divides $N$. We write $\mathcal{B}_{N,\beta}$ for the set of these subboxes,
and $\hat{\mathcal{B}}_{N,\beta}$ for the set of unions of boxes in
$\mathcal{B}_{N,\beta}$. The sets $B\in\hat{\mathcal{B}}_{N,\beta}$ are called
mesoscopic regions.

For $B\in\hat{\mathcal{B}}_{N,\beta}$ (and actually for general
$B\subset D_N$), set
\begin{align}
&  E_{N}(B)=E_{N,0}(B)-\xi^{\varepsilon}|B^{c}|,\nonumber\\
&  E_{N,0}(B)=\inf_{\phi\in{\mathbb{R}}^{D_{N}}:\eqref{eq:S-22}}H_{N}%
(\phi),\label{eq:S-21}\\
&  E_{N}^{\ast}(B)=E_{N}(B)-\min_{B\in\hat{\mathcal{B}}_{N,\beta}}%
E_{N}(B),\nonumber
\end{align}
where the infimum in \eqref{eq:S-21} is taken over all $\phi\in{\mathbb{R}%
}^{D_{N}}$ satisfying the condition:%
\begin{equation}
\phi_{i}=\left\{
\begin{array}
[c]{cc}%
aN & \mathrm{if\ }i\in\partial_{L}D_{N}\\
bN & \mathrm{if\ }i\in\partial_{R}D_{N}\\
0 & \mathrm{if\ }i\in D_N^\circ\backslash B
\end{array}
\right.  . \label{eq:S-22}%
\end{equation}
Let $\bar{\phi}^{B}=(\bar{\phi}_{i}^{B})_{i\in D_{N}}$ be the harmonic
function on $B$ subject to the condition \eqref{eq:S-22}. Then, $\bar{\phi
}^{B}$ is the minimizer of the variational problem \eqref{eq:S-21}. The
macroscopic profile $h^{N}=h_{B}^{N} \big(\equiv h_{B,{\rm PL}}^N\big)
\in C(D)$ is defined from the microscopic
profile $\bar{\phi}^{B}$ by polilinearly interpolating $\frac{1}{N}%
\bar\phi_{\lbrack Nt]}^B, t\in D$, where $[Nt]$ stands for the integer part of $Nt$
taken componentwisely; see \eqref{eq:Ma-H-2}.

The stability at mesoscopic level is formulated as follows:

\begin{proposition}
\label{prop:mesostability} Assume $\alpha>0$ is given and $\beta
,\gamma>4\alpha$. Then, if $N$ is sufficiently large, $E_{N}^{\ast}(B)\leq
N^{d-\gamma}$ for $B\in\hat{\mathcal{B}}_{N,\beta}$ implies $d_{L^{1}}%
(h_{B}^{N},\{\bar{h},\hat{h}\})\leq N^{-\alpha}$.
\end{proposition}

From (1.22) in \cite{DGI}, the polilinear interpolation has the property:
\[
\frac12\int_{D} |\nabla h^{N}(t)|^{2} dt \le\frac1{2N^{d}} \sum_{i\in D_{N}}
|\nabla^{N} \bar\phi_{i}^{B}|^{2} = \frac1{N^{d}} H_{N}(\bar\phi^{B})
=\frac1{N^{d}} E_{N,0}(B).
\]
We also see that $\{t\in D; h^{N}(t)=0\} \supset\frac1{N} (B^{c})^{\circ}$,
which implies that
\[
-|\{t\in D; h^{N}(t)=0\}| + \frac1{N^{d}} |B^{c}| \le\frac1{N^{d}} |\partial
B| \le\frac1{N^{d}} d N^{d-\beta} = d N^{-\beta}.
\]
These two bounds show that
\begin{equation}
\label{eq:S-23}\Sigma(h^{N}) \le\frac1{N^{d}}E_{N}(B) + \xi^{\varepsilon}d
N^{-\beta}.
\end{equation}

We need the next lemma.

\begin{lemma}
\label{lem:mesostability-min}
\[
\frac{1}{N^{d}}\min E_{N}(B)\leq\min\Sigma\leq\frac{1}{N^{d}}\min E_{N}%
(B)+\xi^{\varepsilon}dN^{-\beta}.
\]

\end{lemma}

\begin{proof}
The upper bound follows from \eqref{eq:S-23}. To show the lower, recall
$\min\Sigma=\frac{1}{2}(a-b)^{2}$. Define $\bar{\phi}\equiv\bar{\phi}^{D_{N}%
}=(\bar{\phi}_{i})_{i\in D_{N}}$ by
\[
\bar{\phi}=\psi_{i_{1}}:=aN+(b-a)i_{1},\quad i\in D_{N},
\]
where $i_{1}$ is the first component of $i$. Then, we see that
\[
E_{N}(D_{N})=H_{N}(\bar{\phi})=\frac{1}{2}\sum_{(i_{2},\ldots,i_{d}%
)\in\mathbb{T}_{N}^{d-1}}\sum_{i_{1}=0}^{N-1}(\psi_{i_{1}+1}-\psi_{i_{1}}%
)^{2}=\frac{N^{d}}{2}(b-a)^{2}.
\]
This proves the lower bound.
\end{proof}

From the lower bound in this lemma and \eqref{eq:S-23}, we see that $E^{*}(B)
\le N^{d-\gamma}$ implies $\Sigma^{*}(h^{N}) \le N^{-\gamma}+ \xi
^{\varepsilon}dN^{-\beta}$. Thus, Proposition \ref{prop:mesostability} follows
from Proposition \ref{prop:macrostability}.

We slightly extend Proposition \ref{prop:mesostability} and this will be used in Section 
\ref{section:6.3}.

\begin{proposition} \label{prop:stability-meso}
Let a mesoscopic region $B$ and $A_{2}\subset B$ such that $|B\setminus
A_{2}|\leq N^{d-\frac{1}{8}}$ be given, and assume that%
\begin{equation} \label{eq:S-24}
E_{N,0}\left(  A_{2}\right)  -\xi^{\varepsilon}\left\vert B^{c}\right\vert
-\min E_{N}\leq N^{d-\gamma}.
\end{equation}
Then, we have that%
\begin{equation} \label{eq:S-25}
d_{L^{1}}\left(  h_{A_{2}}^{N},\left\{  \bar{h},\hat{h}\right\}  \right)  \leq
N^{-\alpha},
\end{equation}
where $h_{A_{2}}^{N}$ is defined from $\bar{\phi}^{A_{2}}$, which is harmonic
on $A_{2}$ subject to the condition \eqref{eq:S-22} with $B$ replaced by
$A_{2}$.
\end{proposition}

\begin{proof}
As we saw above, we have that%
\[
\frac{1}{2}\int_{D}|\nabla h_{A_{2}}^{N}(t)|^{2}dt\leq\frac{1}{N^{d}}%
E_{N,0}(A_{2})
\]
and also, since $\{t\in D;h_{A_{2}}^{N}(t)=0\}\supset\frac{1}{N}(B^{c}%
)^{\circ}$ (we don't need the condition on $|B\setminus A_{2}|$),
\[
-|\{t\in D;h_{A_{2}}^{N}(t)=0\}|+\frac{1}{N^{d}}|B^{c}|\leq dN^{-\beta}.
\]
Therefore, \eqref{eq:S-24} together with the lower bound in Lemma
\ref{lem:mesostability-min} implies $\Sigma^{\ast}(h_{A_{2}}^{N})\leq
N^{-\gamma}+\xi^{\varepsilon}dN^{-\beta}$, and we obtain \eqref{eq:S-25} from
Proposition \ref{prop:macrostability}.
\end{proof}

\section{Proof of the lower bound \eqref{eq:1.9}}

This section is concerned with the lower bound on
\begin{equation}
\Xi_{N}:=\frac{Z_{N}^{aN,bN,\varepsilon}}{Z_{N}^{aN,bN}}\mu_{N}%
^{aN,bN,\varepsilon}(\Vert h^{N}-\hat{h}\Vert_{L^{p}(D)}\leq\delta),
\label{eq:3.1-a}%
\end{equation}
where we take $\delta=N^{-\alpha}$ with $\alpha<1$; see Remark \ref{rem:3.1}
below. We divide $D_{N}^{\circ}$ into five disjoint regions: $D_{N}^{\circ
}=A_{L}\cup\gamma_{L}\cup B\cup\gamma_{R}\cup A_{R}$, where
\begin{align*}
A_{L}  &  =\big(\lbrack1,Ns_{1}^{L}-K-1]\cap{\mathbb{Z}}\big)\times
{\mathbb{T}}_{N}^{d-1},\\
\gamma_{L}  &  =\big(\lbrack Ns_{1}^{L}-K,Ns_{1}^{L}]\cap{\mathbb{Z}%
}\big)\times{\mathbb{T}}_{N}^{d-1},\\
B  &  =\big(\lbrack Ns_{1}^{L}+1,Ns_{1}^{R}]\cap{\mathbb{Z}}\big)\times
{\mathbb{T}}_{N}^{d-1},\\
\gamma_{R}  &  =\big(\lbrack Ns_{1}^{R}+1,Ns_{1}^{R}+K]\cap{\mathbb{Z}%
}\big)\times{\mathbb{T}}_{N}^{d-1},\\
A_{R}  &  =\big(\lbrack Ns_{1}^{R}+K+1,N-1]\cap{\mathbb{Z}}\big)\times
{\mathbb{T}}_{N}^{d-1},
\end{align*}
for $K>0$, where $s_{1}^{L}$ and $s_{1}^{R}\in(0,1)$ are the first and the
last $s$'s such that $\hat{h}^{(1)}(s)=0$ and we assume that $Ns_{1}%
^{L},Ns_{1}^{R}\in{\mathbb{Z}}$ for simplicity. Note that the side lengths in
$i_{1}$-direction of these five rectangles are $Ns_{1}^{L}-K-1,K+1,N(s_{1}%
^{R}-s_{1}^{L}),K$ and $N(1-s_{1}^{R})-K-1$ for $A_{L},\gamma_{L},B,\gamma
_{R}$ and $A_{R}$, respectively. Then, restricting the probability in
\eqref{eq:3.1-a} on the event:
\[
\mathcal{A}:=\{\phi;\phi_{i}\not =0\text{ for }i\in A_{L}\cup A_{R}\text{ and
}\phi_{i}=0\text{ for }i\in\gamma_{L}\cup\gamma_{R}\},
\]
we have
\begin{align}
\Xi_{N}\geq &  \frac{Z_{N}^{aN,bN,\varepsilon}}{Z_{N}^{aN,bN}}\mu
_{N}^{aN,bN,\varepsilon}(\Vert h^{N}-\hat{h}\Vert_{L^{p}(D)}\leq
\delta,\mathcal{A})\label{eq:3.2-a}\\
=  &  \Xi_{N}^{1}\times\mu_{A_{L}}^{aN,0}(\Vert h^{N}-\hat{h}\Vert
_{L^{p}(D_{L})}\leq\delta)\nonumber\\
&  \qquad\times\mu_{B}^{0,\varepsilon}(\Vert h^{N}-\hat{h}\Vert_{L^{p}(D_{M}%
)}\leq\delta)\mu_{A_{R}}^{0,bN}(\Vert h^{N}-\hat{h}\Vert_{L^{p}(D_{R})}%
\leq\delta),\nonumber
\end{align}
by the Markov property of $\mu_{N}^{aN,bN,\varepsilon}$, where
$\mu_{A_L}^{aN,0}$ is defined on $A_L$ with boundary conditions
$aN$ and $0$ at the left respectively right boundaries of $A_L$
without pinning, $\mu_{A_R}^{0,bN}$ is similarly defined on $A_R$,
$\mu_B^{0,\varepsilon}$ is defined on $B$ with boundary condition
$0$ with pinning, 
\[
\Xi_{N}^{1}=\frac{Z_{A_{L}}^{aN,0}Z_{B}^{0,\varepsilon}Z_{A_{R}}^{0,bN}}%
{Z_{N}^{aN,bN}}\varepsilon^{|\gamma_{L}|+|\gamma_{R}|},
\]
and $D_{L},D_{M}$ and $D_{R}$ are the macroscopic regions corresponding to
$A_{L}$, $B$ and $A_{R}$, respectively. Since $\gamma_{L}$ and $\gamma_{R}$
are macroscopically close to the hyperplanes $\{t_{1}=s_{1}^{L}\}$ and
$\{t_{1}=s_{1}^{R}\}$ in $D$, respectively (i.e., $\gamma_{L}/N$ is in a
$c\delta$-neighborhood of $\{t_{1}=s_{1}^{L}\}$ with suitable $c>0$ etc.), by
the LDP \cite{BD} for $\mu_{A_{L}}^{aN,0}$, $\mu_{A_{R}}^{0,bN}$ and 
the LDP for $\mu_{B}^{0,0}$
combined with the coupling
argument (see Lemma \ref{lem:3.2} below) implying
$-\tilde{\phi}_{i}^{(2)}\leq\phi
_{i}^{(1)}\leq\phi_{i}^{(2)},i\in B$ for $\phi^{(1)}\sim\mu_{B}^{0,\varepsilon
}$, $\tilde{\phi}^{(2)},\phi^{(2)}\sim\mu_{B}^{0,0,+}:=\mu_{B}^{0,0}%
(\cdot|\phi\geq0)$, three
probabilities in the right hand side of \eqref{eq:3.2-a} are close to $1$ as
$N\rightarrow\infty$. Therefore, for every $c>0$, we have
\begin{equation}
\Xi_{N}\geq(1-c)\Xi_{N}^{1} \label{eq:3.3-a}%
\end{equation}
as $N\rightarrow\infty$.

\begin{remark}\label{rem:3.1}
{\rm(1)} If $d\geq3$, the Gaussian property implies
\[
E^{\mu_{A_{L}}^{aN,0}}\left[  \Vert h^{N}-\hat{h}\Vert_{L^{2}(D_{L})}%
^{2}\right]  \leq\frac{C}{N^{2}},
\]
and others. Therefore, \eqref{eq:3.3-a} holds even for $\delta=N^{-\alpha}$
with $\alpha<1$ at least for $p=2$ (so that for every $1\leq p\leq2$). For
$d=2$, this statement is also true since the above expectation behaves as
$C\log N/N^{2}$.\\
{\rm (2)} To show the weaker estimate \eqref{eq:1.12}, we can simply estimate
$\Xi_{N}\geq\Xi_{N}^{1}$ so that the LDP and the coupling argument for the
above three probabilities are unnecessary.
\end{remark}

We now give the lower bound on $\Xi_{N}^{1}$. Since $A_{L}= E_{Ns_{1}^{L}%
-K-1}$ and $A_{R} = E_{N(1-s_{1}^{R})-K-1}$ (which is reversed), Lemma
\ref{lem:1} shows that
\begin{align*}
&  Z_{N}^{aN,bN} = \exp\left\{  -\frac{N^{d}}2(a-b)^{2}\right\}  Z_{N}%
^{0,0},\\
&  Z_{A_{L}}^{aN,0} = \exp\left\{  -\frac{a^{2}N^{d}}{2(s_{1}^{L}%
-K/N)}\right\}  Z_{A_{L}}^{0},\\
&  Z_{A_{R}}^{0,bN} = \exp\left\{  -\frac{b^{2}N^{d}}{2(1-s_{1}^{R}%
-K/N)}\right\}  Z_{A_{R}}^{0}.
\end{align*}
Therefore, from $1/(s_{1}^{L}-K/N) = 1/s_{1}^{L}+ KN^{-1}/(s_{1}^{L})^{2} +
O_{\varepsilon}(N^{-2})$ and a similar expansion for $1/(1-s_{1}^{R}- K/N)$ as
$N\to\infty$, we have
\[
\Xi_{N}^{1} \ge\exp\left\{  f(a,b)N^{d} - K\tilde{f}(a,b)N^{d-1}
-O_{\varepsilon}(N^{d-2}) \right\}  \Xi_{N}^{2},
\]
where $O_{\varepsilon}(N^{d-2})$ means that the constant may depend on
$\varepsilon$ (since $s_{1}^{L}$ and $s_{1}^{R}$ depend on $\varepsilon$),
and
\begin{align*}
\Xi_{N}^{2}  &  = \frac{Z_{A_{L}}^{0}Z_{B}^{0,\varepsilon}Z_{A_{R}}^{0}}%
{Z_{N}^{0,0}} \varepsilon^{|\gamma_{L}|+|\gamma_{R}|},\\
f(a,b)  &  = \frac12(a-b)^{2} - \frac{a^{2}}{2s_{1}^{L}} - \frac{b^{2}%
}{2(1-s_{1}^{R})}\\
&  = \Sigma(\bar{h}) - \Sigma(\hat{h}) - \xi^{\varepsilon}(s_{1}^{R}-s_{1}%
^{L}),\\
\tilde{f}(a,b)  &  = \frac{a^{2}}{2(s_{1}^{L})^{2}} + \frac{b^{2}}%
{2(1-s_{1}^{R})^{2}}.
\end{align*}
However, we have $\tilde{f}(a,b) = 2\xi^{\varepsilon}$ from Young's relation
for the angles of $\hat{h}$ at $s=s_{1}^{L}$ and $s_{1}^{R}$: $a/s_{1}^{L} =
b/(1-s_{1}^{R}) = \sqrt{2\xi^{\varepsilon}}$, see Section 1.3 of \cite{BFO} or
Section 6 of \cite{Fu05} for example. Moreover, by Lemma \ref{lem:2} and
Remark \ref{rem:3.3}-(2)
\begin{align*}
\frac{Z_{A_{L}}^{0}Z_{A_{R}}^{0}}{Z_{N}^{0,0}}  &  \ge\exp\left\{  \hat{q}^{0}
(|A_{L}|+|A_{R}|-|D_{N}^{\circ}|) -4rN^{d-1} - C \right\}  ,
\end{align*}
and by the lower bound in Lemma \ref{lem:4}
\[
Z_{B}^{0,\varepsilon} \ge\exp\left\{  \hat{q}^{\varepsilon}|B|-\frac32 \bar c
N^{d-1} \right\}  .
\]
Thus, since $|A_{L}|+|A_{R}|+|B|+|\gamma_{L}|+|\gamma_{R}| = |D_{N}^{\circ}|
\big(=N^{d-1}(N-1) \big)$ and $\hat{q}^{\varepsilon}- \hat{q}^{0}=
\xi^{\varepsilon}$, we obtain
\begin{align*}
\log\Xi_{N}^{1} \ge &  f(a,b) N^{d} + \hat{q}^{0}(|A_{L}|+|A_{R}%
|-|D_{N}^{\circ}|) + \hat{q}^{\varepsilon}|B|\\
&  - \big( 4r+3\bar c/2 +2K\xi^{\varepsilon}\big) N^{d-1} + (|\gamma_{L}%
|+|\gamma_{R}|) \log\varepsilon- O_{\varepsilon}(N^{d-2}) -C\\
\ge &  f(a,b) N^{d} + \xi^{\varepsilon}|B| - (C_{1}+2K\xi^{\varepsilon})
N^{d-1} + (|\gamma_{L}|+|\gamma_{R}|) (\log\varepsilon-\hat{q}^{0}) -
O_{\varepsilon}(N^{d-2}),
\end{align*}
with a constant $C_{1}=4r+3\bar c/2>0$ independent of $\varepsilon$; the constant
$C $ is included in $O_{\varepsilon}(N^{d-2})$. However, the balance
condition: $\Sigma(\bar{h}) = \Sigma(\hat{h})$ and $|B| = N^{d} (s_{1}%
^{R}-s_{1}^{L})$ imply that $f(a,b) N^{d} + \xi^{\varepsilon}|B|=0$, so that
the volume order terms cancel. Therefore, from $|\gamma_{L}|+|\gamma_{R}| =
(2K+1) N^{d-1}$, we have
\begin{align*}
\log\Xi_{N}^{1}  &  \ge\big((2K+1) (\log\varepsilon-\hat{q}^{0})
-2K\xi^{\varepsilon}-C_{1} \big) N^{d-1} - O_{\varepsilon}(N^{d-2})\\
&  \ge\big(\log\varepsilon- (2K+1)\hat{q}^{0} -2K\log2 - C_{1}\big) N^{d-1} -
O_{\varepsilon}(N^{d-2}),
\end{align*}
where the second line follows from the upper bound on $\xi^{\varepsilon}$
given in Lemma \ref{lem:3.1} below. It is now clear that, for $\varepsilon>0$
large enough, the coefficient of $N^{d-1}$ in the right hand side is positive
and thus the proof of the lower bound \eqref{eq:1.9} is concluded.

\begin{lemma}
\label{lem:3.1} For $\varepsilon\ge1$, we have that
\[
\log\varepsilon- \hat{q}^{0} \le\xi^{\varepsilon}\le\log2\varepsilon.
\]

\end{lemma}

\begin{proof}
We have an expansion:
\[
Z_{\Lambda_{\ell}}^{0,\varepsilon}=\sum_{A\subset\Lambda_{\ell}}%
\varepsilon^{|\Lambda_{\ell}\setminus A|}Z_{A}^{0}.
\]
To show the upper bound, we rudely estimate: $\varepsilon^{|\Lambda_{\ell
}\setminus A|}\leq\varepsilon^{\ell^{d}}$ for $\varepsilon\geq1$ and
$Z_{A}^{0}\leq e^{\hat{q}^{0}|A|}\leq e^{\hat{q}^{0}\ell^{d}}$ by Lemma
\ref{lem:2}-(1); note that its upper bound holds with $q$ in place of $q^{N}$
for $A\Subset\mathbb{Z}^{d}$. Then, since $\sharp\{A:A\subset\Lambda_{\ell
}\}=2^{\ell^{d}}$, we obtain
\[
Z_{\Lambda_{\ell}}^{0,\varepsilon}\leq2^{\ell^{d}}\varepsilon^{\ell^{d}%
}e^{\hat{q}^{0}\ell^{d}}%
\]
and therefore
\[
\hat{q}^{\varepsilon}\equiv\lim_{\ell\rightarrow\infty}\frac{1}{\ell^{d}}\log
Z_{\Lambda_{\ell}}^{0,\varepsilon}\leq\log2\varepsilon+\hat{q}^{0},
\]
from which the upper bound on $\xi^{\varepsilon}=\hat{q}^{\varepsilon}-\hat
{q}^{0}$ follows (or, recall \eqref{eq:1.5} for $\xi^{\varepsilon}$ and note
that Lemma \ref{lem:2} also shows $\lim_{\ell\rightarrow\infty}\ell^{-d}\log
Z_{\Lambda_{\ell}}^{0}=\hat{q}^{0}$). Taking only the term with $A=\emptyset$
in the expansion, we have $Z_{\Lambda_{\ell}}^{0,\varepsilon}\geq
\varepsilon^{\ell^{d}}$ and this implies the lower bound.
\end{proof}

\begin{remark}
{\rm (1)} To have the large factor $\log\varepsilon$, we need to allow some
spaces for $\gamma_{L}$ and $\gamma_{R}$. For this purpose, in the above
proof, we have cut off the regions $A_{L}$ and $A_{R}$ by letting $K\geq1$,
while the volume of the region $B$ are maintained. It is also possible to
maintain the spaces for $A_{L}$ and $A_{R}$ by taking $K=0$. Instead, we may
cut off the region $B$, but the results are the same. \\
{\rm (2)} In fact, one can take $K=0$ for $\gamma_{L}$ and $K=1$ for
$\gamma_{L}$ so that the required condition for $\varepsilon>0$ is:
$\log\varepsilon>\log2+2\hat{q}^{0}+4r+3\bar c/2.$
\end{remark}

We finally give the coupling argument used above.
Consider the Gibbs probability measure $\mu_{A}^{\psi,\varepsilon}$ on $A$
under the boundary condition $\psi$ given on $D_{N}\setminus A$. Assuming that
$\psi\ge0$ (i.e., $\psi_{i}\ge0$ for all $i\in D_{N}\setminus A$), we compare
it with
\[
\mu_{A}^{\psi,0,+}(\cdot) := \mu_{A}^{\psi,0}(\cdot|\phi\ge0),
\]
by adding the effect of a wall located at the level of $\phi=0$ to the
Gaussian measure $\mu_{A}^{\psi,0}(\cdot)$ without the pinning effect. In
fact, we have the following lemma from an FKG type argument.

\begin{lemma}
\label{lem:3.2} We have the stochastic domination: $\mu_{A}^{\psi,\varepsilon
}\leq\mu_{A}^{\psi,0,+}.$ Namely, one can find a coupling of $\phi
^{\varepsilon}=\{\phi_{i}^{\varepsilon}\}_{i\in D_{N}}$ and $\phi^{0,+}%
=\{\phi_{i}^{0,+}\}_{i\in D_{N}}$ on a common probability space such that
$P(\phi_{i}^{\varepsilon}\leq\phi_{i}^{0,+}\text{ for all }i\in D_{N})=1$, and
$\phi^{\varepsilon}$ and $\phi^{0,+}$ are distributed under $\mu_{A}%
^{\psi,\varepsilon}$ and $\mu_{A}^{\psi,0,+}$, respectively.
\end{lemma}

\begin{proof}
For $\phi=\left(  \phi_{i}\right)  _{i\in D_{N}}\in\mathbb{R}^{D_{N}}$
satisfying the conditions $\phi_{k}=\psi_{k}$ on $D_{N}\backslash A$, we
consider two Hamiltonians%
\[
H_{N}^{\left(  \ell\right)  }\left(  \phi\right)  =H_{N}^{\psi}\left(
\phi\right)  +\sum_{i\in D_{N}\backslash A}U^{\left(  \ell\right)  }\left(
\phi_{i}\right)  ,\ \ell=1,2,
\]
by adding the self potentials $U^{\left(  \ell\right)  }$ defined by
$U^{\left(  1\right)  }\left(  r\right)  =-\beta1_{\left[  0,\alpha\right]
}\left(  r\right)  $ and $U^{\left(  2\right)  }\left(  r\right)
=K1_{(-\infty,0]}\left(  r\right),$ $r>0,$ with $\alpha,\beta,K>0$ to the
original Hamiltonian $H_{N}^{\psi}$ defined under the boundary condition
$\psi.$ The corresponding Gibbs probability measures $\mu_{N}^{\left(
\ell\right)  } $ are defined \ by%
\[
\mu_{N}^{\left( \ell\right)  }\left(  d\phi\right)  =\frac{1}{Z_{N}^{\left(
\ell\right)  }}\mathrm{e}^{-H_{N}^{\left(  \ell\right)  }\left(  \phi\right)
}\prod_{i\in A}d\phi_{i}\prod_{k\in D_{N}\backslash A}\delta_{\psi_{k}}\left(
d\phi_{k}\right)  ,\ \ell=1,2.
\]
It will be shown that the stochastic domination $\mu_{N}^{\left(  1\right)
}\leq\mu_{N}^{\left(  2\right)  }$ holds if $K\geq\beta.$ Once this is shown,
by taking the limits $\alpha\rightarrow0,$ $\beta\rightarrow\infty$ such that
$\varepsilon=\alpha\left(  \mathrm{e}^{\beta}-1\right)  $ (see e.g. (6.34) in
\cite{Fu05}, and $K\rightarrow\infty,$ the lemma is concluded.

It is known that the stochastic domination $\mu_{N}^{\left(  1\right)  }%
\leq\mu_{N}^{\left(  2\right)  }$ holds if the two Hamiltonians satisfy
Holley's condition:%
\begin{equation}
H_{N}^{\left(  2\right)  }\left(  \phi\right)  +H_{N}^{\left(  1\right)
}\left(  \bar{\phi}\right)  \geq H_{N}^{\left(  2\right)  }\left(  \phi
\vee\bar{\phi}\right)  +H_{N}^{\left(  1\right)  }\left(  \phi\wedge\bar{\phi
}\right)  , \label{Holley}%
\end{equation}
for every $\phi,\bar{\phi}\in\mathbb{R}^{D_{N}},$ where $\left(  \phi\vee
\bar{\phi}\right)  _{i}=\phi_{i}\vee\bar{\phi}_{i}$ and $\left(  \phi
\wedge\bar{\phi}\right)  _{i}=\phi_{i}\wedge\bar{\phi}_{i}$, see Theorem 2.2
of \cite{FTo}. Since (\ref{Holley}) holds for $H_{N}^{\psi}$ (i.e., if
$U^{\left(  1\right)  }=U^{\left(  2\right)  }=0$), it is enough to show that%
\[
U^{\left(  2\right)  }\left(  x\right)  +U^{\left(  1\right)  }\left(
y\right)  \geq U^{\left(  2\right)  }\left(  x\vee y\right)  +U^{\left(
1\right)  }\left(  x\wedge y\right)
\]
for all $x,y\in\mathbb{R}$. However, this is equivalent to%
\begin{equation} \label{eq:5.U}
U^{\left(  2\right)  }\left(  x\right)  -U^{\left(  2\right)  }\left(
y\right)  \geq U^{\left(  1\right)  }\left(  x\right)  -U^{\left(  1\right)
}\left(  y\right)
\end{equation}
for every $x,y\in\mathbb{R}$. It is now easy to see that this is true under
the condition $K\geq\beta$.
\end{proof}

\begin{remark}
If the self potentials $U^{(\ell)}$ are smooth, the condition \eqref{eq:5.U}
is equivalent to $\{U^{(2)}\}^{\prime} \le \{U^{(1)}\}^{\prime}$ on ${\mathbb{R}}$.
\end{remark}

\begin{remark}
Lemma \ref{lem:3.2} was applied under the boundary condition $\psi\equiv0$. In
this case, by the symmetry $\phi\mapsto-\phi$ under $\mu_{A}^{0,\varepsilon}$,
we also have the lower bound. When $\psi\equiv0$, it might hold the stochastic
domination: $|\phi_{i}^{\varepsilon}|\leq|\phi_{i}|,i\in D_{N},$
$\phi^{\varepsilon}\sim\mu_{A}^{0,\varepsilon}$, $\phi\sim\mu_{A}^{0,0}$ (this
was true at least when $d=1$, see Section 4.2.3 \cite{BFO}).
\end{remark}

\section{Proof of the upper bound \eqref{eq:1.10}}

We write $Z_{N}^{\varepsilon},Z_{N},\mu_{N}^{\varepsilon}$ instead of
$Z_{N}^{aN,bN,\varepsilon},Z_{N}^{aN,bN},\mu_{N}^{aN,bN,\varepsilon}$,
respectively, and similar at other places. We expand as
\begin{equation}  \label{eq:5.1-c}
\frac{Z_{N}^{\varepsilon}}{Z_{N}}\mu_{N}^{\varepsilon}\left(  \left\Vert
h^{N}-\overline{h}\right\Vert _{L^p(D)}\leq\delta\right)  =\sum_{A\subset
D_{N}^{\circ}}\varepsilon^{\left\vert A^{c}\right\vert }\frac{Z_{A}}{Z_{N}}%
\mu_{A}\left(  \left\Vert h^{N}-\overline{h}\right\Vert_{L^p(D)}\leq\delta
\right).
\end{equation}
Here, $Z_{A}$ refers to boundary conditions $0$ on $A^{c},$ and the usual one
on the cylinder ($A^{c}$ stands for the complement of $A$ in $D_{N}^{\circ}$),
and $\mu_{A}$ is defined with similar boundary conditions. We will consider
the Gaussian field $\mu_{N}$ on $D_{N}^{\circ} $ with the above boundary
conditions.  Note that the Gaussian field $\left\{  \phi_{i}\right\}  _{i\in
D_{N}^{\circ}}$ on $\mathbb{R}^{D_{N}^{\circ}}$ distributed under $\mu_{N}$
has covariance matrix%
\[
\Gamma\overset{\mathrm{def}}{=}\frac{1}{2d}\left(  I-P\right)  ^{-1},
\]
where $P$ is the random walk transition kernel with killing at the boundary
$\partial D_{N}.$ Furthermore, $\phi_{i}$ has mean $m\left(  i\right)
=m_{aN,bN}\left(  i\right)  $ which is given by linearly interpolating between
the boundary condition $aN$ on $\partial_{L}D_{N}$ and $bN$ on $\partial
_{R}D_{N}.$

We take $\delta=(\log N)^{-\alpha_{0}}$ with $\alpha_{0}>d/p$
in \eqref{eq:5.1-c}.  We show that%
\begin{equation}
\sum_{A\subset D_{N}^{\circ},\ \left\vert A^{c}\right\vert \leq\left(  N/\log
N\right)  ^{d}}\varepsilon^{\left\vert A^{c}\right\vert }\frac{Z_{A}}{Z_{N}%
}\leq2\label{eq:4.1}%
\end{equation}
if $N$ is large enough. Note that, if $\left\vert A^{c}\right\vert \geq\left(
N/\log N\right)  ^{d}$, then $h^{N}=0$ on $A^{c}$ so that
\[
\left\Vert h^{N}-\overline{h}\right\Vert_{L^p(D)}
\geq(a\wedge b)(\log N)^{-d/p}.
\]
In particular, for such $A$, we have
\[
\mu_{A}\left(  \left\Vert h^{N}-\overline{h}\right\Vert_{L^p(D)}\leq(\log
N)^{-\alpha_{0}}\right)  =0
\]
as $\alpha_{0}>d/p$. Thus \eqref{eq:4.1} proves \eqref{eq:1.10}.

Now we give the proof of \eqref{eq:4.1}. Recall that
\[
\frac{Z_{A}}{Z_{N}}=\frac{1}{Z_{N}}\int_{\mathbb{R}^{D_{N}^{\circ}}}%
\exp\left[  -H_{N}\left(  \phi\right)  \right]  \prod_{i\in A}d\phi_{i}%
\prod_{i\in A^{c}}\delta_{0}\left(  d\phi_{i}\right)  .
\]
The function%
\[
f_{A^{c}}\left(  \left\{  \phi_{i}\right\}  _{i\in A^{c}}\right)
\overset{\mathrm{def}}{=}\frac{1}{Z_{N}}\int\exp\left[  -H_{N}\left(
\phi\right)  \right]  \prod_{i\in A}d\phi_{i}%
\]
is the density function of the Gaussian distribution on $\mathbb{R}^{A^{c}}$
obtained as the marginal from the Gaussian distribution $\mu_{N}$ on
$\mathbb{R}^{D_{N}^{\circ}}.$ This marginal Gaussian field has the same mean
as $\mu_{N}$ and the covariance matrix $\Gamma_{A^{c}}$ which comes from
restricting the covariance matrix $\Gamma$ to $A^{c}\times A^{c}.$ This
covariance matrix has the representation $\Gamma_{A^{c}}=\left(  I-P_{A^{c}%
}\right)  ^{-1},$ where $P_{A^{c}}\left(  i,j\right)  $ for $i,j\in A^{c}$ is
the probability for a random walk to enter $A^{c}$ at $j$ after leaving $i$
with absorption at $\partial D_{N}.$ So%
\[
\sum_{j\in A^{c}}P_{A^{c}}\left(  i,j\right)  \leq1.
\]
We also write for the escape probability%
\[
e_{A^{c}}\left(  i\right)  \overset{\mathrm{def}}{=}1-\sum_{y\in A^{c}%
}P_{A^{c}}\left(  i,j\right)  ,
\]
and then the capacity of $A^{c}$ with respect to the transient random walk on
$D_{N}^{\circ}$ with killing at the boundary is%
\[
\mathrm{cap}_{D_{N}}\left(  A^{c}\right)  \overset{\mathrm{def}}{=}\sum_{i\in
A^{c}}e_{A^{c}}\left(  j\right)  .
\]

Then we have%
\[
f_{A^{c}}\left(  \left\{  \phi_{i}\right\}  _{i\in A^{c}}\right)  =\frac
{1}{\sqrt{\left(  2\pi\right)  ^{\left\vert A^{c}\right\vert }\det
\Gamma_{A^{c}}}}\exp\left[  -d\left\langle \phi-m,\left(  I-P_{A^{c}}\right)
\left(  \phi-m\right)  \right\rangle _{A^{c}}\right]  ,
\]
where $\left\langle \phi,\psi\right\rangle _{A^{c}}\overset{\mathrm{def}}%
{=}\sum_{i\in A^{c}}\phi_{i}\psi_{i},$ and $m=m_{aN,bN}.$ We therefore get%
\begin{equation}
\frac{Z_{A}}{Z_{N}}=\frac{1}{\sqrt{\left(  2\pi\right)  ^{\left\vert
A^{c}\right\vert }\det\Gamma_{A^{c}}}}\exp\left[  -d\left\langle m,\left(
I-P_{A^{c}}\right)  m\right\rangle _{A^{c}}\right]  .\label{eq:Z/Z}%
\end{equation}

We first estimate the determinant from below%
\begin{align*}
\sqrt{\left(  2\pi\right)  ^{\left\vert A^{c}\right\vert }\det\Gamma_{A^{c}}}
&  =\int\exp\left[  -d\left\langle \phi-m,\left(  I-P_{A^{c}}\right)  \left(
\phi-m\right)  \right\rangle _{A^{c}}\right]  \prod_{i\in A^{c}}d\phi_{i}\\
&  \geq\int\limits_{\left\{  \left\vert \phi_{i}-m_{i}\right\vert
\leq1/2,\ \forall i\in A^{c}\right\}  }\exp\left[  -d\left\langle
\phi-m,\left(  I-P_{A^{c}}\right)  \left(  \phi-m\right)  \right\rangle
_{A^{c}}\right]  \prod_{i\in A^{c}}d\phi_{i}.
\end{align*}
On the other hand,
\[
\sup_{\left\{  \left\vert \phi_{i}-m_{i}\right\vert \leq1/2,\ \forall i\in
A^{c}\right\}  }\left\langle \phi-m,\left(  I-P_{A^{c}}\right)  \left(
\phi-m\right)  \right\rangle _{A^{c}}\leq\left\vert A^{c}\right\vert ,
\]
and therefore%
\[
\sqrt{\left(  2\pi\right)  ^{\left\vert A^{c}\right\vert }\det\Gamma_{A^{c}}%
}\geq\exp\left[  -d\left\vert A^{c}\right\vert \right]  .
\]

We write $p_{L}\left(  i\right)  $ for the probability that the random walk
starting in $i\in A^{c}$ does not return to $A^{c}$ and leaves $D_{N}$ on the
left side, and correspondingly $p_{R}\left(  i\right)  $ for the right exit.
Clearly $p_{L}\left(  i\right)  +p_{R}\left(  i\right)  =e_{A^{c}}\left(
i\right)  .$ Then%
\[
m\left(  i\right)  =\sum_{j}P_{A^{c}}\left(  i,j\right)  m\left(  j\right)
+p_{L}\left(  i\right)  aN+p_{R}\left(  i\right)  bN.
\]
So%
\[
m\left(  i\right)  -\sum_{j}P_{A^{c}}\left(  i,j\right)  m\left(  j\right)
\geq\min\left(  a,b\right)  Ne_{A^{c}}\left(  i\right)  .
\]
Of course, also $m\left(  i\right)  \geq\min\left(  a,b\right)  N.$ Therefore
from \eqref{eq:Z/Z},
\[
\frac{Z_{A}}{Z_{N}}\leq\exp\left[  d\left\vert A^{c}\right\vert \right]
\exp\left[  -dN^{2}\min\left(  a,b\right)  ^{2}\mathrm{cap}_{D_{N}}\left(
A^{c}\right)  \right]  .
\]
Lemma \ref{Le_Capacity_bound} proved below implies that%
\begin{equation}
\mathrm{cap}_{D_{N}}\left(  A^{c}\right)  \geq c\left\vert A^{c}\right\vert
^{\left(  d-2\right)  /d},\label{eq:capacity-lower}%
\end{equation}
from which we conclude that for some $c>0$, depending on $d,a,b$%
\begin{align*}
&  \sum_{A\subset D_{N}^{\circ},\ \left\vert A^{c}\right\vert \leq\left(
N/\log N\right)  ^{d}}\varepsilon^{\left\vert A^{c}\right\vert }\frac{Z_{A}%
}{Z_{N}}\\
&  \leq\sum_{m=0}^{\left(  N/\log N\right)  ^{d}}\chi\left(  m\right)
\varepsilon^{m}\exp\left[dm  -\bar cN^{2}m^{\left(  d-2\right)  /d}\right]  ,
\end{align*}
where $\bar c>0$ and 
$\chi\left(  m\right)  $ is the number of subset $A$ in $D_{N}^{\circ}$
with $|A^{c}|=m.$ Clearly,
\[
\chi\left(  m\right)  \leq\exp\left[  dm\log N\right]  .
\]
So%
\begin{align}
&  \sum_{A\subset D_{N}^{\circ},\ \left\vert A^{c}\right\vert \leq\left(
N/\log N\right)  ^{d}}\varepsilon^{\left\vert A^{c}\right\vert }\frac{Z_{A}%
}{Z_{N}}\label{Est22}\\
&  \quad\leq1+\left(  \frac{N}{\log N}\right)  ^{d}\times\max_{1\leq
m\leq\left(  \frac{N}{\log N}\right)  ^{d}}\nonumber\\
&  \qquad\quad\exp\left[  m\left(  d\log N+\log\varepsilon+d \right)
-\bar cN^{2}m^{\left(  d-2\right)  /d}\right]  .\nonumber
\end{align}
As the function of $m$ in the exponent is convex, it takes its maximum either
at $m=1,$ or at $m=\left(  \frac{N}{\log N}\right)  ^{d}$ (assuming for
simplicity that the latter is an integer). If it takes the maximum at $m=1,$
then we clearly for large $N$ that the whole expression on the right hand side
of (\ref{Est22}) is $\leq2.$ At $m=\left(  \frac{N}{\log N}\right)  ^{d}, $
one has the same situation. We get for the expression in the exponent%
\begin{equation}
N^{d}\left[  \frac{d}{\log^{d-1}N}+\frac{\log\varepsilon+d}{\log^{d}N}-\frac
{\bar c}{\log^{d-2}N}\right]  .\label{Est3}%
\end{equation}
If $N$ is sufficiently large, this is dominated by the third summand, and
therefore the expression in the exponent is for $m=\left(  \frac{N}{\log
N}\right)  ^{d}$ bounded by%
\[
-\frac{CN^{d}}{\log^{d-2}N},
\]
with some $C>0$. This gives for the summand after $1$ in (\ref{Est22}) even
something smaller, namely an expression of order%
\[
\frac{N^{d}}{\log^{d}N}\exp\left[  -\frac{CN^{d}}{\log^{d-2}N}\right]  .
\]
This completes the proof of \eqref{eq:4.1} and therefore \eqref{eq:1.10}.

The rest of this section is devoted to the proof of the capacity estimate
\eqref{eq:capacity-lower}. Recall that, for $A\subset D_{N}^{\circ}$, the
capacity with respect to $D_{N}$ is defined by%
\[
\operatorname*{cap}\nolimits_{D_{N}}\left(  A\right)  :=\sum_{x\in A}%
P_{x}^{RW_{N}^{d}}\left(  T_{\partial D_{N}}<T_{A}\right)
\]
where $T_{A}$ denotes the first hitting time of $A$ after time $0$ for a
random walk on the discrete cylinder.

\begin{lemma}
\label{Le_Capacity_bound}For some constant $c>0,$ depending only on the
dimension $d,$ one has%
\begin{equation}
\operatorname*{cap}\nolimits_{D_{N}}\left(  A\right)  \geq c\left\vert
A\right\vert ^{\left(  d-2\right)  /d}.\label{Claim_Cap}%
\end{equation}

\end{lemma}

\begin{proof}
We will use $c>0$ as a notation for a generic positive (small) constant which
depends only on the dimension and which may change from line to line. In the
course of the proof, we need two other capacities. First the discrete capacity
on $\mathbb{Z}^{d}$: For a finite subset $A\subset\mathbb{Z}^{d},$%
\[
\operatorname*{cap}\nolimits_{\mathbb{Z}^{d}}\left(  A\right)  :=\sum_{x\in
A}P_{x}^{RW^{d}}\left(  T_{A}=\infty\right)  ,
\]
where the random walk here is the standard random walk on $\mathbb{Z}^{d}$. We
will compare $\operatorname*{cap}\nolimits_{D_{N}}$ with $\operatorname*{cap}%
\nolimits_{\mathbb{Z}^{d}}$ and then the latter with the usual Newtonian capacity.

We assume (for simplicity), that $N-3$ is divisible by $6:N=3\left(
2M+1\right)  $ and identify $\mathbb{T}_{N}$ with $\left\{  3M-1,\ldots
,3M+1\right\}  $. Then, subdivide $\mathbb{T}_{N}$ into the $3$ subintervals
$J_{-1}:=\left\{  -3M-1,\ldots,-M-1\right\}  ,\ J_{0}:=\left\{  -M,\ldots
,M\right\}  ,\ J_{1}:=\left\{  M+1,\ldots,3M+1\right\}  $, and $\mathbb{T}%
_{N}^{d-1}$ into the $3^{d-1}$ subboxes $R_{\mathbf{i}}:=J_{i_{1}}\times
\cdots\times J_{i_{d-1}}$, $\mathbf{i}=\left(  i_{1},\ldots,i_{d-1}\right)
\in\left\{  -1,0,1\right\}  ^{d-1}$, and for given $A\subset D_{N}^{\circ}$,
we consider%
\[
A_{\mathbf{i}}:=A\cap\left(  \left[  1,N-1\right]  \times R_{\mathbf{i}%
}\right)  ,
\]
where $\left[  1,N-1\right]  \overset{\mathrm{def}}{=}\left\{  1,\ldots
,N-1\right\}  $. From the monotonicity of the capacity, we get%
\[
\operatorname*{cap}\nolimits_{D_{N}}\left(  A\right)  \geq\operatorname*{cap}%
\nolimits_{D_{N}}\left(  A_{\mathbf{i}}\right)
\]
for every choice of $\mathbf{i}$. We choose $\mathbf{i}$ such that $\left\vert
A_{\mathbf{i}}\right\vert $ is maximal. If we can prove%
\[
\operatorname*{cap}\nolimits_{D_{N}}\left(  A_{\mathbf{i}}\right)  \geq
c\left\vert A_{\mathbf{i}}\right\vert ^{\left(  d-2\right)  /d},
\]
then we obtain (\ref{Claim_Cap}) with an adjustment of $c.$ We therefore can
restrict to sets $A$ which are contained in one of the sets $\left\{
1,\ldots,N-1\right\}  \times R_{\mathbf{i}}$, and we may assume that
$\mathbf{i}=\left(  0,\ldots,0\right)  $ i.e. $A$ is contained in the middle
subbox. As we have periodic boundary conditions on $\mathbb{T}^{d-1}$, this is
no loss of generality.

We can then view $A$ also as a subset of $\mathbb{Z}^{d}$ by the
identification $\mathbb{T}_{N}^{d-1}=\left[  3M-1,3M+1\right]  ^{d-1}$%
\break$\subset\mathbb{Z}^{d-1}$. We now claim that for such an $A$ one has%
\begin{equation}
\operatorname*{cap}\nolimits_{D_{N}}\left(  A\right)  \geq
c\operatorname*{cap}\nolimits_{\mathbb{Z}^{d}}\left(  A\right)
.\label{Cap_Comparison}%
\end{equation}

We denote by $\left\Vert \cdot\right\Vert _{d-1,\infty}$ the subnorm in
$\mathbb{Z}^{d-1}.$ We also write for $0\leq k\leq l$%
\begin{align*}
&  S_{k,l}\overset{\mathrm{def}}{=}\left[  1,N-1\right]  \times\left\{
x\in\mathbb{Z}^{d-1}:k\leq\left\Vert x\right\Vert _{d-1,\infty}\leq l\right\}
,\\
&  \hat{S}_{k,l}\overset{\mathrm{def}}{=}\left\{  0,N\right\}  \times\left\{
x\in\mathbb{Z}^{d-1}:k\leq\left\Vert x\right\Vert _{d-1,\infty}\leq l\right\}
.
\end{align*}
For $k=l$, we write $S_{k}$ instead of $S_{k,k}$. So $A\subset S_{0,M}$. The
boundary of $S_{M+1,3M}$, regarded as a subset of $\mathbb{Z}^{d}$ consists of
the three parts $\hat{S}_{M+1,3M}$,$\ S_{M},\ S_{3M+1}.$ An evident fact is%
\begin{equation}
P_{x}^{RW^{d}}\left(  X_{\tau_{S_{M+1,3M}}}\in\hat{S}_{M+1,3M}\right)  \geq
c>0,\ x\in S_{2M},\label{exit_est}%
\end{equation}
where $\tau_{S}$ is the first exit time from $S$ of a random walk $\left\{
X_{n}\right\}  ,$ starting in $x$. This follows for instance from the weak
convergence of the random walk path to Brownian motion, and the elementary
fact that for a $d$-dimensional Brownian motion starting in $0$, the first
exit from a cylinder $\left[  -\gamma,\gamma\right]  \times\left\{
x\in\mathbb{R}^{d-1}:\left\vert x\right\vert \leq1\right\}  $ through
$\left\{  -\gamma,\gamma\right\}  \times\left\{  \left\vert x\right\vert
\leq1\right\}  $ has probability $p\left(  d,\gamma\right)  >0$.

Consider now a random walk on $\mathbb{Z}^{d}$ starting at $x\in A$. The
escape probability $e_{A}\left(  x\right)  $ to $\infty$ can be bounded as
follows%
\begin{align}
e_{A}\left(  x\right)   &  =P_{x}^{RW^{d}}\left(  T_{A}=\infty\right)  \leq
P_{x}^{RW^{d}}\left(  \tau_{S_{0,2M-1}}<T_{A}\right) \nonumber\\
&  =P_{x}^{RW^{d}}\left(  \tau_{S_{0,2M-1}}<T_{A},X\left(  \tau_{S_{0,2M-1}%
}\right)  \in\hat{S}_{0,2M-1}\right) \nonumber\\
&  +P_{x}^{RW^{d}}\left(  \tau_{S_{0,2M-1}}<T_{A},X\left(  \tau_{S_{0,2M-1}%
}\right)  \in S_{2M}\right) \label{exit_est2}\\
&  \leq P_{x}^{RW^{d}}\left(  \tau_{S_{0,3M}}<T_{A},X\left(  \tau_{S_{0,3M}%
}\right)  \in\hat{S}_{0,3M}\right) \nonumber\\
&  +P_{x}^{RW^{d}}\left(  \tau_{S_{0,2M-1}}<T_{A},X\left(  \tau_{S_{0,2M-1}%
}\right)  \in S_{2M}\right)  .\nonumber
\end{align}
The inequality is coming from the fact that on $\left\{  \tau_{S_{0,2M-1}%
}<T_{A},X\left(  \tau_{S_{0,2M-1}}\right)  \in\hat{S}_{0,2M-1}\right\}  $ one
has $\left\{  \tau_{S_{0,2M-1}}=\tau_{S_{0,3M}}\right\}  $. We estimate the
second summand on the right hand side by (\ref{exit_est}). For abbreviation,
we set $\tau_{1}\overset{\mathrm{def}}{=}\tau_{S_{0,2M-1}}$ and $\tau_{2}%
=\tau_{S_{M+1,3M}}.$ Then, denoting by $\theta_{\tau_{1}}$ the shift operator
by $\tau_{1}$, we have%
\[
\left\{  \tau_{1}<T_{A},\ X_{\tau_{1}}\in S_{2M},\ X_{\tau_{2}}\circ
\theta_{\tau_{1}}\in\hat{S}_{M+1,3M}\right\}  \subset\left\{  \tau_{S_{0,3M}%
}<T_{A},\ X_{S_{0,3M}}\in\hat{S}_{0,3M}\right\}  ,
\]
and therefore, by the strong Markov property, and (\ref{exit_est})%
\begin{align*}
&  P_{x}\left(  \tau_{S_{0,3M}}<T_{A},X_{S_{0,3M}}\in\hat{S}_{0,3M}\right) \\
&  \geq P_{x}\left(  \tau_{1}<T_{A},\ X_{\tau_{1}}\in S_{2M},\ X_{\tau_{2}%
}\circ\theta_{\tau_{1}}\in\hat{S}_{M+1,3M}\right) \\
&  =E_{x}\left(  1_{\left\{  \tau_{1}<T_{A},\ X_{\tau_{1}}\in S_{2M}\right\}
}E_{X_{\tau_{1}}}\left(  X_{\tau_{2}}\in\hat{S}_{M+1,3M}\right)  \right) \\
&  \geq cP_{x}\left(  \tau_{1}<T_{A},\ X_{\tau_{1}}\in S_{2M}\right) .
\end{align*}
Combining this with (\ref{exit_est2}) gives%
\[
e_{A}\left(  x\right)  \leq\left(  1+c^{-1}\right)  P_{x}\left(
\tau_{S_{0,3M}}<T_{A},X_{S_{0,3M}}\in\hat{S}_{0,3M}\right)  .
\]
If for the random walk on $\mathbb{Z}^{d},$ one has $\tau_{S_{0,3M}}%
<T_{A},X_{S_{0,3M}}\in\hat{S}_{0,3M}$, then the random walk on $D_{N}^{\circ
}=\left[  1,N-1\right]  \mathbb{\times T}^{d-1}$ obtained through periodizing
the torus part reaches $\partial D_{N}$ before returning to $A$. Therefore%
\[
P_{x}\left(  \tau_{S_{0,3M}}<T_{A},X_{S_{0,3M}}\in\hat{S}_{0,3M}\right)  \leq
e_{A}^{\mathbb{T}^{d-1}}\left(  x\right)  .
\]
Summing over $x\in A$, this implies (\ref{Cap_Comparison}) (with a changed $c
$).

In order to prove the lemma, it therefore remains to prove that for a finite
subset $A\subset\mathbb{Z}^{d}$, we have%
\[
\operatorname*{cap}\nolimits_{\mathbb{Z}^{d}}\left(  A\right)  \geq
c\left\vert A\right\vert ^{\left(  d-2\right)  /d}.
\]

We denote by $\left(  k_{1},\ldots,k_{d}\right)  ,\ k_{i}\in\left\{
0,1\right\}  $ the $2^{d}$ corner points of a unit box in $\mathbb{Z}^{d}$
spanned by the unit vectors $e_{1},\ldots,e_{d},$ and we write $Q\subset
\mathbb{R}^{d}$ for the closed unit box itself. The discrete translations are
$Q_{y}:=y+Q,$ $y\in\mathbb{Z}^{d}.$ Set
\[
\bar{A}:=\bigcup_{k\in\left\{  0,1\right\}  ^{d}}\left(  A+k\right)
\subset\mathbb{Z}^{d},\ \hat{A}:=\bigcup_{x\in A}Q_{x}\subset\mathbb{R}^{d}.
\]
By the subadditivity and shift invariance of the discrete capacity, we have%
\[
\operatorname*{cap}\nolimits_{\mathbb{Z}^{d}}\left(  A\right)  \geq
2^{-d}\operatorname*{cap}\nolimits_{\mathbb{Z}^{d}}\left(  \bar{A}\right)  .
\]
Define $\phi$ to be the discrete harmonic extension of $1_{\bar{A}}$, i.e.%
\[
\phi\left(  x\right)  =P_{x}\left(  S_{\bar{A}}<\infty\right)  ,
\]
where%
\[
S_{\bar{A}}:=\inf\left\{  n\geq0:X_{n}\in\bar{A}\right\}  ,
\]
$\left\{  X_{n}\right\}  _{n\geq0}$ being the symmetric nearest neighbor
random walk on $\mathbb{Z}^{d}$. $\phi$ is discrete harmonic outside $\bar{A}
$ and satisfies $\lim_{\left\vert x\right\vert \rightarrow\infty}\phi\left(
x\right)  =0$ as $d\geq3$. We write%
\[
\delta\phi\left(  x\right)  :=\left(  \delta_{i}\phi\left(  x\right)  \right)
_{i=1,\ldots,d},\ \delta_{i}\phi\left(  x\right)  :=\phi\left(  x+e_{i}%
\right)  -\phi\left(  x\right)  .
\]
It is well known that the discrete lattice capacity satisfies%
\[
\operatorname*{cap}\nolimits_{\mathbb{Z}^{d}}\left(  \bar{A}\right)  =\frac
{1}{2d}\sum_{x}\left\vert \delta\phi\left(  x\right)  \right\vert ^{2}%
=\frac{1}{2d}\inf\left\{  \sum\nolimits_{x}\left\vert \delta\psi\left(
x\right)  \right\vert ^{2}:\psi\geq1_{\bar{A}}\right\}  .
\]

We interpolate $\phi$ on each of the boxes $Q_{y},\ y\in\mathbb{Z}^{d},$ to a
continuous function $\hat{\phi}:\mathbb{R}^{d}\rightarrow\left[  0,1\right]  $
by defining for $y+x\in\mathbb{R}^{d}$, $y\in\mathbb{Z}^{d},\ x\in Q:$%
\[
\hat{\phi}\left(  y+x\right)  :=\sum_{k\in\left\{  0,1\right\}  ^{d}}%
\prod_{i=1}^{k}x_{i}^{\left(  k_{i}\right)  }\phi\left(  y+\left(
k_{1},\ldots,k_{d}\right)  \right)  ,
\]
where%
\[
x_{i}^{\left(  k_{i}\right)  }=\left\{
\begin{array}
[c]{cc}%
1-x_{i} & \mathrm{for\ }k_{i}=0\\
x_{i} & \mathrm{for\ }k_{i}=1
\end{array}
\right.  .
\]
By the construction, $\hat{\phi}$ is uniquely defined also on the
intersections of different boxes. It is evident that%
\begin{equation}
\hat{\phi}\geq1_{\hat{A}}\label{domination}%
\end{equation}
because for $x\in A$, all corner points of $Q_{x}$ belong to $\bar{A}$ on
which $\phi$ is $1$. The partial derivatives inside of box $Q_{y}$ are%
\begin{align*}
\frac{\partial\hat{\phi}}{\partial x_{i}}\left(  y+x\right)   &  =\sum
_{k_{1},\ldots,k_{i-1},k_{i+1},\ldots,k_{d}\in\left\{  0,1\right\}  }%
\prod_{j:j\neq i}x_{j}^{\left(  k_{j}\right)  }\Big [\\
&  -\phi\left(  y+\left(  k_{1},\ldots,k_{i-1},0,k_{i+1},\ldots,k_{d}\right)
\right)  \Big ]
\end{align*}
From this representation, it follows that with some constant $C\left(
d\right)  >0$%
\[
\frac{1}{2}\int_{\mathbb{R}^{d}}\left\vert \nabla\hat{\phi}\left(  x\right)
\right\vert ^{2}dx\leq C\left(  d\right)  \sum_{x\in\mathbb{Z}^{d}}\left\vert
\delta\phi\left(  x\right)  \right\vert ^{2}=C\left(  d\right)
\operatorname*{cap}\nolimits_{\mathbb{Z}^{d}}\left(  \bar{A}\right)  .
\]
The Newtonian capacity of a compact subset $K\subset\mathbb{R}^{d}$ is defined
by%
\[
\operatorname*{cap}\nolimits_{d}\left(  K\right)  \overset{\mathrm{def}}%
{=}\inf\left\{  \frac{1}{2}\int\left\vert \nabla\psi\left(  x\right)
\right\vert ^{2}dx:\psi\in H_{1}\left(  \mathbb{R}^{d}\right)  ,\ \psi
\geq1_{\hat{A}}\right\}  ,
\]
where $H_{1}$ is the Sobolev space of weakly once differentiable functions on
$\mathbb{R}^{d}$ with square integrable derivative. Using (\ref{domination}),
we get%
\begin{align}
\operatorname*{cap}\nolimits_{d}\left(  \hat{A}\right)   &  \leq C\left(
d\right)  \operatorname*{cap}\nolimits_{\mathbb{Z}^{d}}\left(  \bar{A}\right)
\label{Cap_discr_cont}\\
&  \leq2^{d}C\left(  d\right)  \operatorname*{cap}\nolimits_{\mathbb{Z}^{d}%
}\left(  A\right)  .\nonumber
\end{align}
By the Poincar\'{e}-Faber-Szeg\"{o} inequality for the Newtonian capacity (see
\cite{PS} for $d=3,$ and \cite{Jauregui}, Appendix A for general $d\geq3 $),
one has with some new constant $c>0$%
\[
\operatorname*{cap}\nolimits_{d}\left(  \hat{A}\right)  \geq c\left(
\operatorname*{vol}\left(  \hat{A}\right)  \right)  ^{\left(  d-2\right)
/d}=c\left\vert A\right\vert ^{\left(  d-2\right)  /d},
\]
which, together with (\ref{Cap_discr_cont}) proves the claim.
\end{proof}

\section{The large deviation estimate: Proof of (\ref{eq:1.11})}

\subsection{Preliminaries} \label{section:6.1}

\textbf{Convention: }All statements we make are only claimed to be true for
large enough $N$ without special mentioning.

\textbf{Markov property}: Let $\mu_{\Lambda}$ be the probability measure of
the free field, that is the Gaussian field without pinning,
 on a finite subset $\Lambda$ of the cylinder $\mathbb{Z\times
T}_{N}^{d}$, with arbitrary boundary conditions on $\partial\Lambda$, and let
$B\subset\Lambda$. We write $\mathcal{F}_{A}$ for $\sigma\left(  \phi_{i}:i\in
A\right)  $. Then for any $X\in\mathcal{F}_{B}$ we have%
\begin{equation}
\mu_{\Lambda}\left(  \left.  X\right\vert \mathcal{F}_{B^{c}}\right)
=\mu_{\Lambda}\left(  \left.  X\right\vert \mathcal{F}_{\partial B\cap\Lambda
}\right)  . \label{Markov}%
\end{equation}

\textbf{FKG-inequality}: Let $G:\mathbb{R}^{\Lambda}\rightarrow\mathbb{R}$ be
a measurable function which is non-decreasing in all arguments, and let
$\mu_{\Lambda,\mathbf{x}}$ be the free field on $\Lambda$ with boundary
condition $\mathbf{x}\in\mathbb{R}^{\partial\Lambda}$. The FKG-property states
that $\int G~d\mu_{\Lambda,\mathbf{x}}$ is nondecreasing as a function of
$\mathbf{x}\in\mathbb{R}^{\partial\Lambda}$ in all coordinates.

We will use the expansion%
\begin{equation}
\mu_{N}^{aN,bN,\varepsilon}=\sum_{A\subset D_{N}^{\circ}}p_{N}^{\varepsilon
}\left(  A\right)  \mu_{A}^{aN,bN},\label{Main_Expansion}%
\end{equation}
where $\mu_{A}$ is the standard free field with boundary condition $0$ on
$\partial A\cap D_{N}^{\circ}$ and $aN$, respectively $bN$ on $\partial D_{N},$
extended by the Dirac measure at $0$ on $A^{c}\overset{\mathrm{def}}{=}%
D_{N}^{\circ}\backslash A$, and where%
\[
p_{N}^{\varepsilon}\left(  A\right)  =\frac{\varepsilon^{\left\vert
A^{c}\right\vert }Z_{A}^{aN,bN}}{Z_{A}^{aN,bN,\varepsilon}}%
\]
$\left\{  p_{N}^{\varepsilon}\left(  A\right)  \right\}  _{A\subset D_{N}^{\circ}%
}$ is a probability distribution on the set of subsets of $D_{N}^{\circ}$.

We write%
\[
A_{N,\alpha}\overset{\mathrm{def}}{=}\left\{  \operatorname{dist}_{L^{1}%
}(h^{N},\{\hat{h},\bar{h}\})\geq N^{-\alpha}\right\}  ,
\]
so, in order to prove (\ref{eq:1.11}), we have to prove $\mu_{N}^{\varepsilon
}\left(  A_{N,\alpha}\right)  \rightarrow0$ for small enough $\alpha$. Let%
\[
\Omega_{N}^{+}\overset{\mathrm{def}}{=}\left\{  \phi_{i}\geq-\log N,\ \forall
i\in D_{N}^{\circ}\right\}  .
\]

\begin{lemma}%
\[
\lim_{N\rightarrow\infty}\mu_{N}^{aN,bN,\varepsilon}\left(  \Omega_{N}%
^{+}\right)  =1.
\]

\end{lemma}

\begin{proof}
We use
\begin{align*}
\mu_{A}^{aN,bN}\left(  (\Omega_{N}^{+})^c \right)   &  \leq N^{d}\sup_{i\in A}%
\mu_{A}^{aN,bN}\left(  \phi_{i}\leq-\log N\right)  \\
&  \leq N^{d}\sup_{i\in A}\mu_{A}^{0}\left(  \phi_{i}\leq-\log N\right)  \\
&  \leq N^{d}\sup_{i\in A}\exp\left[  -\frac{\left(  \log N\right)  ^{2}%
}{2G_{A}\left(  i,i\right)  }\right]  \leq N^{d}\exp\left[  -\frac{\left(
\log N\right)  ^{2}}{2C}\right],
\end{align*}
where $\mu_{A}^{0}$ has boundary conditions $0$ on $A^{c}$ (and not just on
$A^{c}\cap D_{N}^{\circ}$). In the last inequality, we have used $G_{A}\left(
i,i\right)  \leq G_{\mathbb{Z}^{d}}\left(  i,i\right)  =C<\infty$ as we assume
$d\geq3.$ For the second inequality, we use FKG and $a,b\geq0.$
Combining with \eqref{Main_Expansion} shows the conclusion.
\end{proof}

Using this lemma, it suffices to prove%
\begin{equation}
\lim_{N\rightarrow\infty}\mu_{N}^{aN,bN,\varepsilon}\left(  A_{N,\alpha}%
\cap\Omega_{N}^{+}\right)  =0\label{LDP_main}%
\end{equation}
for $\alpha$ chosen sufficiently small.

We will consider the random fields on an extended set%
\[
D_{N,\mathrm{ext}}\overset{\mathrm{def}}{=}\left\{  -N,-N+1,\ldots,2N\right\}
\times\mathbb{T}_{N}^{d-1},
\]
with%
\[
D_{N,\mathrm{ext}}^{\circ}\overset{\mathrm{def}}{=}\left\{  -N+1,\ldots
,2N-1\right\}  \times\mathbb{T}_{N}^{d-1},
\]%
\[
\partial D_{N,\mathrm{ext}}\overset{\mathrm{def}}{=}\left\{  -N,2N\right\}
\times\mathbb{T}_{N}^{d-1},\ \check{D}_{N,\mathrm{\mathrm{ext}}}%
\overset{\mathrm{def}}{=}D_{N,\mathrm{ext}}^{\circ}\backslash D_{N}^{\circ}.
\]

We define the measure $\mu_{N,\mathrm{ext}}^{\varepsilon}$ on $\mathbb{R}%
^{D_{N,\mathrm{ext}}}$ with $0$ boundary conditions on $\partial
D_{N,\mathrm{ext}}$ and $\varepsilon$-pinning on $D_{N}^{\circ},$ i.e.%
\begin{align*}
\mu_{N,\mathrm{ext}}^{\varepsilon}\left(  d\phi\right)  =& \frac{1}%
{Z_{N,\mathrm{ext}}^{\varepsilon}}\exp\left[  -\frac{1}{2}\sum_{\left\langle
i,j\right\rangle \subset D_{N,\mathrm{ext}}}\left(  \phi_{i}-\phi_{j}\right)
^{2}\right]  \\
& \times\prod_{i\in D_{N}^{\circ}}\left(  d\phi_{i}+\varepsilon\delta_{0}\left(
d\phi_{i}\right)  \right) 
\prod_{i\in D_{N,\mathrm{ext}}^{\circ}\backslash D_{N}^{\circ}}d\phi
_{i},\quad
 \phi\equiv0\ \mathrm{on\ }\partial D_{N,\mathrm{ext}}.
\end{align*}
$\mu_{N,\mathrm{ext}}$ is the usual Gaussian field corresponding to
$\varepsilon=0$. The reader should pay attention to the fact that pinning for
$\mu_{N,\mathrm{ext}}^{\varepsilon}$ is \textit{only }on $D_{N}^{\circ}$.

We write $\mathbb{F}$ for the set of subsets of $D_{N,\mathrm{ext}}^{\circ}$
satisfying $\check{D}_{N,\mathrm{\mathrm{ext}}}\subset F$. For $F\in
\mathbb{F}$ we write $\mu_{F}^{0}$ for the Gaussian field on $\mathbb{R}^{F}$
with $0$ boundary condition on $\partial F$. It is sometimes convenient to
extend $\mu_{F}$ to $\mathbb{R}^{D_{N,\mathrm{ext}}}$ by multiplying it with
$\prod_{i\notin F}\delta_{0}\left(  d\phi_{i}\right)  .$ Remark that $\partial
D_{N}\subset F.$

We need the following lemma for the proof of Lemma \ref{lemma:6.5} below.

\begin{lemma}  \label{lem:maxmimum-a} 
Let $F\in\mathbb{F}$, and $s,t>0$ satisfy
$s>t/2,t>s/2$. Let $\psi_{F}:F\cup\partial F\rightarrow\mathbb{R}$ be a
function which minimizes $H(\psi)$ subject to the boundary conditions $0$ at
$\partial F,\psi_{F}\geq s$ on $\partial_{L}D_{N}$, $\psi_{F}\geq t$ on
$\partial_{R}D_{N}$. Then $\psi_{F}$ is unique, and is the harmonic function on
$F\backslash\partial D_{N}$ with boundary condition $0$ on $\partial F,\ s$ on
$\partial_{L}D_{N}$ and $t$ on $\partial_{R}D_{N}$.

Furthermore, one has
\begin{equation}
\Delta\psi_{F}(i)=\sum_{j:|i-j|=1}(\psi_{F}(j)-\psi_{F}(i))\leq0,\quad
i\in\partial D_{N}. \label{eq:maximumprinciple}%
\end{equation}

\end{lemma}

\begin{remark}
\label{rem:concave}The condition $s>t/2,t>s/2$ is needed to ensure that
piecewise linear function on $\left[  -1,2\right]  $ which is $s$ at $0$,$\ t$
at $1$, and $0$ at $\left\{  -1,2\right\}  $ is concave. We will later apply
the lemma with $s=aN+o\left(  N\right)  ,\ t=bN+o\left(  N\right)  $, so that
we should have $a>b/2$, $b>a/2$ (and $N$ large). If this is not satisfied, we
can take instead of $D_{N,\mathrm{\mathrm{ext}}}$ the smaller extensions
$\left\{  -cN,-cN+1,\ldots,N+cN\right\}  \times\mathbb{T}_{N}^{d-1}$ with $c$
satisfying%
\[
\frac{bc}{1+c}<a,\ \frac{ac}{1+c}<b,
\]
in which case the corresponding piecewise linear function on $\left[
-c,1+c\right]  $ is concave. After this modification, all the arguments below
go through. For the sake of notational simplicity, we stay with our choice for
$D_{N,\mathrm{\mathrm{ext}}}$ and the conditions on $s,t$.
\end{remark}

To prove this lemma, we prepare another lemma, which reduces the 
variational problem to that on
superharmonic functions and gives a comparison for such functions.

\begin{lemma} \label{lem:maxmimum-p}
{\rm (1)} The minimizer $\psi_{F}$ of $H(\psi)$ subject to the conditions
\begin{equation}
\psi_{F}=0\ \mathrm{at\ }\partial F,\quad\psi_{F}\geq s\ \mathrm{on\ }%
\partial_{L}D_{N},\quad\psi_{F}\geq t\ \mathrm{on\ }\partial_{R}D_{N},
\label{eq:maximumprinciple-3}%
\end{equation}
is characterized as the unique solution satisfying this condition and
\begin{equation}
\left\{
\begin{array}
[c]{cc}%
\Delta\psi_{F}=0 & \mathrm{on\ }F\cup\left(  \partial D_{N}\backslash I\right)
\\
\Delta\psi_{F}\leq0 & \mathrm{on\ }I
\end{array}
\right.  , \label{eq:maximumprinciple-5}%
\end{equation}
where $I=I(\psi_{F})$ is a region in $\partial D_{N}$ given by $I\equiv
I_{L}\cup I_{R}:=\{i\in\partial_{L}D_{N};\psi_{F}(i)=s\}\cup\{i\in\partial
_{R}D_{N};\psi_{F}(i)=t\}$.  \\
{\rm (2)} Assume that $\psi^{(1)}$ and $\psi^{(2)}$ are two solutions of the
problem \eqref{eq:maximumprinciple-5} satisfying $\psi^{(1)}\geq\psi^{(2)}$ on
$F^{c}$ instead of $\psi^{(1)}=\psi^{(2)}=0$ on $F^{c}$ in
\eqref{eq:maximumprinciple-3}. Then, we have that $\psi^{(1)}\geq\psi^{(2)}$
on $F$.
\end{lemma}

\begin{proof}
(1) Let $\psi_{F}$ be the minimizer of $H(\psi)$ subject to the conditions
\eqref{eq:maximumprinciple-3}. Then, $\psi_{F}$ is harmonic on $F\cup(\partial
D_{N}\setminus I)$, since
\begin{align*}
0  &  =\frac{d}{da}H(\psi_{F}+a\delta_{i})\big|_{a=0}=\frac{d}{da}%
\sum_{j:|j-i|=1}(\psi_{F}(j)-\psi_{F}(i)+a)^{2}\big|_{a=0}\\
&  =-2\sum_{j:|j-i|=1}(\psi_{F}(j)-\psi_{F}(i))=-2(\Delta\psi_{F})(i),
\end{align*}
for every $i\in F\cup(\partial D_{N}\setminus I)$, where $\delta_{i}%
\in\mathbb{R}^{D_{N,\mathrm{ext}}}$ is defined by $\delta_{i}(j)=\delta_{ij}$.
For $i\in I$, since
\[
\frac{d}{da}H(\psi_{F}+a\delta_{i})\big|_{a=0+}\geq0,
\]
we have $\Delta\psi_{F}\leq0$. Thus the minimizer $\psi_{F}$ satisfies 
\eqref{eq:maximumprinciple-5}.

To show the uniqueness of the solution $\psi_{F}$ of
\eqref{eq:maximumprinciple-5}, let $\psi^{(1)}$ and $\psi^{(2)}$ be two
solutions of the problem \eqref{eq:maximumprinciple-5}. Then, we have that
\begin{equation}
\left(  \psi^{(1)}(i)-\psi^{(2)}(i)\right)  \left(  \Delta\psi^{(1)}%
(i)-\Delta\psi^{(2)}(i)\right)  \geq0, \label{eq:maximumprinciple-4}%
\end{equation}
for all $i\in F$. In fact, denoting $I^{(k)}=I(\psi_{F}^{(k)}),I_{L}%
^{(k)}=I_{L}(\psi_{F}^{(k)}),I_{R}^{(k)}=I_{R}(\psi_{F}^{(k)})$ for $k=1,2$,
if $i\in F\cup(\partial D_{N}\setminus(I^{(1)}\cup I^{(2)}))$, then
$\Delta\psi^{(1)}(i)=\Delta\psi^{(2)}(i)=0$. If $i\in I_{L}^{(1)}\setminus
I^{(2)}$, then $\psi^{(1)}(i)-\psi^{(2)}(i)=s-\psi^{(2)}(i)<0$ and $\Delta
\psi^{(1)}(i)-\Delta\psi^{(2)}(i)=\Delta\psi^{(1)}(i)\leq0$. The case $i\in
I_{L}^{(2)}\setminus I^{(1)}$ and the cases with $I_{R}^{(1)},I_{R}^{(2)}$ are
similar. If $i\in I^{(1)}\cap I^{(2)}$, then $\psi^{(1)}(i)=\psi^{(2)}(i)$. In
all cases, \eqref{eq:maximumprinciple-4} holds.

From \eqref{eq:maximumprinciple-4}, setting $\psi=\psi^{(1)}-\psi^{(2)}$,
since $\psi(i)=0$ on $\partial F$, we have that
\[
0\leq\sum_{i\in F}\psi(i)\Delta\psi(i)=-\sum_{i,j\in\bar{F}:|i-j|=1}\left(
\psi(i)-\psi(j)\right)  ^{2},
\]
see (2.19) in \cite{Fu05} for this summation by parts formula. This shows
$\psi(i)=\psi(j)$ for all $i,j\in\bar{F}=F\cup\partial F:|i-j|=1$. Since
$\psi(i)=0$ at $\partial F$, this proves $\psi=0$ on $F$, and therefore the uniqueness.

(2) Set $\psi=\psi^{(1)}-\psi^{(2)}$ and assume that $-m=\min_{i\in F}%
\psi(i)<0$. Let $i_{0}\in F$ be the point such that $\psi(i_{0})=-m$. Then,
since $\psi^{(2)}(i_{0})=\psi^{(1)}(i_{0})+m>\psi^{(1)}(i_{0})$, from the
first condition in \eqref{eq:maximumprinciple-5}, we see $\Delta\psi
^{(2)}(i_{0})=0$. Thus, $\Delta\psi(i_{0})=\Delta\psi^{(1)}(i_{0})-\Delta
\psi^{(2)}(i_{0})=\Delta\psi^{(1)}(i_{0})\leq0$. Since we have shown
\[
0\geq\Delta\psi(i_{0})=\sum_{j:|i_{0}-j|=1}(\psi(j)-\psi(i_{0}))
\]
and $\psi(j)-\psi(i_{0})\geq0$, we obtain that $\psi(j)=\psi(i_{0})(=-m) $ for
all $j:|i_{0}-j|=1$. Continuing this procedure, we see that $\psi\equiv-m<0$
on the connected component of $F\cup\partial F$ containing $i_{0}$, but this
contradicts with the boundary condition: $\psi\geq0$ on $F^{c}$.
\end{proof}

\begin{proof}
[Proof of Lemma \ref{lem:maxmimum-a}]The harmonic property of $\psi_{F}$ on
$F$ and the property \eqref{eq:maximumprinciple} are immediate from Lemma
\ref{lem:maxmimum-p}. What are left are to show that $\psi_{F}=s$ on
$\partial_{L}D_{N}$, $\psi_{F}=t$ on $\partial_{R}D_{N}$ and to give the
explicit form of $\psi_{F}$ on $D_{N,\mathrm{ext}}^{\circ}\setminus D_{N}$
stated in the lemma. Indeed, define $\psi^{(1)}$ by
\[
\psi^{(1)}(i)=\left\{
\begin{array}
[c]{cc}%
\left(  \frac{i_{1}}{N}+1\right)  s & \mathrm{on\ }\left\{  -N,\ldots
,0\right\}  \times\mathbb{T}_{N}^{d-1}\\
\frac{N-i_{1}}{N}s+\frac{i_{1}}{N}t & \mathrm{on\ }\left\{  1,\ldots
,N-1\right\}  \times\mathbb{T}_{N}^{d-1}\\
\left(  2-\frac{i_{1}}{N}\right)  t & \mathrm{on\ }\left\{  N,\ldots
,2N\right\}  \times\mathbb{T}_{N}^{d-1}%
\end{array}
\right.  .
\]
Then, by the concavity condition on the segments mentioned in the lemma,
$\psi^{(1)}$ satisfies the condition \eqref{eq:maximumprinciple-5} and
$\psi^{(1)}\geq\psi^{(2)}:=\psi_{F}$ on $F^{c}$. Thus, Lemma
\ref{lem:maxmimum-p}-(2) proves $\psi^{(1)}\geq\psi_{F}$ on $F$. This implies
that $\psi_{F}=s$ on $\partial_{L}D_{N}$, $\psi_{F}=t$ on $\partial_{R}D_{N}$.
Once this is shown, the rest is easy, since $\psi_{F}$ is harmonic on
$D_{N,\mathrm{ext}}^{\circ}\setminus D_{N}$.
\end{proof}

With $F$ still as above, and $\mathbf{x}_{L}\in\mathbb{R}^{\partial_{L}D_{N}%
},\ \mathbf{x}_{R}\in\mathbb{R}^{\partial_{R}D_{N}}$, let $\phi_{F,\mathbf{x}%
_{L},\mathbf{x}_{R}}:F\cap D_{N}^{\circ}\rightarrow\mathbb{R}$ be the harmonic
function with $0$ boundary condition on $\partial F\cap D_{N}^{\circ}$,
$\mathbf{x}_{L}$ on $\partial_{L}D_{N}$, and $\mathbf{x}_{R}$ on $\partial
_{R}D_{N}$.  We set $\Xi\left(  F,\mathbf{x}_{L},\mathbf{x}_{R}\right)  \overset
{\mathrm{def}}{=}H\left(  \phi_{F,\mathbf{x}_{L},\mathbf{x}_{R}}\right)  $.

\begin{lemma}  \label{lemma:6.5}
\label{Le_Approx_Energy} Let $F\in\mathbb{F}$. Then, we have the followings.\\
{\rm (1)} Let $s,t\geq0$. Then%
\begin{align*}
&  \mu_{F,\mathrm{\mathrm{ext}}}\left(  \phi%
\vert
_{\partial_{L}D_{N}}\geq s,\ \phi%
\vert
_{\partial_{R}D_{N}}\geq t\right) \\
&  \leq\exp\left[  -\Xi\left(  F,s,t\right)  -\tfrac{s^{2}}2 N^{d-2}-
\tfrac{t^{2}}2 N^{d-2}\right]  .
\end{align*}
{\rm (2)} Let $\delta>0$ and $\mathbf{x}_{L},\ \mathbf{x}_{R}$ satisfy
$aN-N^{1-\delta}\leq\mathbf{x}_{L}\leq aN,$ $bN-N^{1-\delta}\leq\mathbf{x}%
_{R}\leq bN.$ Then%
\[
\Xi\left(  F,aN,bN\right)  \left(  1-\frac{2N^{-\delta}}{\min\left(
a,b\right)  }\right)  \leq\Xi\left(  F,\mathbf{x}_{L},\mathbf{x}_{R}\right)
\leq\Xi\left(  F,aN,bN\right)  .
\]
\end{lemma}

\begin{proof}
(1) We consider $\psi_{F}$ as in the previous lemmas. With the transformation
of variables $\phi_{i}=\bar{\phi}_{i}+\psi_{F}\left(  i\right)  $, we obtain%
\begin{align*}
&  \mu_{F,\mathrm{ext}}\left(  \phi%
\vert
_{\partial_{L}D_{N}}\geq s,\phi%
\vert
_{\partial_{R}D_{N}}\geq t\right) \\
&  =\left[  -\Xi\left(  F,s,t\right)  -\tfrac{s^{2}}2 N^{d-2}-\tfrac{t^{2}}2 N^{d-2}\right] \\
&  \times\int_{\bar{\phi}_{i}\geq0,\ i\in\partial D_{N}}\exp\left[
-2 \sum\nolimits_{i\in\partial D_{N}}\bar{\phi}_{i}\sum\nolimits_{j}\left(
\psi_{F}\left(  j\right)  -\psi_{F}\left(  i\right)  \right)  \right]
\mu_{F,\mathrm{ext}}\left(  d\bar{\phi}\right)  .
\end{align*}
By Lemma \ref{lem:maxmimum-a}, the integrand is $\leq1$ in the domain of
integration, which proves the claim.

(2) It evidently suffices to prove%
\[
\Xi\left(  F,aN-N^{1-\delta},bN-N^{1-\delta}\right)  \geq\Xi\left(
F,aN,bN\right)  \left(  1-\frac{2N^{-\delta}}{\min\left(  a,b\right)
}\right)  .
\]
Without loss of generality, we assume $b\geq a$. Then
\[
\frac{bN}{bN-N^{1-\delta}}\leq\frac{aN}{aN-N^{1-\delta}}.
\]
Let $\psi$ be the harmonic function on $F$ which is $0$ on $\partial F\cap
D_{N}^{\circ}$, $aN-N^{1-\delta}$ on $\partial_{L}D_{N}$ and $bN-N^{1-\delta}$ on
$\partial_{R}D_{N}$. Define%
\[
\psi^{\prime}\overset{\mathrm{def}}{=}\frac{aN}{aN-N^{1-\delta}}\psi
\]
which is harmonic on $F$, $0$ on $\partial F\cap D_{N}^{\circ}$, $aN$ on
$\partial_{L}D_{N}$ and $\geq bN$ on $\partial_{R}D_{N}$. If we define
$\psi^{\prime\prime}$ to be the harmonic function on $F$ which has boundary
conditions $aN,\ bN$ on $\partial_{L}D_{N},\ \partial_{R}D_{N}$, respectively,
and $0$ on $\partial F\cap D_{N}^{\circ}$, we get%
\begin{align*}
H\left(  \psi\right)   &  =\left(  1-\frac{N^{-\delta}}{a}\right)
^{2}H\left(  \psi^{\prime}\right)  \geq\left(  1-\frac{N^{-\delta}}{a}\right)
^{2}H\left(  \psi^{\prime\prime}\right) \\
&  =\left(  1-\frac{N^{-\delta}}{a}\right)  ^{2}\Xi\left(  F,aN,bN\right)
\geq\left(  1-\frac{2N^{-\delta}}{a}\right)  \Xi\left(  F,aN,bN\right)  .
\end{align*}

\end{proof}

\subsection{Superexponential estimate}

Given $0<\beta<1$, we consider the following coarse graining: We divide $D_{N} $
into $N^{d\left(  1-\beta\right)  }$ subboxes of sidelength $N^{\beta}.$ For
the sake of simplicity, we assume that $N^{\beta}$ divides $N$ as before.
We write $\mathcal{B}_{N}\equiv \mathcal{B}_{N,\beta}$ for the set of these 
subboxes, and $\mathcal{\hat{B}}_{N} \equiv \mathcal{\hat{B}}_{N,\beta}$
for the set of unions of boxes in $\mathcal{B}_{N}.$ We attach to every subbox
$C\in\mathcal{B}_{N}$ the arithmetic mean%
\[
\phi_{C}^{\mathrm{cg,}\beta,N}\overset{\mathrm{def}}{=}N^{-d\beta}\sum_{j\in
C}\phi_{j}.
\]
Then define%
\begin{align*}
\phi^{\mathrm{cg},\beta,N}\left(  i\right)   &  =\phi_{C}^{\mathrm{cg}%
,\beta,N},\ i\in C,\\
h^{\mathrm{cg},\beta,N}\left(  x\right)   &  =\frac{1}{N}\phi^{\mathrm{cg},\beta,N}\left(
\left[  xN\right]  \right)  ,\ x\in D = \left[  0,1\right]  \times\mathbb{T}^{d-1}.
\end{align*}

\begin{proposition} \label{Prop_Coarse_graining}
For every $\eta>0$ satisfying $2\eta+\beta<1$ and for large enough
$N$ (as stated at the beginning of Section \ref{section:6.1}), 
\[
\mu_{N}^{aN,bN,\varepsilon}\left(  \left\Vert h^{\mathrm{cg},\beta,N}%
-h^{N}\right\Vert _{L^1(D)}\geq N^{-\eta}\right)  \leq C \exp\left[  -\frac{1}%
{C}N^{d+1-2\eta-\beta}\right].
\]

\end{proposition}

\begin{proof}
We first consider the $\mu_{N,\mathrm{ext}}^{\varepsilon}$ which is defined as
the free field with $0$ boundary conditions (and no boundary conditions on
$\partial D_{N}$). We use the extension as explained in Section
\ref{section:6.1}. Expanding the product in the usual way, we get%
\begin{equation}
\mu_{N,\mathrm{ext}}^{\varepsilon}=\sum_{A\in\mathbb{F}}\frac{Z_{A}%
}{Z_{N,\mathrm{ext}}^{\varepsilon}}\varepsilon^{\left\vert A^{c}\right\vert
}\mu_{A}, \label{Representation}%
\end{equation}
where $A^{c}\overset{\mathrm{def}}{=}D_{N,\mathrm{\mathrm{ext}}}^{\circ}\backslash
A$, and $\mu_{A}$ is the centered Gaussian field on $D_{N,\mathrm{ext}}%
^{\circ}$ with zero boundary conditions outside on $\partial A$. The
covariance function of $\mu_{A}$ is denoted by $G_{A}.$ It is convenient to
extend $G_{A}\left(  i,j\right)  $ to $i$ or $j\notin A$ by putting it $0$. It
is the Green's function for a random walk on $A$ with Dirichlet boundary condition.

We can define $h^{N},h^{\mathrm{cg},\beta,N}$ in the same way as before, but on
the extended space. The coarse graining is done here on the full
$D_{N,\mathrm{ext}}.$ We first prove that%
\begin{equation}
\mu_{N,\mathrm{ext}}^{\varepsilon}\left(  \left\Vert h^{N}-h^{\mathrm{cg}%
,\beta,N}\right\Vert _{L^1(D)}\geq N^{-\eta}\right)  \leq C\exp\left[  -\frac{1}%
{C}N^{d+1-2\eta-\beta}\right]  \label{CG_zero}%
\end{equation}
provided $2\eta+\beta<1$.

Using the expansion (\ref{Representation}), it suffices to prove the
inequality for $\mu_{A}$, uniformly in $A$. So we have to estimate%
\[
\mu_{A}\left(  \sum\nolimits_{i\in D_{N}^{\circ}}\left\vert N^{-d\beta}%
\sum\nolimits_{j\in C_{i}}\left(  \phi_{j}-\phi_{i}\right)  \right\vert \geq
N^{1+d-\eta}\right)
\]
where $C_{i}\in\mathcal{B}_{N,\beta,\mathrm{ext}}$ denotes the box in which
$i$ lies. The sum over the extended region $D_{N,\mathrm{ext}}^{\circ}$ 
of the absolute values is%
\[
\sup_{\mathbf{\sigma}}\sum_{i\in D_{N,\mathrm{ext}}^{\circ}}\sigma_{i}\left(
N^{-d\beta}\sum\nolimits_{j\in C_{i}}\left(  \phi_{j}-\phi_{i}\right)
\right)  ,
\]
where $\mathbf{\sigma}=\left(  \sigma_{i}\right)  \in\left\{  -1,1\right\}
^{D_{N,\mathrm{ext}}^{\circ}}.$ Therefore, with%
\[
X\left(  \mathbf{\sigma}\right)  \overset{\mathrm{def}}{=}\sum\nolimits_{i\in
D_{N,\mathrm{ext}}^{\circ}}\sigma_{i}\left(  N^{-d\beta}\sum\nolimits_{j\in
C_{i}}\left(  \phi_{j}-\phi_{i}\right)  \right),
\]%
we have
\begin{align*}
&  \mu_{A}\left(  \sum\nolimits_{i\in D_{N,\mathrm{ext}}^{\circ}}\left\vert
N^{-d\beta}\sum\nolimits_{j\in C_{i}}\phi_{j}-\phi_{i}\right\vert \geq
N^{1+d-\eta}\right) \\
&  \leq2^{\left\vert D_{N,\mathrm{ext}}^{\circ}\right\vert }\sup
_{\mathbf{\sigma}}\mu_{A}\left(  X\left(  \mathbf{\sigma}\right)  \geq
N^{1+d-\eta}\right),
\end{align*}
where $\mu_A=\mu_{A, \mathrm{ext}}$.
The $X\left(  \mathbf{\sigma}\right)  $ are centered Gaussian variables, so we
just have to estimate the variances, uniformly in $\mathbf{\sigma}$ and $A.$%
\begin{align*}
\operatorname*{var}\nolimits_{\mu_{A}}\left(  X\left(  \mathbf{\sigma}\right)
\right)   &  \leq\sum_{i,k\in D_{N,\mathrm{ext}}^{\circ}}\left\vert E_{\mu
_{A}}\left(  N^{-2d\beta}\sum\nolimits_{j^{\prime}\in C_{i}}\left(
\phi_{j^{\prime}}-\phi_{i}\right)  \sum\nolimits_{j\in C_{k}}\left(  \phi
_{j}-\phi_{k}\right)  \right)  \right\vert \\
&  \leq2\sum_{i,k\in D_{N,\mathrm{ext}}^{\circ}}\left\vert E_{\mu_{A}}\left(
N^{-d\beta}\phi_{i}\sum\nolimits_{j\in C_{k}}\left(  \phi_{j}-\phi_{k}\right)
\right)  \right\vert \\
&  \leq2\sum_{i\in D_{N,\mathrm{ext}}^{\circ}}N^{-d\beta}\sum_{k\in
D_{N,\mathrm{ext}}^{\circ}}\sum_{j\in C_{k}}\left\vert G_{A}\left(
i,j\right)  -G_{A}\left(  i,k\right)  \right\vert \\
&  \leq2\sum_{i\in D_{N,\mathrm{ext}}^{\circ}}N^{-d\beta}\sum_{k\in
D_{N,\mathrm{ext}}^{\circ}}\sum_{j:d\left(  j,k\right)  \leq\rho\left(
d,\beta\right)  }\left\vert G_{A}\left(  i,j\right)  -G_{A}\left(  i,k\right)
\right\vert ,
\end{align*}
where $G_{A}$ is the Green's function of ordinary random walk with killing at
exiting $A$ or reaching $\partial D_{N,\mathrm{ext}}$. $d\left(  j,k\right)  $
is any reasonable distance on the discrete torus, for instance the length of
the shortest path from $j$ to $k$. $\rho\left(  d,\beta\right)  $ is the
diameter of the boxes in $\mathcal{B}_{N,\beta}$. If we define $K\left(
d,\beta\right)  $ to be the ball of radius $\rho\left(  d,\beta\right)  $
around $0\in D_{N,\mathrm{ext}}$, we can also write the above expression as%
\[
2\sum_{j\in K}N^{-d\beta}\sum_{i\in D_{N,\mathrm{ext}}^{\circ}}\sum_{k\in
D_{N,\mathrm{ext}}^{\circ}}\left\vert G_{A}\left(  i,k+j\right)  -G_{A}\left(
i,k\right)  \right\vert .
\]
For $i\in A$, let $\pi_{A}\left(  i,\cdot\right)  $ be the first exit
distribution from $A$ of a random walk starting in $i$. It is well known that%
\[
G_{A}\left(  i,k\right)  =G_{N,\mathrm{ext}}\left(  i,k\right)  -\sum_{s}%
\pi_{A}\left(  i,s\right)  G_{N,\mathrm{ext}}\left(  s,k\right)
\]
where $G_{N,\mathrm{ext}}$ is the the Green's function on $D_{N,\mathrm{ext}}$
with Dirichlet boundary condition on $\partial D_{N,\mathrm{ext}}.$ Therefore%
\begin{align*}
\left\vert G_{A}\left(  i,k+j\right)  -G_{A}\left(  i,k\right)  \right\vert
&  \leq\left\vert G_{N,\mathrm{ext}}\left(  i,k+j\right)  -G_{N,\mathrm{ext}%
}\left(  i,k\right)  \right\vert \\
&  +\sum_{s}\pi_{A}\left(  i,s\right)  \left\vert G_{N,\mathrm{ext}}\left(
s,k+j\right)  -G_{N,\mathrm{ext}}\left(  s,k\right)  \right\vert .
\end{align*}
Let%
\[
\mu\left(  j\right)  \overset{\mathrm{def}}{=}\sup_{i\in A}\sum_{k\in
D_{N,\mathrm{ext}}}\left\vert G_{N,\mathrm{ext}}\left(  i,k+j\right)
-G_{N,\mathrm{ext}}\left(  i,k\right)  \right\vert .
\]
Then we obtain%
\begin{align*}
\sum_{i\in D_{N,\mathrm{ext}}^{\circ}}\sum_{k\in D_{N,\mathrm{ext}}^{\circ}%
}\left\vert G_{A}\left(  i,k+j\right)  -G_{A}\left(  i,k\right)  \right\vert
&  \leq\mu\left(  j\right)  \left\vert A\right\vert +\mu\left(  j\right)
\sum_{i\in A}\sum_{s}\pi_{A}\left(  i,s\right) \\
&  =2\mu\left(  j\right)  \left\vert A\right\vert .
\end{align*}
We prove further down that%
\begin{equation}
\mu\left(  j\right)  \leq Cd\left(  j,0\right)  N \label{Green_derivative}%
\end{equation}
From that, we obtain%
\[
\operatorname*{var}\nolimits_{\mu_{A}}\left(  X\left(  \mathbf{\sigma}\right)
\right)  \leq CN^{1+\beta}\left\vert A\right\vert \leq CN^{1+d+\beta},
\]
and therefore%
\begin{align*}
&  \mu_{A}\left(  \sum\nolimits_{i\in D_{N,\mathrm{ext}}^{\circ}}\left\vert
N^{-d\beta}\sum\nolimits_{j\in C_{i}}\phi_{j}-\phi_{i}\right\vert \geq
N^{1+d-\eta}\right) \\
&  \leq2^{3N^{d}}\exp\left[  -N^{2+2d-2\eta}N^{-1-d-\beta}\right]  \leq
\exp\left[  -\frac{1}{C}N^{1+d-2\eta-\beta}\right]
\end{align*}
provided $2\eta+\beta<1$, and $N$ is large enough. This proves (\ref{CG_zero}%
), but we still have to prove (\ref{Green_derivative}).

For a fixed $j\in K\left(  d,\beta\right)  $ we can find a nearest neighbor
path of length $d\left(  j,0\right)  $ connecting $0$ with $j$. In order to
prove (\ref{Green_derivative}), we therefore only have to prove that for any
$e$ with $\left\vert e\right\vert =1$, we have%
\[
\sum_{k}\left\vert G_{N}\left(  0,k\right)  -G_{N}\left(  0,k+e\right)
\right\vert =O\left(  N\right)  .
\]
This was shown in Lemma \ref{lem:2.5-b}.

Next, we discuss how to transfer the result to the one we are interested in,
namely the corresponding approximation result on $D_{N}$ with boundary
conditions $aN$ and $bN$, respectively. For $a,b>0$ consider the event%
\begin{align}
\Lambda_{N,a,b}\overset{\mathrm{def}}{=}\Big \{\phi &  :\phi_{i}\in\left[
aN,aN+N^{-2d}\right]  ,\ i\in\partial_{L}D_{N},\label{Def_Lambdaab}\\
\phi_{i}  &  \in\left[  bN,bN+N^{-2d}\right]  ,\ i\in\partial_{R}%
D_{N}\Big \}.\nonumber
\end{align}
Applying Lemma \ref{Le_Approx_Energy} with $F=D_{N,\mathrm{ext}}^{\circ
},\ s=aN,\ t=bN$, we get%
\begin{equation}
\mu_{N,\mathrm{ext}}\left(  \Lambda_{N,a,b}\right)  =\exp\left[  -N^{d}%
\frac{a^{2}+\left(  b-a\right)  ^{2}+b^{2}}{2}+O\left(  N^{d-1}\right)
\right]  \mu_{N,\mathrm{ext}}\left(  \Lambda_{N,0,0}\right)  .
\label{MuNext_Lambdaab}%
\end{equation}
Furthermore%
\begin{equation}
\mu_{N,\mathrm{ext}}\left(  \Lambda_{N,0,0}\right)  \geq\left(  CN^{-2d}%
\right)  ^{2N^{d-1}}. \label{LB_Lambda00}%
\end{equation}
To prove this, we enumerate the points in $\partial D_{N}$ as $k_{1}%
,\ldots,k_{2N^{d-1}}$, and prove%
\begin{equation}
\mu_{N,\mathrm{ext}}\left(  \phi_{k_{1}}\in\left[  0,N^{-2d}\right]  \right)
\geq CN^{-2d}, \label{Gauss1}%
\end{equation}%
\begin{equation}
\mu_{N,\mathrm{ext}}\left(  \left.  \phi_{k_{j+1}}\in\left[  0,N^{-2d}\right]
\right\vert \phi_{k_{i}}=x_{i},\ \forall i\leq j\right)  \geq CN^{-2d},
\label{Gauss2}%
\end{equation}
uniformly in $x_{i}\in\left[  0,N^{-2d}\right]  $, and $j\leq2N^{d-1}.$
(\ref{Gauss1}) follows from the fact that $\phi_{k_{1}}$ is centered under
$\mu_{N,\mathrm{ext}}$ and $\operatorname*{var}\left(  \phi_{k_{1}}\right)  $
is bounded and bounded away from $0,$ uniformly in $N$, as we assume $d\geq3$.
Under the conditional distribution $\mu_{N,\mathrm{ext}}\left(  \left.
\cdot\right\vert \phi_{k_{i}}=x_{i},\ \forall i\leq j\right)  ,$
$\phi_{k_{j+1}}$ is not centered, but has an expectation in $\left[
0,N^{-2d}\right]  .$ Furthermore, the conditional variance is bounded and
bounded away from $0,$ uniformly in $N$, the choice of the enumeration, and
$j$. So (\ref{Gauss2}) follows, too. This implies (\ref{LB_Lambda00}).

From that, we get%
\begin{equation}
\mu_{N,\mathrm{ext}}^{\varepsilon}\left(  \Lambda_{N,a,b}\right)  \geq
\frac{Z_{N}}{Z_{N,\mathrm{ext}}}\mu_{N,\mathrm{ext}}\left(  \Lambda
_{N,a,b}\right)  \geq\exp\left[  -CN^{d}\right]  . \label{LB_Lambdaab}%
\end{equation}

Some more notations: If $\mathbf{x}=\left(  x_{i}\right)  _{i\in\partial
_{L}D_{N}},\ \mathbf{y}=\left(  y_{i}\right)  _{i\in\partial_{R}D_{N}}$, we
write $\mu_{N}^{\mathbf{x},\mathbf{y},\varepsilon}$ for the field on $D_{N}$
with boundary conditions $\mathbf{x}$ and $\mathbf{y}$ on $\partial D_{N},$
and $\varepsilon$-pinning. If we have an event $Q$ which depends on the field
variables only inside $D_{N}^{\circ},$ then%
\[
\mu_{N,\mathrm{ext}}^{\varepsilon}\left(  \left.  Q\right\vert \mathbf{\phi
}_{L}=\mathbf{x},\mathbf{\phi}_{R}=\mathbf{y}\right)  =\mu_{N}^{\mathbf{x}%
,\mathbf{y},\varepsilon}\left(  Q\right)  ,
\]
where $\mathbf{\phi}_{L}=\left\{  \phi_{i}\right\}  _{i\in\partial_{L}D_{N}},$
and $\mathbf{\phi}_{R}$ similarly. This follows from the Markov property and
the fact that the pinning is only inside $D_{N}^{\circ}.$

If $\mathbf{\phi}$ is an element in $\mathbb{R}^{D_{N}^{\circ}}$, we write
$\mathbf{\phi\vee}\left\{  \mathbf{x},\mathbf{y}\right\}  $ for the
configuration which is extended by $\mathbf{x}$ on $\partial_{L}D_{N},$ and
$\mathbf{y}$ on $\partial_{R}D_{N}.$ We set%
\[
U_{N,a,b}\overset{\mathrm{def}}{=}\left\{  \left(  \mathbf{x},\mathbf{y}%
\right)  :x_{i}\in\left[  aN,aN+N^{-2d}\right]  ,\ y_{i}\in\left[
bN,bN+N^{-2d}\right]  ,\right\}
\]
If $\mathbf{\phi}$ is a configuration which satisfies $\left\vert \phi
_{i}\right\vert \leq N^{d}$ for all $i\in D_{N}^{\circ},$ and $\left(
\mathbf{x},\mathbf{y}\right)  \in U_{N,a,b}$, then%
\[
H_{N}\left(  \mathbf{\phi\vee}\left\{  \mathbf{x},\mathbf{y}\right\}  \right)
=H_{N}\left(  \mathbf{\phi\vee}\left\{  aN,bN\right\}  \right)  +O\left(
N^{d-1}N^{-d}\right)  .
\]
Therefore, it follows that for any $Q\subset\left\{  \mathbf{\phi:}\left\vert
\phi_{i}\right\vert \leq N^{d},\ \forall i\in D_{N}^{\circ}\right\}  ,$ one
has%
\[
\mu_{N}^{\varepsilon,\mathbf{x},\mathbf{y}}\left(  Q\right)  =\mu
_{N}^{\varepsilon,aN,bN}\left(  Q\right)  \left(  1+O\left(  N^{-1}\right)
\right)  .
\]
We therefore have%
\begin{align}
&  \mu_{N}^{aN,bN,\varepsilon}\left(  Q\right)  \mu_{N,\mathrm{ext}%
}^{\varepsilon}\left(  \Lambda_{N,a,b}\right)  \nonumber\\
&  =\int_{U_{N,a,b}}\mu_{N}^{\mathbf{x},\mathbf{y},\varepsilon}\left(
Q\right)  \mu_{N,\mathrm{ext}}^{\varepsilon}\left(  \mathbf{\phi}_{L}\in
d\mathbf{x},\mathbf{\phi}_{R}\in d\mathbf{y}\right)  \left(  1+O\left(
N^{-1}\right)  \right)  \label{Extension_trick}\\
&  =\int_{U_{N,a,b}}\mu_{N,\mathrm{ext}}^{\varepsilon}\left(  \left.
Q\right\vert \mathbf{\phi}_{L}=\mathbf{x},\mathbf{\phi}_{R}=\mathbf{y}\right)
\mu_{N,\mathrm{ext}}^{\varepsilon}\left(  \mathbf{\phi}_{L}\in d\mathbf{x}%
,\mathbf{\phi}_{R}\in d\mathbf{y}\right)  (1+O(N^{-1}))  \nonumber\\
&  =\mu_{N,\mathrm{ext}}^{\varepsilon}\left(  Q\cap\Lambda_{N,a,b}\right)
(1+O(N^{-1})) 
\leq\mu_{N,\mathrm{ext}}^{\varepsilon}\left(  Q\right)
(1+O(N^{-1})) ,\nonumber
\end{align}
i.e., with (\ref{LB_Lambdaab})%
\begin{equation}
\mu_{N}^{aN,bN,\varepsilon}\left(  Q\right)  \leq\mu_{N,\mathrm{ext}%
}^{\varepsilon}\left(  Q\right)  \exp\left[  CN^{d}\right]  .\label{Bound_F}%
\end{equation}
We apply this to%
\[
Q\overset{\mathrm{def}}{=}\left\{  \left\Vert h^{\mathrm{cg},\beta,N}%
-h^{N}\right\Vert _{L^1(D)}\geq N^{-\eta}\right\}  \cap\left\{  \left\vert \phi
_{i}\right\vert \leq N^{d},\ \forall i\in D_{N}\right\}  .
\]
Evidently, the restriction to $\left\vert \phi_{i}\right\vert \leq N^{d}$ is
harmless, as%
\begin{equation}
\mu_{N}^{aN,bN,\varepsilon}\left(  \left\vert \phi_{i}\right\vert
>N^{d},\ \mathrm{some\ }i\right)  \leq CN^{d}\exp\left[  -\frac{1}{C}%
N^{2d}\right]  ,\label{Bound_verylarge}%
\end{equation}
and therefore, from \eqref{CG_zero} and \eqref{Bound_F},
\begin{align*}
\mu_{N}^{aN,bN,\varepsilon}\left(  \left\Vert h^{\mathrm{cg},\beta,N}%
-h^{N}\right\Vert _{L^1(D)}\geq N^{-\eta}\right)   &  \leq C\exp\left[  -\frac
{1}{C}N^{d+1-2\eta-\beta}+CN^{d}\right]  +CN^{d}\exp\left[  -\frac{1}{C}%
N^{2d}\right]  \\
&  \leq C\exp\left[  -\frac{1}{C}N^{d+1-2\eta-\beta}\right]  ,
\end{align*}
for large enough $N$, provided $0<2\eta+\beta<1$. This proves Proposition
\ref{Prop_Coarse_graining}.
\end{proof}

One simple consequence of this proposition is the following lemma; recall
\eqref{eq:Ma-H-2} for $h_{\mathrm{{PL}}}^{N}$.

\begin{lemma}
\label{lem:Ma-H-1} For every $\eta>0$, we have that
\[
\mu_{N}^{aN,bN,\varepsilon} \big( \|h^{N} -h_{\mathrm{{PL}}}^{N}\|_{L^{1}(D)}
\ge N^{-\eta} \big)
\le\exp\{-C N^{d+1-2\eta}\}.
\]

\end{lemma}

\begin{proof}
First, noting that $\sum_{v\in\{0,1\}^{d}} \left[  \prod_{\alpha=1}^{d}
\big( v_{\alpha}\{Nt_{\alpha}\} + (1-v_{\alpha})(1-\{Nt_{\alpha}\})
\big) \right]  =1$, we see that
\begin{align*}
\|h^{N} -h_{\mathrm{{PL}}}^{N}\|_{L^{1}(D)}  &  \le\frac1{N^{d+1}} \sum_{i\in
D_{N}} \sum_{v\in\{0,1\}^{d}} |\phi(i)-\phi(i+v)|\\
&  \le\frac{C_{d}}{N^{d+1}} \sum_{i,j\in D_{N}: |i-j|=1} |\phi(i)-\phi(j)|.
\end{align*}
Therefore, from \eqref{Bound_F} in the proof of Proposition
\ref{Prop_Coarse_graining} and the expansion \eqref{Representation}, it
suffices to prove
\[
\mu_{A,\mathrm{ext}} \left(  \sum_{i,j\in D_{N}: |i-j|=1} |\phi(i)-\phi(j)|
\ge N^{d+1-\eta} \right)  \le\exp\{-C N^{d+1-2\eta}\},
\]
uniformly in $A \subset D_{N,\mathrm{ext}}^{\circ}$. As we discussed in the
proof of Proposition \ref{Prop_Coarse_graining}, setting
\[
X(\sigma) = \sum_{i,j\in D_{N}: |i-j|=1} \sigma_{ij}(\phi(i)-\phi(j))
\]
for $\sigma= (\sigma_{ij}) \in\{-1,1\}^{{\mathbb{B}}_{N}}$, ${\mathbb{B}}_{N}
= \{(i,j); i,j\in D_{N}, |i-j|=1\}$, it suffices to show that
\begin{equation}
\label{eq:Ma-H-3}\mu_{A,\mathrm{ext}} \left(  X(\sigma) \ge N^{d+1-\eta}
\right)  \le\exp\{-C N^{d+1-2\eta}\},
\end{equation}
uniformly in $A$ and $\sigma$. However, $X(\sigma)$ are centered Gaussian
variables and
\begin{align*}
\mathrm{var}_{\mu_{A,\mathrm{ext}}} (X(\sigma))  &  = \sum_{%
\genfrac{}{}{0pt}{}{\scriptstyle i,j\in D_{N}: |i-j|=1
}{\scriptstyle i^{\prime},j^{\prime}\in D_{N}: |i^{\prime}-j^{\prime}|=1}%
} \sigma_{ij}\sigma_{i^{\prime}j^{\prime}} \big( G_{A}(i,i^{\prime}%
)-G_{A}(i,j^{\prime})-G_{A}(j,i^{\prime})+G_{A}(j,j^{\prime})\big)\\
&  \le C_{1}\sum_{i,j\in D_{N}, |e|=1} |G_{A}(i,j)-G_{A}(i,j+e)|\\
&  \le C_{2}\sum_{i,j\in D_{N}, |e|=1} |G_{N,\mathrm{ext}}%
(i,j)-G_{N,\mathrm{ext}}(i,j+e)|\\
&  \le C_{3} N^{d+1},
\end{align*}
by the estimate shown in the proof of Proposition \ref{Prop_Coarse_graining}.
This combined with the Gaussian property of $X(\sigma)$ immediately implies \eqref{eq:Ma-H-3}.
\end{proof}

We draw some other easy consequences from the coarse graining estimate: Given
$\gamma>0$ we define the \textbf{mesoscopic wetted region} by%
\[
\mathcal{M}_{N} \equiv \mathcal{M}_{N}(\phi)
\overset{\mathrm{def}}{=}\bigcup\left\{  C\in\mathcal{B}%
_{N}:\phi_{C}^{\mathrm{cg},\beta,N}\geq N^{\gamma}\right\}  .
\]

We write%
\begin{align*}
\mu_{N}^{aN,bN,\varepsilon}\left(  A_{N,\alpha}\cap\Omega_{N}^{+}\right)   &
=\sum_{B\in\mathcal{\hat{B}}}\mu_{N}^{aN,bN,\varepsilon}\left(  A_{N,\alpha
}\cap\Omega_{N}^{+}\cap\left\{  \mathcal{M}_{N}=B\right\}  \right)  \\
&  \leq\left\vert \mathcal{\hat{B}}\right\vert \max_{B\in\mathcal{\hat{B}}}%
\mu_{N}^{aN,bN,\varepsilon}\left(  A_{N,\alpha}\cap\Omega_{N}^{+}\cap\left\{
\mathcal{M}_{N}=B\right\}  \right)  \\
&  =\exp\left[  N^{d\left(  1-\beta\right)  }\log2\right]  \max_{B\in
\mathcal{\hat{B}}}\mu_{N}^{aN,bN,\varepsilon}\left(  A_{N,\alpha}\cap
\Omega_{N}^{+}\cap\left\{  \mathcal{M}_{N}=B\right\}  \right)  .
\end{align*}
In order to prove (\ref{LDP_main}), it therefore suffices to prove that there
exists $\delta_{1}<d\beta$ and $\alpha>0$ such that%
\begin{equation}
\max_{B\in\mathcal{\hat{B}}}\mu_{N}^{aN,bN,\varepsilon}\left(  A_{N,\alpha
}\cap\Omega_{N}^{+}\cap\left\{  \mathcal{M}_{N}=B\right\}  \right)
\leq\mathrm{e}^{-N^{d-\delta_{1}}},\label{LDP_main2}%
\end{equation}
$N$ large, uniformly in $B.$

Let $\partial^{\ast}B\overset{\mathrm{def}}{=}\partial B\cap D_{N}^{\circ}$.
Any point $i\in\partial^{\ast}\mathcal{M}_N$ is in block $C$ with $\phi_{C}\leq
N^{\gamma}.$ If also $\phi\in\Omega_{N}^{+}$, we conclude that%
\[
\phi(i)\leq N^{d\beta+\gamma}\log N.
\]
We will choose $\gamma,\beta$ such that $d\beta+\gamma<1$, and then choose%
\begin{equation}
\kappa_{1}\overset{\mathrm{def}}{=}\frac{1-d\beta-\gamma}{2}, \label{kappa}%
\end{equation}
so that if $i\in\partial^{\ast}\mathcal{M}_N$ we have%
\begin{equation}
\phi(i)\leq N^{1-\kappa_{1}}. \label{UP_phiondstarb}%
\end{equation}

\begin{lemma}
[Volume filling lemma]\label{Le_Volumefilling}Assume $\gamma+\eta>1$, and
$2\eta+\beta<1$. Then%
\[
\mu_{N}^{aN,bN,\varepsilon}\left(  \left\vert \mathcal{M}_{N}\cap\left\{
i:\phi\left(  i\right)  =0\right\}  \right\vert \geq N^{d+1-\gamma-\eta
}\right)  \leq C\exp\left[  -\frac{1}{C}N^{d+1-2\eta-\beta}\right]  .
\]

\end{lemma}

\begin{proof}
Remark that%
\begin{align*}
&  \sum_{i}\left\vert \phi\left(  i\right)  -\phi^{\mathrm{cg},\beta,N}\left(
i\right)  \right\vert \\
&  \geq\sum_{i\in\mathcal{M}_N\cap\left\{  i:\phi\left(  i\right)  =0\right\}
}\left\vert \phi\left(  i\right)  -\phi^{\mathrm{cg},\beta,N}\left(  i\right)
\right\vert \geq\left\vert \mathcal{M}_{N}\cap\left\{
i:\phi\left(  i\right)  =0\right\}  \right\vert N^{\gamma}.
\end{align*}
Therefore, from Proposition \ref{Prop_Coarse_graining} we get%
\begin{align*}
&  \mu_{N}^{aN,bN,\varepsilon}\left(  \left\vert \mathcal{M}_{N}
\cap\left\{  i:\phi\left(  i\right)  =0\right\}  \right\vert \geq
N^{d+1-\gamma-\eta}\right)  \\
&  \leq\mu_{N}^{aN,bN,\varepsilon}\left(  N^{-d-1}\sum_{i}\left\vert
\phi\left(  i\right)  -\phi^{\mathrm{cg},\beta,N}\left(  i\right)  \right\vert
\geq N^{-\eta}\right)  \\
&  \leq C\exp\left[  -\frac{1}{C}N^{d+1-2\eta-\beta}\right]
\end{align*}
which proves the claim.
\end{proof}

The different requirements on $\beta,\eta,\gamma>0$ are%
\begin{align*}
2\eta+\beta &  <1,\\
d\beta+\gamma &  <1,\\
\eta+\gamma &  >1.
\end{align*}
We can fulfill them by taking for instance%
\[
\beta=\frac{1}{10d},\ \gamma=\frac{4}{5},\ \eta=\frac{1}{4}.
\]
From now on, we keep these constants fixed under the above restrictions, for
instance with the above values. We put%
\[
\kappa_{2}\overset{\mathrm{def}}{=}\gamma+\eta-1,\ \kappa_{3}\overset
{\mathrm{def}}{=}\frac{1-\left(  2\eta+\beta\right)  }{2},
\]
so that, by the volume filling lemma, we have%
\begin{equation}
\mu_{N}^{\varepsilon}\left(  \left\vert \mathcal{M}_{N}\cap\left\{
i:\phi\left(  i\right)  =0\right\}  \right\vert \geq N^{d-\kappa_{2}}\right)
\leq\exp\left[  -N^{d+\kappa_{3}}\right]  . \label{volume_filling}%
\end{equation}

\subsection{Proof of \eqref{LDP_main2}}

\label{section:6.3}

If $A\subset D_{N}^{\circ}$, we write $A_{\mathrm{ext}}\overset{\mathrm{def}%
}{=}A\cup\left(  D_{N,\mathrm{ext}}\backslash D_{N}^{\circ}\right)  $. Using
Lemma \ref{lem:3} (patching at $\partial D_N$), we have%
\[
Z_{A_{\mathrm{ext}}}=Z_{A}Z_{D_{N,\mathrm{ext}}\backslash D_{N}^{\circ}}%
\exp\left[  O\left(  N^{d-1}\right)  \right]  ,
\]
and using Lemma \ref{lem:2}, one has%
\[
Z_{D_{N,\mathrm{ext}}\backslash D_{N}^{\circ}}=\exp\left[  2\hat{q}^{0}%
N^{d}+O\left(  N^{d-1}\right)  \right].
\]
Note that these partition functions are defined without pinning.
Therefore
\begin{align*}
Z_{N,\mathrm{ext}}^{\varepsilon} &  :=\sum_{A\subset D_{N}^{\circ}}%
\varepsilon^{\left\vert D_{N}^{\circ}\backslash A\right\vert }%
Z_{A_{\mathrm{ext}}}\\
&  =\exp\left[  2N^{d}\hat{q}^{0}+O\left(  N^{d-1}\right)  \right]
\sum_{A\subset D_{N}^{\circ}}\varepsilon^{\left\vert D_{N}^{\circ}\backslash
A\right\vert }Z_{A}\\
&  =\exp\left[  2N^{d}\hat{q}^{0}+N^{d}\hat{q}^{\varepsilon}+O\left(
N^{d-1}\right)  \right]  ,
\end{align*}
where we have used a version of (2.3.4) of \cite{BI}. Therefore,%
\[
\mu_{N,\mathrm{ext}}^{\varepsilon}=\exp\left[  -N^{d}\hat{q}^{\varepsilon
}-2N^{d}\hat{q}^{0}+O\left(  N^{d-1}\right)  \right]  \sum_{A\subset
D_{N}^{\circ}}\varepsilon^{\left\vert D_{N}^{\circ}\backslash A\right\vert
}Z_{A_{\mathrm{ext}}}\mu_{A,\mathrm{ext}}.
\]
However, we can estimate
\begin{gather*}
\sum_{A\subset D_{N}^{\circ}}\varepsilon^{\left\vert D_{N}^{\circ}\backslash
A\right\vert }Z_{A,\mathrm{ext}}\mu_{A,\mathrm{ext}}\left(  \Lambda
_{N,a,b}\right)  \geq Z_{D_{N,\mathrm{ext}}^{\circ}}\mu_{N,\mathrm{ext}%
}\left(  \Lambda_{N,a,b}\right)  \\
=Z_{D_{N,\mathrm{ext}}^{\circ}}\exp\left[  -\frac{N^{d}}{2}\left(  a^{2}%
+b^{2}+\left(  b-a\right)  ^{2}\right)  +O\left(  N^{d-1} \log N\right)  \right]
\end{gather*}
by (\ref{MuNext_Lambdaab}) and (\ref{LB_Lambda00}). Using%
\[
Z_{D_{N,\mathrm{ext}}^{\circ}}=\exp\left[  3N^{d}\hat{q}^{0}+O\left(
N^{d-1}\right)  \right],
\]
and recalling $\xi^\varepsilon=\hat q^\varepsilon -\hat q^0$ as in
Remark \ref{rem:3.3}, we obtain
\begin{equation}
\mu_{N,\mathrm{ext}}^{\varepsilon}\left(  \Lambda_{N,a,b}\right)  \geq
\exp\left[  -N^{d}\left\{  \frac{a^{2}+b^{2}+\left(  b-a\right)  ^{2}}2 
+\xi^{\varepsilon}\right\} + O\left(  N^{d-1}\log N \right)\right].
\label{Eq2}%
\end{equation}

We use now $\mu_{N}^{\varepsilon,\mathbf{x},\mathbf{y}}$ as defined after
(\ref{LB_Lambdaab}). Arguing in the same way as in (\ref{Extension_trick}), we
obtain with the abbreviation
$
\mathcal{B}_{N,\alpha}\overset{\mathrm{def}}{=}\left\{  \mathcal{M}_N=B\right\}
\cap\Omega_{N}^{+}\cap A_{N,\alpha}%
$,
\begin{align*}
&  \mu_{N}^{\varepsilon,aN,bN}\left(  \mathcal{B}_{N,\alpha}\right)
\mu_{N,\mathrm{ext}}^{\varepsilon}\left(  \Lambda_{N,a,b}\right)  \\
&  =\mu_{N,\mathrm{ext}}^{\varepsilon}\Big (\mathcal{B}_{N,\alpha}\cap\left\{
\phi%
\vert
_{\partial D_{N}}\in U_{N,a,b}\right\}  \Big )\left(  1+O\left(
N^{-1}\right)  \right)  .
\end{align*}
Combining this with (\ref{Eq2}) gives%
\begin{align}
\mu_{N}^{\varepsilon,aN,bN}\left(  \mathcal{B}_{N,\alpha}\right)   &  \leq
\mu_{N,\mathrm{ext}}^{\varepsilon}\left(  \mathcal{B}_{N,\alpha}\cap\left\{
\phi%
\vert
_{\partial D_{N}}\in U_{N,a,b}\right\}  \right)  \label{Equ3}\\
&  \times\exp\left[  N^{d}\left\{  \frac{a^{2}+b^{2}+\left(  b-a\right)  ^{2}%
}{2}+\xi^{\varepsilon}\right\} +O\left(  N^{d-1}\log N\right) \right].
\nonumber
\end{align}
For the expression on the right hand side, we use the usual splitting%
\[
\mu_{N,\mathrm{ext}}^{\varepsilon}\left(  \cdot\right)  =\sum_{A\subset
D_{N,\mathrm{ext}}^{\circ},\ A^{c}\subset D_{N}^{\circ}}\frac{\varepsilon
^{\left\vert D_{N}^{\circ}\backslash A\right\vert }Z_{A_{\mathrm{ext}}}%
}{Z_{N,\mathrm{ext}}^{\varepsilon}}\mu_{A,\mathrm{ext}}\left(  \cdot\right).
\]
From \eqref{volume_filling}, we know that we can restrict the summation
to $A$ with $|B\setminus A| \le N^{d-\kappa_2}$, 
up to a contribution of order $\exp\left[
-N^{d+\kappa_{3}}\right]  $, which we can neglect. Splitting $A$ into
$A_{1}\cup A_{2}$ with $A_{2}\overset{\mathrm{def}}{=}A\cap B$, and using
(2.3.4) of \cite{BI} and Lemma \ref{lem:2},%
\begin{align*}
Z_{A_{1}\cup A_{2},\mathrm{ext}} &  \leq Z_{A_{2},\mathrm{ext}}Z_{A_{1}}%
\exp\left[  O\left(  N^{d-\beta}\right)  \right]  \\
&  \leq Z_{B,\mathrm{ext}}Z_{A_{1}}\exp\left[  O\left(  N^{d-\beta}\right)
\right]  \\
&  \leq Z_{A_{1}}\exp\left[ (2N^d+ \left\vert B\right\vert) \hat{q}^{0}+O\left(
N^{d-\beta}\right)  \right]  ,
\end{align*}
it suffices to estimate%
\[
J_{N}\left(  B,A_{2}\right)  =\sum_{A_{1}:A_{1}\cap B=\emptyset}%
\varepsilon^{\left\vert B^{c}\cap A_{1}^{c}\right\vert }Z_{A_{1}}\mu
_{A_{1}\cup A_{2},\mathrm{ext}}\left(  \mathcal{B}_{N,\alpha}\cap\left\{  \phi%
\vert
_{\partial D_{N}}\in U_{N,a,b}\right\}  \right)
\]
uniformly in $B,A_{2}$. If we prove that for all $\delta>0$ sufficiently small, 
there exists
$\alpha<1$ such that for all mesoscopic $B$ and all $A_{2}\subset B$ with
$\left\vert B\backslash A_{2}\right\vert \leq N^{d-\kappa_{2}}$ we have%
\begin{equation}
J_{N}\left(  B,A_{2}\right)  \exp\left[  N^{d}\frac{a^{2}+b^{2}+\left(
b-a\right)  ^{2}}{2}-\left\vert B^{c}\right\vert \hat q^{0}\right]  \leq
\exp\left[  -N^{d-\delta}\right]  \label{Est_nec_J}%
\end{equation}
for large enough $N$ (uniformly in $B,A_{2}$), we have proved \eqref{LDP_main2}.

Note that 
\[
\mathcal{B}_{N,\alpha}\subset\left\{  -\log N\leq\phi%
\vert
_{\partial^{\ast}B}\leq N^{d-\kappa_{1}}\right\}  \cap\left\{  \mathcal{M}_N%
=B\right\}  \cap A_{N,\alpha}.
\]
On $\partial^{\ast}B\cap\left(  A_{1}\cup A_{2}\right)  ^{c}$, $\phi$ is of
course $0$ under $\mu_{A_{1}\cup A_{2},\mathrm{ext}}$. We define $\hat{\mu
}_{B,A_{1},A_{2},\mathbf{x}}$ to be the free field on $\mathbb{R}^{A_{2}%
\cup\left(  D_{N,\mathrm{ext}}\backslash D_{N}^{\circ}\right)  }$ with
boundary condition $0$ on $\partial D_{N,\mathrm{ext}}\cap\left(  A_{1}\cup
A_{2}\right)  ^{c}$ and boundary condition $\mathbf{x}$ on $\partial^{\ast
}B\cap\left(  A_{1}\cup A_{2}\right)  $. Then%
\begin{align*}
&  \mu_{A_{1}\cup A_{2},\mathrm{ext}}\left(  \mathcal{B}_{N,\alpha}%
\cap\left\{  \phi%
\vert
_{\partial D_{N}}\in U_{N,a,b}\right\}  \cap A_{N,\alpha}\right)  \\
&  \leq\mu_{A_{1}\cup A_{2},\mathrm{ext}}\left(  \left\{  -\log N\leq\phi%
\vert
_{\partial^{\ast}B}\leq N^{d-\kappa_{1}}\right\}  ,\mathcal{M}_N=B,\phi%
\vert
_{\partial D_{N}}\in U_{N,a,b},A_{N,\alpha}\right)  \\
&  \leq\int\limits_{-\log N\leq\mathbf{x}\leq N^{1-\kappa_{1}}}\hat{\mu
}_{B,A_{1},A_{2},\mathbf{x}}\left(  \phi%
\vert
_{\partial D_{N}}\in U_{N,a,b},\mathcal{M}_N=B,A_{N,\alpha}\right)  \mu
_{A_{1}\cup A_{2},\mathrm{ext}}\left(  \phi%
\vert
_{\partial^{\ast}B}\in d\mathbf{x}\right)  \\
&  \leq\mu_{A_{1}\cup A_{2},\mathrm{ext}}\left(  -\log N\leq\phi%
\vert
_{\partial^{\ast}B}\leq N^{1-\kappa_{1}}\right)  \\
&  \times\sup_{\mathbf{x}\leq N^{1-\kappa_{1}}}\hat{\mu}_{B,A_{1}%
,A_{2},\mathbf{x}}\left(  \phi%
\vert
_{\partial D_{N}}\in U_{N,a,b},\mathcal{M}_N=B,A_{N,\alpha}\right)  \\
&  \leq\sup_{\mathbf{x}\leq N^{1-\kappa_{1}}}\hat{\mu}_{B,A_{1},A_{2}%
,\mathbf{x}}\left(  \phi%
\vert
_{\partial D_{N}}\in U_{N,a,b},\mathcal{M}_N=B,A_{N,\alpha}\right)  .
\end{align*}

There is a slightly awkward dependence of the right hand side on $A_{1}$: If a point
$i\in\partial^{\ast}B$ is in $\partial^{\ast}A_{2}$ but not in $A_{1}$, then
the boundary condition there is $0.$ However, if it is in $A_{1}$, then the
boundary condition can be arbitrary $\leq N^{1-\kappa_{1}}$. If we allow for
arbitrary boundary condition $\mathbf{x}$ on $\partial^{\ast}A_{2}$, of course
with $\mathbf{x}\leq N^{1-\kappa_{1}}$ and denote the corresponding measure on
$\mathbb{R}^{A_{2}}$ by $\bar{\mu}_{A_{2},\mathbf{x}}$, then%
\begin{align*}
&  \sup_{\mathbf{x}\leq N^{1-\kappa_{1}}}\hat{\mu}_{B,A_{1},A_{2},\mathbf{x}%
}\left(  \phi%
\vert
_{\partial D_{N}}\in U_{N,a,b},\mathcal{M}_N=B,A_{N,\alpha}\right)  \\
&  \leq\sup_{\mathbf{x}\leq N^{1-\kappa_{1}}}\bar{\mu}_{A_{2},\mathbf{x}%
}\left(  \phi%
\vert
_{\partial D_{N}}\in U_{N,a,b},\mathcal{M}_N=B,\ A_{N,\alpha}\right)  ,
\end{align*}
and the right hand side has no longer a dependence on $A_{1}$. Therefore, we
just get%
\begin{align*}
J_{N}\left(  B,A_{2}\right)   &  =\sum_{A_{1}:A_{1}\cap B=\emptyset
}\varepsilon^{\left\vert B^{c}\cap A_{1}^{c}\right\vert }Z_{A_{1}}\mu
_{A_{1}\cup A_{2},\mathrm{ext}}\left(  \mathcal{B}_{N,\alpha},\ \phi%
\vert
_{\partial D_{N}}\in U_{N,a,b}\right)  \\
&  \leq\left(  \sum_{A_{1}:A_{1}\cap B=\emptyset}\varepsilon^{\left\vert
B^{c}\cap A_{1}^{c}\right\vert }Z_{A_{1}}\right)  \sup_{\mathbf{x}\leq
N^{1-\kappa_{1}}}\bar{\mu}_{A_{2},\mathbf{x}}\left(  \phi%
\vert
_{\partial D_{N}}\in U_{N,a,b},\mathcal{M}_N=B,\ A_{N,\alpha}\right)  \\
&  =\exp\left[  \left\vert B^{c}\right\vert \hat q^{\varepsilon}+O\left(
N^{d-\beta}\right)  \right]  \sup_{\mathbf{x}\leq N^{1-\kappa_{1}}}\bar{\mu
}_{A_{2},\mathbf{x}}\left(  \phi%
\vert
_{\partial D_{N}}\in U_{N,a,b},\mathcal{M}_N=B,\ A_{N,\alpha}\right)  .
\end{align*}
Therefore, we are left with estimating the above supremum. We distinguish two cases:

\noindent\textbf{First case:}%
\begin{equation}
E_{N,0}\left(  A_{2}\right)  -\xi^{\varepsilon}\left\vert B^{c}\right\vert
\geq N^{d}\inf_{h}\Sigma\left(  h\right)  +N^{d-\chi}\label{first_case}%
\end{equation}
with $\chi>0$ to be chosen later. In this case, we drop $\mathcal{M}_N
=B,\ A_{N,\alpha}$ and obtain%
\begin{align*}
&  \sup_{\mathbf{x}\leq N^{1-\kappa_{1}}}\bar{\mu}_{A_{2},\mathbf{x}}\left(
\phi%
\vert
_{\partial D_{N}}\in U_{N,a,b},\mathcal{M}_N=B,\ A_{N,\alpha}\right)  \\
&  \leq\sup_{\mathbf{x}\leq N^{1-\kappa_{1}}}\bar{\mu}_{A_{2},\mathbf{x}%
}\left(  \phi%
\vert
_{\partial_{L}D_{N}}\geq aN,\phi%
\vert
_{\partial_{R}D_{N}}\geq bN\right)  .
\end{align*}
By the FKG inequality,
the last expression can be estimated from above by putting all boundary
conditions (including at $\partial D_{N,\mathrm{ext}}$) at $N^{1-\kappa_{1}}$.
By shifting the field and the boundary conditions down by $N^{1-\kappa_{1}}$,
we obtain from Lemma \ref{lemma:6.5} that the right hand side is%
\begin{align*}
&  \leq\exp\left[  -\Xi\left(  A_{2},aN-N^{1-\kappa_{1}},bN-N^{1-\kappa_{1}%
}\right)  -N^{d}\frac{a^{2}+b^{2}}{2}+O\left(  N^{d-\kappa_{4}}\right)
\right]  \\
&  =\exp\left[  -\Xi\left(  A_{2},aN,bN\right)  -N^{d}\frac{a^{2}+b^{2}}%
{2}+O\left(  N^{d-\kappa_{5}}\right)  \right]  \\
&  =\exp\left[  -E_{N,0}\left(  A_{2}\right)  -N^{d}\frac{a^{2}+b^{2}}{2}+O\left(
N^{d-\kappa_{5}}\right)  \right]  ,
\end{align*}
with some constant $\kappa_{4},\kappa_{5}>0$, which depend only on the fixed
values $\beta,\gamma,\eta$. Summarizing, we get%
\begin{align*}
&  \exp\left[  N^{d}\frac{a^{2}+b^{2}+\left(  b-a\right)  ^{2}}{2}-\left\vert
B^{c}\right\vert \hat{q}^{0}\right]  J_{N}\left(  B,A_{2}\right)  \\
&  \leq\exp\left[  N^{d}\frac{\left(  b-a\right)  ^{2}}{2}+\left\vert
B^{c}\right\vert \xi^{\varepsilon}-E_{N,0}\left(  A_{2}\right)  +O\left(
N^{d-\min\left(  \beta,\kappa_{5}\right)  }\right)  \right]  .
\end{align*}
Remember now, that we have%
\[
\frac{\left(  b-a\right)  ^{2}}{2}=\inf_{h}\Sigma\left(  h\right).
\]
Therefore, from \eqref{first_case}, 
if we choose $\chi>0$ small enough, but smaller than $\min\left(
\beta,\kappa_{5}\right)  ,$ we have proved the bound \eqref{Est_nec_J}
in this case.  (Here actually, $\alpha$
plays no role). This $\chi$ will be fixed from now on.

\noindent\textbf{Second case: }%
\begin{equation}
E_{N,0}\left(  A_{2}\right)  -\xi^{\varepsilon}\left\vert B^{c}\right\vert
\leq N^{d}\inf_{h}\Sigma\left(  h\right)  +N^{d-\chi}. \label{Second}%
\end{equation}

Given $\mathbf{x}\in\mathbb{R}_{\partial^{\ast}A_{2}}$, $-\log N\leq
\mathbf{x}\leq N^{1-\kappa_{3}}$, $\mathbf{y}_{L}\in\mathbb{R}^{\partial
_{L}D_{N}}$, and $\mathbf{y}_{R}\in\mathbb{R}^{\partial_{R}D_{N}}$ with
$aN\leq\mathbf{y}_{L}\leq aN+N^{-2d},\ bN\leq\mathbf{y}_{R}\leq bN+N^{-2d}$,
we write $\phi_{\mathbf{x},\mathbf{y}_{L},\mathbf{y}_{R}}$ for the harmonic
function with these boundary conditions. If the boundary conditions are $0$
and $aN,bN$ respectively, we write $\phi_{A_{2}}$ (or $\bar\phi^{A_2}$
in Section \ref{section:3.2}).  From the maximum principle, we
know that%
\[
\sup_{i\in A_{2}}\left\vert \phi_{\mathbf{x},\mathbf{y}_{L},\mathbf{y}_{R}%
}\left(  i\right)  -\phi_{A_{2}}\left(  i\right)  \right\vert \leq
N^{1-\kappa_{3}},
\]
and therefore%
\[
\sum_{i\in A_{2}}\left\vert \phi_{\mathbf{x},\mathbf{y}_{L},\mathbf{y}_{R}%
}\left(  i\right)  -\phi_{A_{2}}\left(  i\right)  \right\vert \leq
N^{d+1-\kappa_{3}}.
\]
By the stability (rigidity) results obtained in Proposition \ref{prop:stability-meso}, 
we have that either%
\[
\sum_{i}\left\vert \phi_{A_{2}}\left(  i\right)  -N\bar{h}\left(  \frac{i}%
{N}\right)  \right\vert \leq N^{d+1-\kappa_{6}}%
\]
or%
\[
\sum_{i}\left\vert \phi_{A_{2}}\left(  i\right)  -N\hat{h}\left(  \frac{i}%
{N}\right)  \right\vert \leq N^{d+1-\kappa_{6}},
\]
where $\kappa_{6}>0$ depends on $\chi.$ Therefore, putting $\kappa_{7}%
\overset{\mathrm{def}}{=}\min\left(  \kappa_{6},\kappa_{3}\right)  $, we have,
uniformly in $\mathbf{x},\mathbf{y}_{L},\mathbf{y}_{R}$ satisfying the above
conditions that either%
\[
\sup_{\mathbf{x},\mathbf{y}_{L},\mathbf{y}_{R}}\sum_{i}\left\vert
\phi_{\mathbf{x},\mathbf{y}_{L},\mathbf{y}_{R}}\left(  i\right)  -N\bar
{h}\left(  \frac{i}{N}\right)  \right\vert \leq N^{d+1-\kappa_{7}}%
\]
or%
\[
\sup_{\mathbf{x},\mathbf{y}_{L},\mathbf{y}_{R}}\sum_{i}\left\vert
\phi_{\mathbf{x},\mathbf{y}_{L},\mathbf{y}_{R}}\left(  i\right)  -N\hat
{h}\left(  \frac{i}{N}\right)  \right\vert \leq N^{d+1-\kappa_{7}}.
\]
Therefore, if we choose $\alpha>0$ smaller than $\kappa_{7}$ we have that%
\[
\operatorname*{dist}\nolimits_{L_{1}}\left(  h_{N},\left\{  \hat{h},\bar
{h}\right\}  \right)  \geq N^{-\alpha}%
\]
implies%
\[
\sum_{i}\left\vert \phi_{\mathbf{x},\mathbf{y}_{L},\mathbf{y}_{R}}\left(
i\right)  -\phi\left(  i\right)  \right\vert \geq\frac{1}{2}N^{1+d-\alpha}%
\]
for all $\mathbf{x},\mathbf{y}_{L},\mathbf{y}_{R}$ under the above
restrictions. Therefore,%
\begin{align*}
&  \bar{\mu}_{A_{2},\mathbf{x}}\left(  \phi%
\vert
_{\partial D_{N}}\in U_{N,a,b},\mathcal{M}_N=B,\ A_{N,\alpha}\right) \\
&  \leq\bar{\mu}_{A_{2},\mathbf{x}}\left(  \phi%
\vert
_{\partial D_{N}}\in U_{N,a,b},\sum_{i}\left\vert \phi_{\mathbf{x}%
,\mathbf{y}_{L},\mathbf{y}_{R}}\left(  i\right)  -\phi\left(  i\right)
\right\vert \geq\frac{1}{2}N^{1+d-\alpha}\right)  .
\end{align*}
Applying the Markov property at $\partial D_{N}$, we can bound that by%
\[
\bar{\mu}_{A_{2},\mathbf{x}}\left(  \phi%
\vert
_{\partial_{L}D_{N}}\geq aN,\ \phi%
\vert
_{\partial_{R}D_{N}}\geq bN\right)  \sup_{\mathbf{x,\mathbf{y}}_{L}%
,\mathbf{y}_{R}}\widetilde{\mu}_{A_{2},\mathbf{x,\mathbf{y}}_{L}%
,\mathbf{y}_{R}}\left(  \sum_{i}\left\vert \phi_{\mathbf{x},\mathbf{y}%
_{L},\mathbf{y}_{R}}\left(  i\right)  -\phi\left(  i\right)  \right\vert
\geq\frac{1}{2}N^{1+d-\alpha}\right)  ,
\]
where $\widetilde{\mu}_{A_{2},\mathbf{x,\mathbf{y}}_{L},\mathbf{y}_{R}}$ is
the free field on $\mathbb{R}^{A_{2}}$ with boundary conditions
$\mathbf{x,\mathbf{y}}_{L},\mathbf{y}_{R}$. Remark that $\phi_{\mathbf{x}%
,\mathbf{y}_{L},\mathbf{y}_{R}}\left(  i\right)  $ is the expectation of
$\phi\left(  i\right)  $ under $\widetilde{\mu}_{A_{2},\mathbf{x,\mathbf{y}%
}_{L},\mathbf{y}_{R}}.$   We write $\tilde E$ for the expectation under
$\tilde\mu := \widetilde{\mu}_{A_{2},\mathbf{x,\mathbf{y}
}_{L},\mathbf{y}_{R}}.$  Then, 
\begin{align*}
m:= & \tilde E \left[\sum_i\big|\tilde E[\phi(i)] - \phi(i)\big|\right] \\
\le&  \sum_i \sqrt{\operatorname{var}_{\tilde\mu}(\phi(i))}
= O(N^d),
\end{align*}
uniformly in $A_{2},\mathbf{x,\mathbf{y}}_{L},\mathbf{y}_{R}$.
Therefore, if $\alpha<1$, by (4.4) of \cite{Le}
\begin{align*}
\tilde\mu \left(\sum_i\big|\tilde E[\phi(i)] - \phi(i)\big|\ge \tfrac12 N^{1+d-\alpha}\right)
& \le \tilde\mu \left(\sum_i\big|\tilde E[\phi(i)] - \phi(i)\big|\ge m+\tfrac14 N^{1+d-\alpha}\right)\\
& \le \exp\left( - \frac{N^{2+2d-2\alpha}}{32\sigma^2}\right),
\end{align*}
where
$$
\sigma^2 = \sup\left\{ \operatorname{var}_{\tilde\mu} \left(\sum_ig(i)\phi(i) \right);
\sup_i |g(i)|\le 1\right\}.
$$
However, one can estimate 
$$
\sigma^2 \le \sum_{i,j\in A_2} G_{A_2}(i,j) \le C N^{d+2}.
$$
Therefore, if $0< 2\alpha<\delta$, we get
\[
\widetilde{\mu}_{A_{2},\mathbf{x,\mathbf{y}}_{L},\mathbf{y}_{R}}\left(
\sum_{i}\left\vert \phi_{\mathbf{x},\mathbf{y}_{L},\mathbf{y}_{R}}\left(
i\right)  -\phi\left(  i\right)  \right\vert \geq\frac{1}{2}N^{1+d-\alpha
}\right)  \leq\exp\left[  -N^{d-\delta}\right]  ,
\]
uniformly in $A_{2},\mathbf{x,\mathbf{y}}_{L},\mathbf{y}_{R}.$ Estimating
$\bar{\mu}_{A_{2},\mathbf{x}}\left(  \phi%
\vert
_{\partial_{L}D_{N}}\geq aN,\ \phi%
\vert
_{\partial_{R}D_{N}}\geq bN\right)  $ in the same way as in the first case, we
arrive at \eqref{Est_nec_J} also in this case.  .

\end{document}